\documentclass[12pt]{amsart}

\usepackage[T1]{fontenc}
\usepackage{amssymb}
\usepackage{bm}
\usepackage{ytableau}
\usepackage{amsmath}
 \usepackage{mathpple}
\usepackage{boondox-cal}
\usepackage{amsmath,amsfonts, amscd,amsthm, graphicx}
  \usepackage{amsrefs}
\usepackage[all]{xy}
\usepackage{color}
\usepackage{tikz-cd}

\usepackage{hyperref}

\def\XXint#1#2#3{{\setbox0=\hbox{$#1{#2#3}{\int}$}
\vcenter{\hbox{$#2#3$}}\kern-.5\wd0}}

\numberwithin{equation}{section}

\newcommand{\mbf}{\mathbf}
\newcommand{\mbb}{\mathbb}
\newcommand{\mf}{\mathfrak}
\newcommand{\mc}{\mathcal}

\renewcommand{\S}{\mathcal S}

\theoremstyle{plain}
\newtheorem{thm}{Theorem}[section]

\newtheorem{thm-defn}{Theorem/Definition}[section]
\newtheorem{lem}[thm]{Lemma}
\newtheorem{lem-defn}[thm]{Lemma/Definition}
\newtheorem{prop}[thm]{Proposition}

\newtheorem{prop-defn}[thm]{Proposition-Definition}

\newtheorem{defn}[thm]{Definition}

\newtheorem{thm-alg}[thm]{Theorem/Algorithm}

\allowdisplaybreaks[4]  

\begin{document}

 \title{Quantum difference equations and Maulik-Okounkov quantum affine algebras of affine type $A$ }
  \author{Tianqing Zhu}
  \dedicatory{To my grandpa Zhang Gongfu with highest admiration}
  \address{Yau Mathematical Sciences Center}
\email{ztq20@mails.tsinghua.edu.cn}

  \date{}

  \maketitle
\begin{abstract} In this paper we prove the isomorphism of the positive half of the quantum toroidal algebra and the positive half of the Maulik-Okounkov quantum affine algebra of affine type $A$ via the monodromy representation for the Dubrovin connection. The main tool is on the proof of the fact that the degeneration limit of the algebraic quantum difference equation is the same as that of the Okounkov-Smirnov geometric quantum difference equation.
\end{abstract}

\tableofcontents
\section{\textbf{Introduction}}

\subsection{Maulik-Okounkov quantum algebras}
Stable envelopes were first introduced by Maulik and Okounkov in \cite{MO12}. It is a Steinberg correspondence class in the equivariant cohomology $H_{G}^*(X\times X^A)$. Here $X$ is a symplectic resolution, $X^A$ is the torus $A$-fixed point over $X$, and $A\subset G$ is a subtorus of a Lie group $G$. The stable envelope is a unique class determined by the root chamber $\mc{C}$ of the corresponding Lie group $G$, the supporting condition and the degree condition.

In the case of the quiver varieties $M_{Q}(\mbf{v},\mbf{w})$ of quiver type $Q$, stable envelopes can be used to construct the geometric $R$-matrix $R_{Q,\mc{C}_{ij}}(u_i-u_j)$. The geometric $R$-matrix satisfies the rational Yang-Baxter equation with spectral parametres. Using the RTT formalism, we can use the geometric $R$-matrix to construct the Maulik-Okounkov Yangian algebra $Y_{\hbar}^{MO}(\mf{g}_{Q})$. The MO Yangian algebra $Y_{\hbar}^{MO}(\mf{g}_{Q})$ contains the Maulik-Okounkov Lie algebra $\mf{g}_{Q}^{MO}$, which is generated by the first-order expansion coefficients $r_{Q}$ of the geometric $R$-matrix:
\begin{align}
R_{Q}(u)=\text{Id}+\frac{r_{Q}}{u}+O(u^{-2})
\end{align}
This gives a natural filtration on $Y_{\hbar}^{MO}(\mf{g}_{Q})$ such that $\text{gr }Y_{\hbar}^{MO}(\mf{g}_{Q})\cong U(\mf{g}_{Q}[u])$ as mentioned in \cite{MO12}.

The Maulik-Okounkov Lie algebra admits the root decomposition written as:
\begin{align}
\mf{g}_{Q}=\mf{h}\oplus\bigoplus_{\alpha\in\mbb{Z}^I}\mf{g}_{\alpha}=\mf{h}\oplus\mf{n}_{Q}^{MO}\oplus\mf{n}_{Q}^{MO,-}
\end{align}

Here $\mf{n}_{Q}^{MO}$ is the Lie subalgebra consisting of all the positive root vectors. It was proved in \cite{BD23} that $\mf{g}_{Q}^{MO}$ is isomorphic to the the Borcherd-Kac-Moody algebra $\mf{g}_{Q}$ of the corresponding quiver type $Q$ using the BPS Lie algebra and the perverse filtration over the preprojective cohomological Hall algebra. It was also proved in \cite{SV23} that the positive half of the MO Yangian is isomorphic to the preprojective cohomological Hall algebra of the corresponding quiver type $Q$.

Moreover, stable envelopes can also be defined in the settings of the equivariant $K$-theory. It was introduced in \cite{O15} and \cite{OS22}, and the existence was proved in \cite{AO21} and \cite{O21}.  It is a class $\text{Stab}_{\mc{C},s}$ in $K_{G}(X\times X^A)$ with similar restriction condition as in the case of the equivariant cohomology, and it not only is determined by the root chamber $\mc{C}$, but also determined by the slope parametre $s\in\text{Pic}(X)\otimes\mbb{Q}$. 

Similarly we can define the $K$-theoretic geometric $R$-matrix $\mc{R}_{Q,\mc{C}_{ij}}^s(u_i/u_j)$ of slope $s$ from the $K$-theoretic stable envelope of slope $s$. Also one can use the RTT formalism to construct the Maulik-Okounkov quantum affine algebra $U_{q}^{MO}(\hat{\mf{g}}_Q)$. It has been proved in \cite{N23} that there is an algebra embedding of the Drinfeld double $\mc{A}$ of the preprojective $K$-theoretic Hall algebra into the MO quantum affine algebra:
\begin{align}
\mc{A}\hookrightarrow U_{q}^{MO}(\hat{\mf{g}}_Q)
\end{align}

In the case when $Q$ is the affine type $A$, the Drinfeld double $\mc{A}$ is isomorphic to the quantum toroidal algebra $U_{q,t}(\hat{\hat{\mf{sl}}}_n)$, so we have the algebra embedding:
\begin{align}
U_{q,t}(\hat{\hat{\mf{sl}}}_n)\hookrightarrow U_{q}^{MO}(\hat{\mf{g}}_Q)
\end{align}

\subsection{Quantum difference equations}

Quantum difference equation was first introduced in \cite{O15} as the difference equation of the Kahler variable $z$ for the capping operator $\mbf{J}(u,z)$. The capping operator appears in the theory of  quasimap counting from $\mbb{P}^1$ to the quiver varieties $M(\mbf{v},\mbf{w})$. The difference equation is written as:
\begin{align}
\mbf{J}(p^{\mc{L}}z,u)\mc{L}=\mbf{M}_{\mc{L}}(u,z)\mbf{J}(z,u)
\end{align}

The operator $\mbf{M}_{\mc{L}}(u,z)$ is called the geometric quantum difference operator. It turns out that \cite{OS22} the geometric quantum difference operator $\mbf{M}_{\mc{L}}(u,z)$ is conjugate to the ordered product of the geometric monodromy operators $\mbf{B}^{MO}_{w}(z)$ via the stable envelope $\text{Stab}_{\mc{C},s}$ of slope $s$:
\begin{align}
\mc{B}_{\mc{L}}^{MO,s}(z)=\mc{L}\prod^{\leftarrow}_{w\in[s,s-\mc{L})}\mbf{B}_{w}^{MO}(z)
\end{align}
with each monodromy operator $\mbf{B}_{w}^{MO}(z)\in\widehat{U_{q}^{MO}(\mf{g}_w)}$ an element in the completion of the MO wall subalgebra.

It is also proved in\cite{OS22} that the difference operators $\mc{A}^{s}_{\mc{L}}:=T_{\mc{L}}^{-1}\mbf{M}_{\mc{L}}(u,z)$ commute with each other for arbitrary $\mc{L}\in\text{Pic}(M(\mbf{v},\mbf{w}))$, and it is independent of the choice of the path from $s$ to $s-\mc{L}$.

There is also the algebraic analog of the quantum difference equation. In \cite{Z23}, we have constructed the algebraic quantum difference equation for the quantum toroidal algebra $U_{q,t}(\hat{\hat{\mf{sl}}}_n)$. It is constructed via constructing the monodromy operators $\mbf{B}_{\mbf{m}}(z)\in\widehat{\mc{B}_{\mbf{m}}}$ in each slope subalgebras $\mc{B}_{\mbf{m}}$. It is proved in \cite{Z23} that the algebraic quantum difference operator:
\begin{align}
\mc{B}_{\mc{L}}^{s}=\mc{L}\prod^{\leftarrow}_{\mbf{m}\in[s,s-\mc{L})}\mbf{B}_{\mbf{m}}(z)
\end{align}
is independent of the choice of the path from $s$ to $s-\mc{L}$.

\subsection{Main result of the paper}
The main result of the paper is the isomorphism of the positive half of the MO quantum affine algebra of affine type $A$ and the positive half of the quantum toroidal algebras:
\begin{thm}{(See Theorem \ref{main-theorem-1})}
The positive half of the quantum toroidal algebra $U_{q,t}^{+}(\hat{\hat{\mf{sl}}}_n)$ is isomorphic to the positive half of the Maulik-Okounkov quantum affine algebra $U_{q}^{MO,+}(\hat{\mf{g}}_{Q})$ for the affine type $A$.
\end{thm}

The strategy of the proof depends on the analysis of the degeneration limit, or cohomological limit, of the geometric quantum difference equation of the affine type $A$. Then we do the comparison with the degeneration limit of the algebraic quantum difference equation studied in \cite{Z23}\cite{Z24}.

Here we give the sketch of the proof of the Theorem \ref{main-theorem-1} in this paper:
\begin{enumerate}
\item The property of both MO quantum affine algebra and quantum toroidal algebra that is used in the proof is the factorisation property \cite{OS22}\cite{N15}, which means that they are all generated by the wall $R$-matrices $R_{w}^{MO}$ and $R_{w}$ and the Cartan generators. For the positive half, we only need to focus on the wall $R$-matrices in both quantum toroidal algebra and MO quantum affine algebra. We need to prove that $R_{w}^{MO}=R_{w}$.
\item In \cite{Z24}, we have proved that the degeneration limit of the algebraic quantum difference equation of affine type $A$ is the Dubrovin connection of the quantum equivariant cohomology of the affine type $A$ quiver varieties. The Dubrovin connection of affine type $A$ is written as:
\begin{align}\label{dubrovin-baby}
\nabla_{\lambda}:=d_{\lambda}-Q(\lambda)=d_{\lambda}-c_1(\lambda)\cup--(\sum_{i<j}\frac{(\lambda\cdot[i,j))}{1-q^{-[i,j)}}E_{[i,j)}E_{-[i,j)}),\qquad E_{\pm[i,j)}\in\hat{\mf{sl}}_n
\end{align}

This connection has the regular singularities around $q^{-[i,j)}=1$ and $q=0,\infty$. Using the result from the quantum difference equations, one can show that \cite{Z24} the monodromy representation of the flat section of \ref{dubrovin-baby} is generated by $\mbf{m}((1\otimes S_{\mbf{m}})(R_{\mbf{m}}^{-})^{-1})$.

\item In this paper we first show that we have the similar degeneration limit of the geometric quantum difference equation to the Dubrovin connections:
\begin{thm}{(See Theorem \ref{degeneration-of-geometric-qde})}
The degeneration limit $\mc{A}^{s,coh}_{\mc{L}}$ of the Okounkov-Smirnov quantum difference equation $\mc{A}^{s}_{\mc{L}}$ of the affine type $A$ quiver varieties is the Dubrovin connection of the affine type $A$ quiver varieties.
\end{thm}

Using the similar analysis of the algebraic quantum difference equations. We can compute the generators of the monodromy representations in terms of the MO quantum affine algebras:
\begin{thm}{(See Theorem \ref{geometry-monodromy-rep-casimir})}
The monodromy representation:
\begin{align}
\pi_{1}(\mbb{P}^n\backslash\textbf{Sing},0^{+})\rightarrow\text{Aut}(H_{T}(M(\mbf{v},\mbf{w})))
\end{align}
of the Dubrovin connection is generated by $\mbf{B}_{w,MO}^*$ with $q_1=e^{2\pi i\hbar_1},q_2=e^{2\pi i\hbar_2}$, i.e. the monodromy operators $\mbf{B}_{w,MO}=\mbf{m}((1\otimes S_{w})(R_{w}^{-,MO})^{-1})$ in the fixed point basis. The based point $0^+$ is a point sufficiently close to $0\in\mbb{P}^n$.
\end{thm}

\item We can also show that the fundmental solution $\Psi_{0,\infty}^{(MO)}(z)$ at $z=0,\infty$ of both algebraic and geometric quantum difference equation have the same degeneration limit solution $\psi_{0,\infty}(z)$ of the Dubrovin connection. Using the result, we can obtain that the monodromy representation from algebraic and geometric sides are the same:
\begin{thm}{(See Theorem \ref{same-monodromy})}
As the element in $\text{End}(K(\mbf{v},\mbf{w}))$, we have that:
\begin{align}
\mbf{m}((1\otimes S_{w})R_{w}^{MO})=\mbf{m}((1\otimes S_{w})R_{w})
\end{align}
\end{thm}

Thus using the lemma \ref{important-lemma}, we obtain that $R_{w}^{MO}=R_{w}$. Since $U_{q}^{+,MO}(\hat{\mf{g}}_{Q})$ is generated by $R_{w}^{MO}$, $U_{q,t}^{+}(\hat{\hat{\mf{sl}}}_n)$ is generated by $R_{w}$ because of the slope factorisation. We obtain the main result.

Note that the computation for the Jordan case is a little bit different, and we leave all the computation in another paper\cite{Z24-2}.

\end{enumerate}

\subsection{Outline of the paper}
The paper is organised as follows: In section two we introduce the quantum toroidal algebra $U_{q,t}(\hat{\hat{\mf{sl}}}_n)$, its slope subalgebra and the slope factorisation. We also introduce the wall subalgebra in the slope subalgebra, and the quantum toroidal algebra action on the equivariant $K$-theory of the Nakajima quiver varieties. 

In section three we introduce the Maulik-Okounkov quantum affine algebra and its corresponding Maulik-Okounkov wall subalgebra. We will explain how to get the degeneration limit, or cohomological limit, of the MO wall subalgebra, to the Lie subalgebra of the MO Lie algebra.

In section four we will provide key technique in the paper, we will show that the wall subalgebra of the quantum toroidal algebra is the Hopf subalgebra of the Maulik-Okounkov wall subalgebra. Using the result we prove that the degeneration limit of the wall subalgebra and the MO wall subalgebra are the same.

In section five we introduce the Okounkov-Smirnov geometric quantum difference equation and the algebraic quantum difference equation. We will give the analysis of the solution of the Okounkov-Smirnov quantum difference equation.

In section six we will show that the degeneration limit of the geometric quantum difference equation is the Dubrovin connection. Using this fact we will compute the monodromy representation of the Dubrovin connection in terms of the wall $R$-matrices of the MO wall subalgebra.

In section seven we will combine all the results above to prove the main theorem.

\subsection{Further directions}
The main theorem \ref{main-theorem-1} implies that the algebraic quantum difference equation is the same as the Okounkov-Smirnov quantum difference equation of affine type $A$. This implies that the regular part of the fundamental solution of the algebraic quantum difference equation is the capping operator of the affine type $A$ quiver varieties.

One further direction is the generalisation of the theorem to arbitrary quiver type. This requires the further understanding about the slope subalgebra for the Drinfeld double of arbitrary preprojective $K$-theoretic Hall algebra. The slope subalgebra of arbitrary preprojective KHA was defined in \cite{N22}. The techniques about dealing with the degeneration limit of the slope subalgebra of arbitrary type is different, and this is a work in progress \cite{Zhu}

The second further direction is the computation of the $R$-matrix of the quantum toroidal algebra valued in the Fock module. The computation via the generators of the quantum toroidal algebra can be complicated via the bosonic operators, while the computation on the geometric side can be reduced via the techniques of "the universal cover of the quiver" mentioned in \cite{MO12}. It turns out that in this case we can write down the geometric $R$-matrix in terms of the fermionic operators as in \cite{S16}. The computation on both sides would reveal some boson/fermion correspondence for the $R$-matrix. This might reveal new structures in the theory of symmetric polynomials. We will put it in our future work.

The third further direction is on the difference equation related to the vertex function associated to the affine type $A$ quiver varieties. For the bare vertex function, there is a strong connection between the qKZ equation and the bare vertex functions. As it was shown in \cite{KPSZ21}\cite{PSZ20} that the vertex function for the cotangent bundle of flag varieties are related to the trigonometric Ruijsenaars-Schneider model, which can be constructed via the type $A$ trigonometric $R$-matrix. We expect that the work would give light to the details of the connection between the vertex function and the quantum integrable system related to the quantum toroidal algebra. For the capped and capping vertex function, one can also use the quantum toroidal algebra representation theory methods to do the calculation, see \cite{AD24} for the computation of the capped vertex functions.

The vertex function can also been interpreted as the "inner product" of Whittaker vectors for Verma modules of the quantised Coulomb Branch. It has been conjectured that the equivariant $K$-theory of the quasimap moduli space from $\mbb{P}^1$ to quiver varieties is the Verma module for the $K$-theoretic quantised Coulomb Branch of the corresponding quiver type. The conjecture is partially confirmed for hypertoric varieties in \cite{ZZ23}. We expect that the isomorphism would provide tools for proving the conjecture.

\subsection*{Acknowledgments.}The author would like to thank Andrei Negu\c{t}, Andrei Okounkov, Andrey Smirnov, Hunter Dinkins and Nicolai Reshetikhin for their helpful discussions on shuffle algebras, quantum difference equations and qKZ equations. The author is supported by the international collaboration grant BMSTC and ACZSP (Grant no. Z221100002722017). Part of this work was done while the author was visiting Department of Mathematics at Columbia University. The author thanks for their hospitality and provisio of excellent working environment.

\section{\textbf{Quantum toroidal algebra and quiver varieties}}

\subsection{Quantum toroidal algebra $U_{q,t}(\hat{\hat{\mf{sl}}}_{n})$}

The quantum toroidal algebra $U_{q,t}(\hat{\hat{\mf{sl}}}_{n})$ is a $\mbb{Q}(q,t)$-algebra defined as:
\begin{align}
U_{q,t}(\hat{\hat{\mf{sl}}}_{n})=\mbb{Q}(q,t)\langle\{e_{i,d}^{\pm}\}_{1\leq i\leq n}^{d\in\mbb{Z}},\{\varphi_{i,d}^{\pm}\}_{1\leq i\leq n}^{d\in\mbb{N}_{0}}\rangle/(~)
\end{align}
The relation between the generators can be described in the generating functions:
\begin{equation}
e_i^{ \pm}(z)=\sum_{d \in \mathbb{Z}} e_{i, d}^{ \pm} z^{-d} \quad \varphi_i^{ \pm}(z)=\sum_{d=0}^{\infty} \varphi_{i, d}^{ \pm} z^{\mp d}
\end{equation}

with $\varphi_{i,d}^{\pm}$ commute among themselves and:
\begin{equation}
\begin{gathered}
e_i^{ \pm}(z) \varphi_j^{ \pm^{\prime}}(w) \cdot \zeta\left(\frac{w^{ \pm 1}}{z^{ \pm 1}}\right)=\varphi_j^{ \pm^{\prime}}(w) e_i^{ \pm}(z) \cdot \zeta\left(\frac{z^{ \pm 1}}{w^{ \pm 1}}\right) \\
e_i^{ \pm}(z) e_j^{ \pm}(w) \cdot \zeta\left(\frac{w^{ \pm 1}}{z^{ \pm 1}}\right)=e_j^{ \pm}(w) e_i^{ \pm}(z) \cdot \zeta\left(\frac{z^{ \pm 1}}{w^{ \pm 1}}\right) \\
{\left[e_i^{+}(z), e_j^{-}(w)\right]=\delta_i^j \delta\left(\frac{z}{w}\right) \cdot \frac{\varphi_i^{+}(z)-\varphi_i^{-}(w)}{q-q^{-1}}}
\end{gathered}
\end{equation}

Here $i,j\in\{1,\cdots,n\}$ and $z$,$w$ are variables of color $i$ and $j$. Here:
\begin{equation}
\zeta\left(\frac{x_i}{x_j}\right)=\frac{\left[\frac{x_j}{q t x_i}\right]^{\delta_{j-1}^i}\left[\frac{t x_j}{q x_i}\right]^{\delta_{j+1}^i}}{\left[\frac{x_j}{x_i}\right]^{\delta_j^i}\left[\frac{x_j}{q^2 x_i}\right]^{\delta_j^i}}
\end{equation}

with the Serre relation:
\begin{equation}
\begin{aligned}
& e_i^{ \pm}\left(z_1\right) e_i^{ \pm}\left(z_2\right) e_{i \pm^{\prime} 1}^{ \pm}(w)+\left(q+q^{-1}\right) e_i^{ \pm}\left(z_1\right) e_{i \pm^{\prime} 1}^{ \pm}(w) e_i^{ \pm}\left(z_2\right)+e_{i \pm^{\prime} 1}^{ \pm}(w) e_i^{ \pm}\left(z_1\right) e_i^{ \pm}\left(z_2\right)+\\
& +e_i^{ \pm}\left(z_2\right) e_i^{ \pm}\left(z_1\right) e_{i \pm^{\prime} 1}^{ \pm}(w)+\left(q+q^{-1}\right) e_i^{ \pm}\left(z_2\right) e_{i \pm^{\prime} 1}^{ \pm}(w) e_i^{ \pm}\left(z_1\right)+e_{i \pm^{\prime} 1}^{ \pm}(w) e_i^{ \pm}\left(z_2\right) e_i^{ \pm}\left(z_1\right)=0
\end{aligned}
\end{equation}

The standard coproduct structure is imposed as:
\begin{equation}
\Delta: U_{q,t}(\hat{\hat{\mf{sl}}}_{n})\longrightarrow U_{q,t}(\hat{\hat{\mf{sl}}}_{n})\widehat{\otimes} U_{q,t}(\hat{\hat{\mf{sl}}}_{n})
\end{equation}

\begin{equation}
\begin{array}{ll}
\Delta\left(e_i^{+}(z)\right)=\varphi_i^{+}(z) \otimes e_i^{+}(z)+e_i^{+}(z) \otimes 1 & \Delta\left(\varphi_i^{+}(z)\right)=\varphi_i^{+}(z) \otimes \varphi_i^{+}(z) \\
\Delta\left(e_i^{-}(z)\right)=1 \otimes e_i^{-}(z)+e_i^{-}(z) \otimes \varphi_i^{-}(z) & \Delta\left(\varphi_i^{-}(z)\right)=\varphi_i^{-}(z) \otimes \varphi_i^{-}(z)
\end{array}
\end{equation}

\subsection{Quantum affine algebra $U_{q}(\hat{\mf{gl}}_{n})$}
The quantum affine algebra $U_{q}(\hat{\mf{gl}}_{n})$ is a $\mbb{Q}(q)$-algebra generated by:
\begin{align}
\mbb{Q}(q)\langle e_{\pm[i;j]},\psi_{s}^{\pm1},c^{\pm1}\rangle^{s\in\{1,\cdots,n\}}_{(i<j)\in\mbb{Z}^2/(n,n)\mbb{Z}}
\end{align}
The generators $e_{\pm[i;j)}$ satisfy the well-known RTT relation:
\begin{align}
R(\frac{z}{w})T_{1}^{+}(z)T_{2}^{+}(w)=T_{2}^{+}(w)T_{1}^{+}(z)R(\frac{z}{w})
\end{align}
\begin{align}
R(\frac{z}{w})T_{1}^{-}(z)T_{2}^{-}(w)=T_{2}^{-}(w)T_{1}^{-}(z)R(\frac{z}{w})
\end{align}
\begin{align}
R(\frac{z}{wc})T_{2}^{-}(w)T_{1}^{+}(z)=T_{1}^{+}(z)T_{2}^{-}(w)R(\frac{zc}{w})
\end{align}

Here $T^{\pm}(z)$ are the generating function for $e_{\pm[i;j)}$:
\begin{equation}
\begin{gathered}
T^{+}(z)=\sum_{1 \leq i \leq n}^{i \leq j} e_{[i ; j)} \psi_i \cdot E_{j \bmod n, i} z^{\left\lceil\frac{j}{n}\right\rceil-1} \\
T^{-}(z)=\sum_{1 \leq i \leq n}^{i \leq j} e_{-[i ; j)} \psi_i^{-1} \cdot E_{i, j \bmod n} z^{-\left\lceil\frac{j}{n}\right\rceil+1}
\end{gathered}
\end{equation}

And $R(z/w)$ is the standard $R$-matrix for $\mf{gl}_{n}$:
\begin{equation}
R\left(\frac{z}{w}\right)=\sum_{1 \leq i, j \leq n} E_{i i} \otimes E_{j j}\left(\frac{z q-w q^{-1}}{w-z}\right)^{\delta_j^i}+\left(q-q^{-1}\right) \sum_{1 \leq i \neq j \leq n} E_{i j} \otimes E_{j i} \frac{w^{\delta_{i>j}} z^{\delta_{i<j}}}{w-z}
\end{equation}

In addition, we have the relations:
\begin{equation}
\psi_k \cdot e_{ \pm[i ; j)}=q^{ \pm\left(\delta_k^i-\delta_k^j\right)} e_{ \pm[i ; j)} \cdot \psi_k \quad \forall \operatorname{arcs}[i ; j) \text { and } k \in \mathbb{Z}
\end{equation}

The algebra $U_{q}(\hat{\mf{gl}}_{n})$ has the coproduct structure given by:
\begin{align}
\Delta(T^{+}(z))=T^{+}(z)\otimes T^{+}(zc_1),\qquad\Delta(T^{-}(z))=T^{-}(zc_2)\otimes T^{-}(z)
\end{align}

and the antipode map:
\begin{align}
S(T^{\pm}(z))=(T^{\pm}(z))^{-1}
\end{align}

The corresponding antipode map elements can be written as:
\begin{equation}
\begin{gathered}
S^{+}(z)=\sum_{1 \leq i \leq n}^{i \leq j}(-1)^{j-i} f_{[i ; j)} \psi_j \cdot E_{j \bmod n, i} z^{\left\lceil\frac{j}{n}\right\rceil-1} \\
S^{-}(z)=\sum_{1 \leq i \leq n}^{i \leq j}(-1)^{j-i} f_{-[i ; j)} \psi_j^{-1} \cdot E_{i, j \bmod n} z^{-\left\lceil\frac{j}{n}\right]+1}
\end{gathered}
\end{equation}

Unwinding the relation we have that:
\begin{align}
\frac{e_{\pm[a,c)}e_{\pm[b,d)}}{q^{\delta^b_a-\delta^d_b+\delta^d_a}}-\frac{e_{\pm[b,d)}e_{\pm[a,c)}}{q^{\delta^b_c-\delta^d_b+\delta^d_c}}=(q-q^{-1})[\sum_{a\leq x<c}^{x\equiv d}e_{\pm[b,c+d-x)}e_{\pm[a,x)}-\sum_{a<x\leq c}^{x\equiv b}e_{\pm[x,c)}e_{\pm[a+b-x,d)}]
\end{align}

\begin{align}\label{gln-relation}
[e_{[a,c)},e_{-[b,d)}]=(q-q^{-1})[\sum^{x\equiv b}_{a\leq x<c}\frac{e_{[c+b-x,d)}e_{[a,x)}}{q^{\delta^b_c+\delta^a_c-\delta^a_b}}\frac{\psi_x}{\psi_c}-\sum_{a<x\leq c}^{x\equiv d}\frac{e_{[x,c)}e_{-[b,a+d-x)}}{q^{-\delta^b_a+\delta^d_b-\delta^d_a}}\frac{\psi_x}{\psi_a}]
\end{align}

The isomorphism $U_{q}(\hat{\mf{gl}}_{n})\cong U_{q}(\hat{\mf{sl}}_{n})\otimes U_{q}(\hat{\mf{gl}}_{1})$ is given by the following:
\begin{align}
e_{[i;i+1)}=x_{i}^{+}(q-q^{-1})\\
e_{-[i;i+1)}=x_{i}^{-}(q^{-2}-1)
\end{align}

For the generators $\{p_{k}\}_{k\in\mbb{Z}}$ of the Heisenberg algebra $U_{q}(\hat{\mf{gl}}_{1})\subset U_{q}(\hat{\mf{gl}}_n)$, they satisfy the following commutation relations:
\begin{align}
[p_k,p_{-k}]=\frac{c^{nk}-c^{-nk}}{n_k},\qquad n_k=\frac{q^k-q^{-k}}{k}
\end{align}

It can be packaged into the following proposition proved by Negu\c{t} in \cite{N19}:
\begin{prop}\label{implicit-heisenberg-generators}
There exists constants $\alpha_1,\alpha_2,\cdots\in\mbb{Q}(q)$ such that:
\begin{equation}
f_{ \pm[i ; j)}=\sum_{k=0}^{\left\lfloor\frac{j-i}{n}\right\rfloor} e_{ \pm[i ; j-n k)} g_{ \pm k}
\end{equation}
Here $\sum_{k=0}^{\infty}g_{\pm k}x^{k}=\exp(\sum_{k=1}^{\infty}\alpha_{k}p_{\pm k}x^k)$
\end{prop}

\subsection{Shuffle algebra realization of $U_{q,t}(\hat{\hat{\mf{sl}}}_n)$}
Here we review the reconstruction of the quantum toroidal algebra via the shuffle algebra and the induced slope factorization of the quantum toroidal algebra. For details see \cite{N15}.

Fix a fractional field $\mbb{F}=\mbb{Q}(q,t)$ and consider the space of symmetric rational functions:
\begin{align}
\widehat{\text{Sym}}(V):=\bigoplus_{\mbf{k}=(k_1,\cdots,k_{n})\in\mbb{N}^n}\mbb{F}(\cdots,z_{i1},\cdots,z_{ik_i},\cdots)^{\text{Sym}}_{1\leq i\leq n}
\end{align}
Here "Sym" means to symmetrize the rational function for each $z_{i1},\cdots,z_{ik_i}$. We endow the vector space with the shuffle product:
\begin{align}
F*G=\text{Sym}[\frac{F(\cdots,z_{ia},\cdots)G(\cdots,z_{jb},\cdots)}{\mbf{k}!\mbf{l}!}\prod_{1\leq a\leq k_i}^{1\leq i\leq n}\prod_{k_{j}+1\leq b\leq k_{j}+l_{j}}^{1\leq j\leq n}\zeta(\frac{z_{ia}}{z_{jb}})]
\end{align}

We define the subspace $\mc{S}^{\pm}_{\mbf{k}}\subset\widehat{\text{Sym}}(V)$ by:
\begin{align}
\mc{S}^{+}:=\{F(\cdots,z_{ia},\cdots)=\frac{r(\cdots,z_{ia},\cdots)}{\prod_{1\leq a\neq b\leq k_i}^{1\leq i\leq n}(qz_{ia}-q^{-1}z_{ib})}\}
\end{align}

where $r(\cdots,z_{i1},\cdots,z_{ik_i},\cdots)^{1\leq i\leq n}_{1\leq a\leq k_i}$ is any symmetric Laurent polynomial that satisfies the wheel conditions:
\begin{align}
r(\cdots,q^{-1},t^{\pm},q,\cdots)=0
\end{align}
for any three variables of colors $i,\cdots,i\pm1,i$.

The shuffle algebra $\mc{S}^+$ has two natural bigrading given by the number of the variables in each color $\mbf{k}\in\mbb{N}^n$ and the homogeneous degree of the rational functions $d\in\mbb{Z}$:
\begin{align}
\mc{S}^+=\bigoplus_{(\mbf{k},d)\in\mbb{N}^n\times\mbb{Z}}\mc{S}^+_{\mbf{k},d}
\end{align}

Similarly we can define the negative shuffle algebra $\mc{S}^{-}:=(\mc{S}^{+})^{op}$ which is the same as $\mc{S}^{+}$ with the opposite shuffle product. Now we slightly enlarge the positive and negative shuffle algebras by the generators $\{\varphi_{i,d}^{\pm}\}_{1\leq i\leq r}^{d\geq0}$:
\begin{align}
\mc{S}^{\geq}=\langle\mc{S}^+,\{(\varphi^+_{i,d})^{d\geq0}_{1\leq i\leq r}\}\rangle,\qquad \mc{S}^{\leq}=\langle\mc{S}^-,\{(\varphi^-_{i,d})^{d\geq0}_{1\leq i\leq r}\}\rangle
\end{align}
And here $\varphi^{\pm}_{i,d}$ commute with themselves and have the relation with $\mc{S}^{\pm}$ as follows:
\begin{align}\label{bialgebra-pairing-1}
\varphi^{+}_i(w)F=F\varphi^+_i(w)\prod_{1\leq a\leq k_j}^{1\leq j\leq n}\frac{\zeta(w/z_{ja})}{\zeta(z_{ja}/w)}
\end{align}
\begin{align}\label{bialgebra-pairing-2}
\varphi^{-}_i(w)G=G\varphi^-_i(w)\prod_{1\leq a\leq k_j}^{1\leq j\leq n}\frac{\zeta(z_{ja}/w)}{\zeta(w/z_{ja})}
\end{align}

The Drinfeld pairing $\langle-,-\rangle:\mc{S}^{\leq}\otimes\mc{S}^{\geq}\rightarrow\mbb{Q}(q,t)$ between the positive and negative shuffle algebras $\mc{S}^{\geq,\neq}$ is given by:
\begin{align}
&\langle\varphi^{-}_{i}(z),\varphi^{+}_{j}(w)\rangle=\frac{\zeta(w/z)}{\zeta(z/w)}\\
&\langle G,F\rangle=\frac{1}{\mbf{k}!}\int_{\lvert z_{ia}\lvert=1}^{\lvert q\lvert<1\lvert p\lvert}\frac{G(\cdots,z_{ia},\cdots)F(\cdots,z_{ia},\cdots)}{\prod_{a\leq k_i,b\leq k_j}^{1\leq i,j\leq n}\zeta_{p}(z_{ia}/z_{jb})}\prod_{1\leq a\leq k_i}^{1\leq i\leq n}\frac{dz_{ia}}{2\pi iz_{ia}}|_{p\mapsto q}
\end{align}
for $F\in\mc{S}^{+}$, $G\in\mc{S}^{-}$. Here $\zeta_{p}$ is defined as follows:
\begin{align}
\zeta_{p}(\frac{x_i}{x_j})=\zeta(\frac{x_i}{x_j})\frac{[\frac{x_j}{p^2x_i}]^{\delta^{i}_{j}}}{[\frac{x_j}{q^2x_i}]^{\delta^{i}_{j}}}
\end{align}

This defines the shuffle algebra $\mc{S}:=\mc{S}^{\leq}\hat{\otimes}\mc{S}^{\geq}$. Also the coproduct $\Delta:\mc{S}\rightarrow\mc{S}\hat{\otimes}\mc{S}$ is given by:
\begin{align}
&\Delta(\varphi_{i}^{\pm}(w))=\varphi^{\pm}_{i}(w)\otimes\varphi^{\pm}_{i}(w)\\
&\Delta(F)=\sum_{0\leq\mbf{l}\leq\mbf{k}}\frac{[\prod_{1\leq j\leq n}^{b>l_{j}}\varphi_{j}^+(z_{jb})]F(\cdots,z_{i1},\cdots,z_{il_i}\otimes z_{i,l_i+1},\cdots,z_{ik_i},\cdots)}{\prod_{1\leq i\leq n}^{a\leq l_i}\prod_{1\leq j\leq n}^{b>l_j}\zeta(z_{jb}/z_{ia})},\qquad F\in\mc{S}^+\\
&\Delta(G)=\sum_{0\leq\mbf{l}\leq\mbf{k}}\frac{G(\cdots,z_{i1},\cdots,z_{il_i}\otimes z_{i,l_i+1},\cdots,z_{ik_i},\cdots)[\prod_{1\leq j\leq n}^{b>l_{j}}\varphi_{j}^-(z_{jb})]}{\prod_{1\leq i\leq n}^{a\leq l_i}\prod_{1\leq j\leq n}^{b>l_j}\zeta(z_{ia}/z_{jb})},\qquad G\in\mc{S}^-
\end{align}

In human language, the above coproduct formula means the following:For the right hand side of the formula, in the limit $\lvert z_{ia}\lvert<<\lvert z_{jb}\lvert$ for all $a\leq l_i$ and $b>l_{i}$, and then place all monomials in $\{z_{ia}\}_{a\leq l_i}$ to the left of the $\otimes$ symbol and all monomials in $\{z_{jb}\}_{b>l_j}$ to the right of the $\otimes$ symbol. Also we have the antipode map $S:\mc{S}\rightarrow\mc{S}$ which is an anti-homomorphism of both algebras and coalgebras:
\begin{align}
&S(\varphi^+_i(z))=(\varphi^+_i(z))^{-1},\qquad S(F)=[\prod_{1\leq a\leq k_i}^{1\leq i\leq n}(-\varphi_i^+(z_{ia}))^{-1}]*F\\
&S(\varphi^-_i(z))=(\varphi^-_i(z))^{-1},\qquad S(G)=G*[\prod_{1\leq a\leq k_i}^{1\leq i\leq n}(-\varphi_i^-(z_{ia}))^{-1}]
\end{align}

The following theorem has been proved by\cite{N15}:
\begin{thm}
There is a bigraded isomorphism of bialgebras
\begin{align}
Y:U_{q,t}(\hat{\hat{\mf{sl}}}_{n})\rightarrow\mc{S}
\end{align}
given by:
\begin{align}
Y(\varphi_{i,d}^{\pm})=\varphi^{\pm}_{i,d},\qquad Y(e_{i,d}^{\pm})=\frac{z_{i1}^d}{[q^{-2}]}
\end{align}
\end{thm}

\subsection{Slope subalgebra of the shuffle algebra}
One of the wonderful property for the shuffle algebra is that it admits the slope factorization.
For the positive half of the shuffle algebra $\mc{S}^+$, we define the subspace of slope $\leq\mbf{m}$ as:
\begin{align}
\mc{S}^{+}_{\leq\mbf{m}}:=\{F\in\mc{S}^+|\lim_{\xi\rightarrow\infty}\frac{F(z_{i1},\cdots,z_{ia_i},\xi z_{i,a_i+1},\cdots,\xi z_{i,a_i+b_i})}{\xi^{\mbf{m}\cdot\mbf{b}-\frac{(\mbf{a},\mbf{b})}{2}}}<\infty\text{ }\forall\mbf{a},\mbf{b}\}
\end{align}

Similarly, for the negative half:
\begin{align}
\mc{S}^{-}_{\leq\mbf{m}}=\{G\in\mc{S}^-|\lim_{\xi\rightarrow0}\frac{G(\xi z_{i1},\cdots,\xi z_{ia_i},z_{i,a_i+1},\cdots,z_{i,a_i+b_i})}{\xi^{-\mbf{m}\cdot\mbf{a}+\frac{(\mbf{a},\mbf{b})}{2}}}<\infty\text{ }\forall\mbf{a},\mbf{b}\}
\end{align}

It can be checked that for the elements $F\in\mc{S}_{\leq\mbf{m}}^{+}$, $G\in\mc{S}_{\leq\mbf{m}}^{-}$:
\begin{align}
\Delta(F)=\Delta_{\mbf{m}}(F)+(\text{anything})\otimes(\text{slope}<\mbf{m}),\qquad\Delta_{\mbf{m}}(F)=\sum_{\mbf{a}+\mbf{b}=\mbf{k}}\lim_{\xi\rightarrow\infty}\frac{\Delta(F)_{\mbf{a}\otimes\mbf{b}}}{\xi^{\mbf{m}\cdot\mbf{b}}}
\end{align}

\begin{align}
\Delta(G)=\Delta_{\mbf{m}}(G)+(\text{slope}<\mbf{m})\otimes(\text{anything}),\qquad\Delta_{\mbf{m}}(G)=\sum_{\mbf{a}+\mbf{b}=\mbf{k}}\lim_{\xi\rightarrow\infty}\frac{\Delta(G)_{\mbf{a}\otimes\mbf{b}}}{\xi^{-\mbf{m}\cdot\mbf{a}}}
\end{align}

The slope subalgebra $\mc{B}_{\mbf{m}}^{\pm}$ is defined as:
\begin{align}
\mc{B}^{\pm}_{\mbf{m}}:=\bigoplus_{(\mbf{k},d)\in\mbb{N}^n\times\mbb{Z}}^{\mbf{m}\cdot\mbf{k}=d}\mc{S}^{\pm}_{\leq\mbf{m}}\cap\mc{S}^{\pm}_{\pm\mbf{k},d}
\end{align}

It is easy to see that the morphism $\Delta_{\mbf{m}}$ descends to a true coproduct on $\mc{B}_{\mbf{m}}^{\geq,\leq}:=\mc{B}_{\mbf{m}}^{\pm}\otimes\mbb{F}[\varphi_i^{\pm1}]$. Using the Drinfeld double defined for the shuffle algebra. We eventually obtain the slope subalgebra $\mc{B}_{\mbf{m}}\subset\mc{S}$.

It is known that the slope subalgebra $\mc{B}_{\mbf{m}}$ for the quantum toroidal algebra $U_{q,t}(\widehat{\widehat{\mf{sl}}}_{n})$ is isomorphic to:
\begin{align}\label{rootquantum}
\mc{B}_{\mbf{m}}\cong\bigotimes^{g}_{h=1}U_{q}(\widehat{\mf{gl}}_{l_h})
\end{align}

For the proof of the above isomorphism \ref{rootquantum} see \cite{N15}. The isomorphism is constructed in the following way. For $\mbf{m}\in\mbb{Q}^{n}$, it is proved in \cite{N15} that as the shuffle algebra, 
The generators of the slope algebra $\mc{B}_{\mbf{m}}$ can be written in the following way. For $\mbf{m}\cdot[i;j\rangle\in\mbb{Z}$, we denote the following elements:
\begin{equation}\label{positive-generators}
P_{\pm[i ; j)}^{\pm \mathbf{m}}=\operatorname{Sym}\left[\frac{\prod_{a=i}^{j-1} z_a^{\left\lfloor m_i+\ldots+m_a\right\rfloor-\left\lfloor m_i+\ldots+m_{a-1}\right\rfloor}}{t^{\text {ind } d_{[i, j\rangle}^m q^{i-j}} \prod_{a=i+1}^{j-1}\left(1-\frac{q_2 z_a}{z_{a-1}}\right)} \prod_{i \leq a<b<j} \zeta\left(\frac{z_b}{z_a}\right)\right]\in\mc{S}^{\pm}
\end{equation}
\begin{equation}\label{negative generators}
Q_{\mp[i ; j)}^{\pm \mathbf{m}}=\operatorname{Sym}\left[\frac{\prod_{a=i}^{j-1} z_a^{\left\lfloor m_i+\ldots+m_{a-1}\right\rfloor-\left\lfloor m_i+\ldots+m_a\right\rfloor}}{t^{-\mathrm{ind}_{[i, j)}^m} \prod_{a=i+1}^{j-1}\left(1-\frac{q_1 z_{a-1}}{z_a}\right)} \prod_{i \leq a<b<j} \zeta\left(\frac{z_a}{z_b}\right)\right]\in\mc{S}^{\mp}
\end{equation}

Here $\text{ind}^{\mbf{m}}_{[i;j)}$ is defined as:
\begin{align}
\text{ind}^{\mbf{m}}_{[i;j)}=\sum_{a=i}^{j-1}(m_i+\cdots+m_a-\lfloor m_i+\cdots+m_{a}\lfloor)
\end{align}

The positive and negative part of the slope subalgebra $\mc{B}_{\mbf{m}}^{\pm}$ for the shuffle algebra $\mc{S}^{\pm}$ can be defined as the algebra generated by $\{P^{\mbf{m}}_{\pm[i;j)}\}_{i\leq j}$ and $\{Q^{\mbf{m}}_{\pm[i;j)}\}_{i\leq j}$. We can define the slope subalgebra $\mc{B}_{\mbf{m}}$ as the Drinfeld double of $\mc{B}_{\mbf{m}}^{\pm}$ with the Cartan elements $\{\varphi_{\pm[i;j)}\}$ by the bialgebra pairing defined in \ref{bialgebra-pairing-1} and \ref{bialgebra-pairing-2}.

We can check that the antipode map $S_{\mbf{m}}:\mc{B}_{\mbf{m}}\rightarrow\mc{B}_{\mbf{m}}$ has the following relation:
\begin{align}
S_{\mbf{m}}(P^{\mbf{m}}_{[i;j)})=Q^{\mbf{m}}_{[i;j)},\qquad S_{\mbf{m}}(Q^{\mbf{m}}_{-[i;j)})=P^{\mbf{m}}_{-[i;j)}
\end{align}

\begin{equation}
\Delta_{\mathbf{m}}\left(P_{[i ; j)}^{\mathbf{m}}\right)=\sum_{a=i}^j P_{[a ; j)}^{\mathbf{m}} \varphi_{[i ; a)} \otimes P_{[i ; a)}^{\mathbf{m}} \quad \Delta_{\mathbf{m}}\left(Q_{[i ; j)}^{\mathrm{m}}\right)=\sum_{a=i}^j Q_{[i ; a)}^{\mathbf{m}} \varphi_{[a ; j)} \otimes Q_{[a ; j)}^{\mathbf{m}}
\end{equation}

\begin{equation}
\Delta_{\mathbf{m}}\left(P_{-[i ; j)}^{\mathbf{m}}\right)=\sum_{a=i}^j P_{-[a ; j)}^{\mathbf{m}} \otimes P_{-[i ; a)}^{\mathbf{m}} \varphi_{-[a ; j)} \quad \Delta_{\mathbf{m}}\left(Q_{-[i ; j)}^{\mathbf{m}}\right)=\sum_{a=i}^j Q_{-[i ; a)}^{\mathbf{m}} \otimes Q_{-[a ; j)}^{\mathbf{m}} \varphi_{-[i ; a)}
\end{equation}

Here $\Delta_{\mbf{m}}$ is defined as $(5.27)$ and $(5.28)$ in \cite{N15}. The isomorphism \ref{rootquantum} is given by:
\begin{align}
e_{[i;j)}=P^{\mbf{m}}_{[i;j)_{h}},\qquad e_{-[i;j)}=Q^{\mbf{m}}_{-[i;j)_h},\qquad\varphi_{k}=\varphi_{[k;v_{\mbf{m}}(k))}
\end{align}

We have the factorisation of the quantum toroidal algebra proved in \cite{N15} and \cite{Z23}:
\begin{thm}\label{factorisation}
There is an isomorphism of algebras:
\begin{align}
U_{q,t}^{+}(\hat{\hat{\mf{sl}}}_n)\cong\bigotimes_{\mu\in\mbb{Q}}\mc{B}_{\mbf{m}+\mu\bm{\theta}}^+,\qquad\mbf{m}\in\mbb{Q}^n,\bm{\theta}\in(\mbb{Z}^+)^n
\end{align}

Moreover, 
\begin{align}
U_{q,t}(\hat{\hat{\mf{sl}}}_n)\cong\bigotimes_{\mu\in\mbb{Q}\cup\{\infty\}}\mc{B}_{\mbf{m}+\mu\bm{\theta}},\qquad\mbf{m}\in\mbb{Q}^n,\bm{\theta}\in(\mbb{Z}^+)^n
\end{align}

The isomorphism induces the Hopf algebra embedding $\mc{B}_{\mbf{m}}\hookrightarrow U_{q,t}(\hat{\hat{\mf{sl}}}_n)$.
\end{thm}

The theorem gives the factorisation of the universal $R$-matrix:
\begin{align}
R^{\infty}=\prod_{m\in\mbb{Q}}^{\rightarrow}(R_{s+m\bm{\theta}}^+)R_{\infty}=(\prod_{m\in\mbb{Q}}^{\rightarrow}\prod_{h=1}^{g^{(m)}}R_{U_{q}(\hat{\mf{gl}}_{h}^{(m)})}^+)\cdot R_{\text{Heisenberg}}^{\otimes n}
\end{align}

Now we choose arbitrary $\mbf{m}\in\mbb{Q}^n$, we can twist the coproduct by:
\begin{align}\label{twisted-coproduct}
\Delta_{(\mbf{m})}(a)=[\prod^{\rightarrow}_{r\in\mbb{Q}_{>0}\cup\{\infty\}}R_{\mbf{m}+r\bm{\theta}}^+]\cdot\Delta(a)\cdot[\prod^{\rightarrow}_{r\in\mbb{Q}_{>0}\cup\{\infty\}}R_{\mbf{m}+r\bm{\theta}}^+]^{-1}
\end{align}

\begin{align}
R^{\mbf{m}}=\prod^{\leftarrow}_{\mu\in\mbb{Q}_{-}}R_{\mbf{m}+\mu\bm{\theta}}^{-}R_{\infty}\prod^{\leftarrow}_{\mu\in\mbb{Q}_{+}\sqcup\{0\}}R_{\mbf{m}+\mu\bm{\theta}}^{+}
\end{align}

The antipode map $S$ can be twisted as:
\begin{align}
S_{\mbf{m}}(a)=U_{\mbf{m}}S(a)U_{\mbf{m}}^{-1},\qquad U_{\mbf{m}}=\mbf{m}((1\otimes S)[\prod^{\rightarrow}_{r\in\mbb{Q}_{>0}\cup\{\infty\}}R_{\mbf{m}+r\bm{\theta}}^+])
\end{align}

This makes $(U_{q,t}(\hat{\hat{\mf{sl}}}_n),\Delta_{(\mbf{m})},S_{\mbf{m}},\epsilon,\eta)$ a Hopf algebra over $\mbb{Q}(q,t)$.

It turns out that $\mc{B}_{\mbf{m}}$ is a Hopf subalgebra of $(U_{q,t}(\hat{\hat{\mf{sl}}}_n),\Delta_{(\mbf{m})},S_{\mbf{m}},\epsilon,\eta)$:
\begin{prop}\label{twist-of-the-coproduct}
$\Delta_{(\mbf{m})}(a)=\Delta_{\mbf{m}}(a)$ when $a\in\mc{B}_{\mbf{m}}$. Morevoer, this means that $\mc{B}_{\mbf{m}}$ is a Hopf subalgebra of $(U_{q,t}(\hat{\hat{\mf{sl}}}_n),\Delta_{(\mbf{m})},S_{\mbf{m}},\epsilon,\eta)$.
\end{prop}
For the proof see \cite{Z23}.

\subsection{Wall subalgebra for the slope subalgebra}

The slope factorisation of the quantum toroidal algebra implies that the wall set for affine type $A$ quiver varieties are given as:
\begin{align}
\{\mbf{m}\in\mbb{R}^n|\mbf{m}\cdot[i,j)\in\mbb{Z}\}
\end{align}
This means that each wall hyperplane is indexed by $w=([i,j),n)\in\mbb{Z}^n\times\mbb{Z}$. In this way we can naively "define" the wall subalgebra $U_{q}(\mf{g}_{w})$ as:
\begin{align}
\bigoplus_{k\in\mbb{N}}^{\mbf{m}\cdot[i,j)=n}\mc{S}^{+}_{\leq\mbf{m}}\cap\mc{S}^{+}_{k[i,j),kn}
\end{align}

It is easy to check that this is a subalgebra of $\mc{B}_{\mbf{m}}$ but not a Hopf subalgebra of $\mc{B}_{\mbf{m}}$. So we need some modification of the definition.

To make sense of the wall subalgebra, first we write down the slope factorisation of the universal $R$-matrix of the quantum toroidal algebra:
\begin{align}
R^{\mbf{m}}=\prod^{\leftarrow}_{\mu\in\mbb{Q}_{-}}R_{\mbf{m}+\mu\bm{\theta}}^{-}R_{\infty}\prod^{\leftarrow}_{\mu\in\mbb{Q}_{+}\sqcup\{0\}}R_{\mbf{m}+\mu\bm{\theta}}^{+}
\end{align}

We take the universal $R$-matrix in the representation $K(\mbf{w}_1)\otimes K(\mbf{w}_2)$. Moreover, $R$ is also a linear morphism in the subspace:
\begin{align}
R^{\mbf{m}}:\bigoplus_{\mbf{v}_1+\mbf{v}_2=\mbf{v}}K(\mbf{v}_1,\mbf{w}_1)\otimes K(\mbf{v}_2,\mbf{w}_2)\rightarrow \bigoplus_{\mbf{v}_1+\mbf{v}_2=\mbf{v}}K(\mbf{v}_1,\mbf{w}_1)\otimes K(\mbf{v}_2,\mbf{w}_2)
\end{align}

and so is each slope $R$-matrix $R_{\mbf{m}+\mu\bm{\theta}}^{\pm}$. Since each slope subalgebra $\mc{B}_{\mbf{m}}$ is isomorphic to $\bigotimes_{h=1}^g U_{q}(\hat{\mf{gl}}_{l_h})$, the universal $R$-matrix for $\mc{B}_{\mbf{m}}$ can be factorised as:
\begin{align}
R_{\mbf{m}}^{\pm}=q^{\Omega}\prod^{\leftarrow}_{\alpha}R_{\alpha}^{\pm}
\end{align}
via using the Khoroshkin-Tolstoy factorisation. Here $\alpha$ are the affine roots corresponding to the wall hyperplane intersecting with the slope point $\mbf{m}$ with some order, i.e. $\langle\alpha,\mbf{m}\rangle\in\mbb{Z}$. If we restrict the morphism to $\bigoplus_{\mbf{v}_1+\mbf{v}_2=\mbf{v}}K(\mbf{v}_1,\mbf{w}_1)\otimes K(\mbf{v}_2,\mbf{w}_2)$, we can see that the ordered product $\prod^{\leftarrow}_{\alpha}R_{\alpha}^{\pm}$ will be reduced to finitely many product terms.

Thus restricted to $\bigoplus_{\mbf{v}_1+\mbf{v}_2=\mbf{v}}K(\mbf{v}_1,\mbf{w}_1)\otimes K(\mbf{v}_2,\mbf{w}_2)$, we can formally write down the universal $R$-matrix $R^{\mbf{m}}$ of slope $\mbf{m}$ as:
\begin{align}
R^{\mbf{m}}=\prod^{\leftarrow}_{w<\mbf{m}}R_{w}^-R_{\infty}\prod^{\leftarrow}_{w\geq\mbf{m}}R_{w}^+
\end{align}

It is easy to check that $q^{\Omega}R_{w}^{\pm}$ satisfies the Yang-Baxter equation, and the generators in $R_{w}^{\pm}$ generates a Hopf subalgebra.

\begin{defn}
We define the \textbf{wall subalgebra} $U_{q}(\mf{g}_{w})\subset\mc{B}_{\mbf{m}}$ as the subalgebra generated by the universal $R$-matrix $R_{w}$.
\end{defn}

In our examples, since all the slope subalgebras $\mc{B}_{\mbf{m}}$ is isomorphic to $\bigotimes_{h=1}^{g}U_{q}(\hat{\mf{gl}}_{l_h})$ generated by $\varphi_{[i,j)}$, $P^{\mbf{m}}_{\pm[i,j)}$, an easy computation check for the universal $R$-matrix can show that all the wall subalgebra $U_{q}(\mf{g}_w)$ is either isomorphic to $U_{q}(\mf{sl}_2)$ or $U_{q}(\hat{\mf{gl}}_1)$ generated by $\varphi_{[i,j)}$ and $P^{\mbf{m}}_{[i,i+1)_h}$ with $[i,i+1)_h$ living on the wall $w$ or $p_{k}$ the Heisenberg operators in $\bigotimes_{h=1}^{g}U_{q}(\hat{\mf{gl}}_{l_h})$ determined by the formulax of the Proposition \ref{implicit-heisenberg-generators}.

\subsection{Geometric action on $K_{T}(M(\mbf{w}))$}
We first fix the notation for the Nakajima quiver varieties.

Given a quiver $Q=(I,E)$, consider the following quiver representation space
\begin{align}
\text{Rep}(\mbf{v},\mbf{w}):=\bigoplus_{h\in E}\text{Hom}(V_{i(h)},V_{o(h)})\oplus\bigoplus_{i\in I}\text{Hom}(V_i,W_i)
\end{align}

here $\mbf{v}=(\text{dim}(V_1),\cdots,\text{dim}(V_I))$ is the dimension vector for the vector spaces at the vertex, $\mbf{w}=(\text{dim}(W_1),\cdots,\text{dim}(W_I))$ is the dimension vector for the framing vector spaces.

Denote $G_{\mbf{v}}:=\prod_{i\in I}GL(V_i)$ and $G_{\mbf{w}}:=\prod_{i\in I}GL(W_i)$. There is a natural action of $G_{\mbf{v}}$ and $G_{\mbf{w}}$ on $\text{Rep}(\mbf{v},\mbf{w})$, and thus a natural Hamiltonian action on $T^*\text{Rep}(\mbf{v},\mbf{w})$ with respect to the standard symplectic form $\omega$, now we have the moment map
\begin{align}
\mu:T^*\text{Rep}(\mbf{v},\mbf{w})\rightarrow\mf{g}_{\mbf{v}}^*
\end{align}

\begin{align}
\mu(X_e,Y_e,A_i,B_i)=\sum_{e}X_{i(e)}Y_{i(e)}-Y_{o(e)}X_{o(e)}+A_iB_i
\end{align}

One can define \textbf{Nakajima variety}:
\begin{align}\label{Quiver-variety}
M_{\theta,0}(\mbf{v},\mbf{w}):=\mu^{-1}(0)//_{\theta}G_{\mbf{v}}
\end{align}

with all the $\theta$-stable $G_{\mbf{v}}$-orbits in $\mu^{-1}(0)$. Here we choose the stability condition $\theta=(1,\cdots,1)$.

We will abbreviate $M(\mbf{v},\mbf{w})$ as the Nakajima quiver variety defined in \ref{Quiver-variety}.

In this paper we only consider the affine type $A$ quiver, i.e. the edges are given by $e=(i,i+1)$ with $i=i\text{mod}(n)$.

The action of $\mbb{C}^*_{q}\times\mbb{C}^*_t\times G_{\mbf{w}}$ on $M(\mbf{v},\mbf{w})$ is given by:
\begin{align}
(q,t,U_i)\cdot(X_e,Y_e,A_i,B_i)_{e\in E,i\in I}=(qtX_e,qt^{-1}Y_e,qA_iU_i,qU_i^{-1}B_i)
\end{align}

The fixed points set of the torus action $\mbb{C}^*_q\times\mbb{C}^*_t$ on $M_{\theta,0}(\mbf{v},\mbf{w})$ is indexed by the $|\mbf{w}|$-partitions $\bm{\lambda}=(\lambda_1,\cdots,\lambda_{\mbf{w}})$. For each box $\square\in(\lambda_1,\cdots.\lambda_{\mbf{w}})$, we define the following character function:
\begin{align}
\chi_{\square}=u_iq^{x+y+1}t^{x-y},\qquad\square\in\lambda_i
\end{align}

In our context, we will denote $M(\mbf{v},\mbf{w})$ as $M_{\theta,0}(\mbf{v},\mbf{w})$. We always denote $K_{G_{\mbf{w}}\times \mbb{C}_{q}^*\times\mbb{C}_t^*}(M(\mbf{v},\mbf{w}))$ as the $T_{\mbf{w}}\times \mbb{C}_{q}^*\times\mbb{C}_t^*$-localised equivariant $K$-theory of $M(\mbf{v},\mbf{w})$. It has the basis denoted by the multipartition $|\bm{\lambda}\rangle$.

 If we choose the cocharacter $\sigma:\mbb{C}^*\rightarrow T_{\mbf{w}}$ such that $\mbf{w}=u_1\mbf{w}_1+\cdots+u_k\mbf{w}_k$, the fixed point component can be written as:
\begin{align}
M(\mbf{v},\mbf{w})^{\sigma}=\bigsqcup_{\mbf{v}_1+\cdots+\mbf{v}_k=\mbf{v}}M(\mbf{v}_1,\mbf{w}_1)\times\cdots\times M(\mbf{v}_k,\mbf{w}_k)
\end{align}

We denote:
\begin{align}
K_{T}(M(\mbf{w})):=\bigoplus_{\mbf{w}}K_{\mbb{C}_q^*\times\mbb{C}_t^*\times G_{\mbf{w}}}(M(\mbf{v},\mbf{w}))_{loc}
\end{align}

Thus we can choose the fixed point basis $|\bm{\lambda}\rangle$ of the corresponding partition $\bm{\lambda}$ to span the vector space $K_{T}(M(\mbf{w}))$.

Here we briefly review the geometric action of $U_{q,t}(\hat{\hat{\mf{sl}}}_{n})$ on $K_{T}(M(\mbf{w}))$. The construction is based on Nakajima's simple correspondence. 

Consider a pair of vectors $(\mbf{v}_+,\mbf{v}_-)$ such that $\mbf{v}_{+}=\mbf{v}_{-}+\mbf{e}_{i}$. There is a simple correspondence
\begin{align}
Z(\mbf{v}_+,\mbf{v}_{-},\mbf{w})\hookrightarrow M(\mbf{v}_+,\mbf{w})\times M(\mbf{v}_-,\mbf{w})
\end{align}
parametrises pairs of quadruples $(X^{\pm},Y^{\pm},A^{\pm},B^{\pm})$ that respect a fixed collection of quotients $(V^+\rightarrow V^-)$ of codimension $\delta^{i}_{j}$ with only the semistable and zeros part for the moment map $\mu$ for each $M(\mbf{v}_+,\mbf{w})$. And the variety $Z(\mbf{v}_+,\mbf{v}_{-},\mbf{w})$ is smooth with a tautological line bundle:
\begin{align}
\mc{L}|_{V^+\rightarrow V^-}=\text{Ker}(V_{i}^+\rightarrow V_{i}^{-})
\end{align}
and the natural projection maps:
\begin{equation}
\begin{tikzcd}
&Z(\mbf{v}_+,\mbf{v}_{-},\mbf{w})\arrow[ld,"\pi_{+}"]\arrow[rd,"\pi_{-}"]&\\
M(\mbf{v}_+,\mbf{w})&&M(\mbf{v}_-,\mbf{w})
\end{tikzcd}
\end{equation}

Using this we could consturct the operator:
\begin{align}
e_{i,d}^{\pm}:K_{T}(M(\mbf{v}^{\mp},\mbf{w}))\rightarrow K_{T}(M(\mbf{v}^{\pm},\mbf{w})),\qquad e_{i,d}^{\pm}(\alpha)=\pi_{\pm*}(\text{Cone}(d\pi_{\pm})\mc{L}^d\pi_{\mp}^{\pm}(\alpha))
\end{align}

and take all $\mbf{v}$ we have the operator $e_{i,d}^{\pm}:K_{T}(M(\mbf{w}))\rightarrow K_{T}(M(\mbf{w}))$. Also we have the action of $\varphi^{\pm}_{i,d}$ given by the multiplication of the tautological class, which means that:
\begin{align}
\varphi_{i}^{\pm}(z)=\overline{\frac{\zeta{(\frac{z}{X})}}{\zeta(\frac{X}{z})}}\prod^{u_{j}\equiv i}_{1\leq j\leq\mbf{w}}\frac{[\frac{u_j}{qz}]}{[\frac{z}{qu_{j}}]}
\end{align}

In particular, the element $\varphi_{i,0}$ acts on $K_{T}(M(\mbf{v},\mbf{w}))$ as $q^{\alpha_{i}^T(\mbf{w}-C\mbf{v})}$.

The above construction gives the following well-known result \cite{N01}\cite{N15}:
\begin{thm}
For all $\mbf{w}\in\mbb{N}^r$, the operator $e_{i,d}^{\pm}$ and $\varphi^{\pm}_{i,d}$ give rise to an action of $U_{q,t}(\hat{\hat{\mf{sl}}}_{n})$ on $K_{T}(M(\mbf{w}))$.
\end{thm}

In terms of the shuffle algebra, we can give the explicit formula of the action of the quantum toroial algebra $U_{q,t}(\hat{\hat{\mf{sl}}}_{n})$ on $K_{T}(M(\mbf{w}))$:

Given $F\in\mc{S}_{\mbf{k}}^+$, we have that
\begin{align}\label{shuffle-formula-1}
\langle\bm{\lambda}|F|\bm{\mu}\rangle=F(\chi_{\bm{\lambda}\backslash\bm{\mu}})\prod_{\blacksquare\in\bm{\lambda}\backslash\bm{\mu}}[\prod_{\square\in\bm{\mu}}\zeta(\frac{\chi_{\blacksquare}}{\chi_{\square}})\prod_{i=1}^{\mbf{w}}[\frac{u_i}{q\chi_{\blacksquare}}]]
\end{align}

Similarly, for $G\in\mc{S}_{-\mbf{k}}^{-}$, we have
\begin{align}\label{shuffle-formula-2}
\langle\bm{\mu}|G|\bm{\lambda}\rangle=G(\chi_{\bm{\lambda}\backslash\bm{\mu}})\prod_{\blacksquare\in\bm{\lambda}\backslash\bm{\mu}}[\prod_{\square\in\bm{\lambda}}\zeta(\frac{\chi_{\square}}{\chi_{\blacksquare}})\prod_{i=1}^{\mbf{w}}[\frac{\chi_{\blacksquare}}{qu_i}]]^{-1}
\end{align}

Here $F(\chi_{\bm{\lambda}\backslash\bm{\mu}})$ and $G(\chi_{\bm{\lambda}\backslash\bm{\mu}})$ are the rational symmetric function evaluation at the box $\chi_{\square}$ with $\square\in\bm{\lambda}\backslash\bm{\mu}$.

\section{\textbf{Maulik-Okounkov quantum affine algebra}}

\subsection{Stable envelopes}

We first review the definition of the $K$-theoretic stable envelopes for the quiver varieties.

Given $M(\mbf{v},\mbf{w})$ a Nakajima quiver variety and $G_{\mbf{v}}\times A_{E}$ acting on $M(\mbf{v},\mbf{w})$. Given a subtorus $T\subset G_{\mbf{v}}\times A_{E}$ in the kernel of $q$. By definition, the $K$-theoretic stable envelope is a $K$-theory class
\begin{align}
\text{Stab}_{\mc{C}}^{s}\subset K_{G}(X\times X^T)
\end{align}

such that it induces the morphism
\begin{align}
\text{Stab}_{\mc{C},s}:K_{G}(X^T)\rightarrow K_{G}(X)
\end{align}

such that if we write $X^T=\sqcup_{\alpha}F_{\alpha}$ into components:
\begin{itemize}
	\item The diagonal term is given by the structure sheaf of the attractive space:
	\begin{align}
	\text{Stab}_{\mc{C},s}|_{F_{\alpha}\times F_{\alpha}}=(-1)^{\text{rk }T_{>0}^{1/2}}(\frac{\text{det}(\mc{N}_{-})}{\text{det}T_{\neq0}^{1/2}})^{1/2}\otimes\mc{O}_{\text{Attr}}|_{F_{\alpha}\times F_{\alpha}}
	\end{align}

	\item The $T$-degree of the stable envelope has the bounding condition for $F_{\beta}\leq F_{\alpha}$:
	\begin{align}
	\text{deg}_{T}\text{Stab}_{\mc{C},s}|_{F_{\beta}\times F_{\alpha}}+\text{deg}_{T}s|_{F_{\alpha}}\subset\text{deg}_{T}\text{Stab}_{\mc{C},s}|_{F_{\beta}\times F_{\beta}}+\text{deg}_{T}s|_{F_{\beta}}
	\end{align}

	We require that for $F_{\beta}<F_{\alpha}$, the inclusion $\subset$ is strict.
\end{itemize}

The uniqueness and existence of the $K$-theoretic stable envelope was given in \cite{AO21} and \cite{O21}. In \cite{AO21}, the consturction is given by the abelinization of the quiver varieties. In \cite{O21}, the construction is given by the stratification of the complement of the attracting set, which is much more general.

The stable envelope has the factorisation property called the triangle lemma \cite{O15}. Given a subtorus $T'\subset T$ with the corresponding chamber $\mc{C}_{T'},\mc{C}_{T}$, we have the following diagram commute:
\begin{equation}\label{triangle-lemma}
\begin{tikzcd}
K_{G}(X^T)\arrow[rr,"\text{Stab}_{\mc{C}_T,s}"]\arrow[dr,"\text{Stab}_{\mc{C}_{T/T'},s}"]&&K_{G}(X)\\
&K_{G}(X^{T'})\arrow[ur,"\text{Stab}_{\mc{C}_{T'},s}"]&
\end{tikzcd}
\end{equation}

\subsection{Maulik-Okounkov quantum algebra and wall subalgebras}

Let us focus on the case of the quiver varieties $M(\mbf{v},\mbf{w})$. Choose the framing torus $\sigma:\mbb{C}^*\rightarrow A_{\mbf{w}}\subset G_{\mbf{w}}$ such that:
\begin{align}
\mbf{w}=a_1\mbf{w}_1+\cdots+a_k\mbf{w}_k
\end{align}

In this case the fixed point is given by:
\begin{align}
M(\mbf{v},\mbf{w})^{\sigma}=\bigsqcup_{\mbf{v}_1+\cdots+\mbf{v}_k=\mbf{v}}M(\mbf{v}_1,\mbf{w}_1)\times\cdots\times M(\mbf{v}_k,\mbf{w}_k)
\end{align}

Denote $K(\mbf{w}):=\bigoplus_{\mbf{v}}K_{G_{\mbf{w}}}(M(\mbf{v},\mbf{w}))$, it is easy to see that the stable envelope $\text{Stab}_{s}$ gives the map:
\begin{align}
\text{Stab}_{\mc{C},s}:K(\mbf{w}_1)\otimes\cdots\otimes K(\mbf{w}_k)\rightarrow K(\mbf{w}_1+\cdots+\mbf{w}_k)
\end{align}

Using the $K$-theoretic stable envelope, we can define the geometric $R$-matrix as:
\begin{align}
\mc{R}^{s}_{\mc{C}}:=\text{Stab}_{-\mc{C},s}^{-1}\circ\text{Stab}_{\mc{C},s}:K(\mbf{w}_1)\otimes\cdots\otimes K(\mbf{w}_k)\rightarrow K(\mbf{w}_1)\otimes\cdots\otimes K(\mbf{w}_k)
\end{align}

Written in the component of the weight subspaces, the geometric $R$-matrix can be written as:
\begin{equation}
\begin{aligned}
&\mc{R}^{s}_{\mc{C}}:=\text{Stab}_{-\mc{C},s}^{-1}\circ\text{Stab}_{\mc{C},s}:\bigoplus_{\mbf{v}_1+\cdots+\mbf{v}_k=\mbf{v}}K(\mbf{v}_1,\mbf{w}_1)\otimes\cdots\otimes K(\mbf{v}_k,\mbf{w}_k)\\
&\rightarrow \bigoplus_{\mbf{v}_1+\cdots+\mbf{v}_k=\mbf{v}}K(\mbf{v}_1,\mbf{w}_1)\otimes\cdots\otimes K(\mbf{v}_k,\mbf{w}_k)
\end{aligned}
\end{equation}

From the triangle diagram \ref{triangle-lemma} of the stable envelope, we can further factorise the geometric $R$-matrix into the smaller parts:
\begin{align}\label{abstract-decomposition}
\mc{R}^s_{\mc{C}}=\prod_{1\leq i<j\leq k}\mc{R}^s_{\mc{C}_{ij}}(\frac{a_i}{a_j}),\qquad \mc{R}^s_{\mc{C}_{ij}}(\frac{a_i}{a_j}):K(\mbf{w}_i)\otimes K(\mbf{w}_j)\rightarrow K(\mbf{w}_i)\otimes K(\mbf{w}_j)
\end{align}

Each $\mc{R}^{s}_{\mc{C}_{ij}}(u)$ satisfies the trigonometric Yang-Baxter equation with the spectral parametres.

In the language of the integrable model, we denote $V_{i}(u_i)$ as the modules of type $K(\mbf{w})$ defined above with the spectral parametre $u_i$. The formula \ref{abstract-decomposition} means that:
\begin{align}
\mc{R}^s_{\bigotimes^{\leftarrow}_{i\in I}V_i(u_i),\bigotimes^{\leftarrow}_{j\in J}V_j(u_j)}=\prod_{i\in I}^{\rightarrow}\prod_{j\in J}^{\leftarrow}\mc{R}^{s}_{V_i,V_j}(\frac{u_i}{u_j}):
\end{align}

We can also consider the dual module $V_i^*(u_i)$ as the module isomorphic to $V_i(u_i)$ as graded vector space, with the $R$-matrices defined as:
\begin{equation}
\begin{aligned}
&\mc{R}^s_{V_1^*,V_2}=((\mc{R}^s_{V_1,V_2})^{-1})^{*_1}\\
&\mc{R}^s_{V_1,V_2^*}=((\mc{R}^s_{V_1,V_2})^{-1})^{*_2}\\
&\mc{R}^s_{V_1^*,V_2^*}=(\mc{R}^s_{V_1,V_2})^{*_{12}}
\end{aligned}
\end{equation}
$*_{k}$ means transpose with respect to the $k$-th factor.

\begin{defn}
The Maulik-Okounkov quantum affine algebra $U_{q}^{MO}(\hat{\mf{g}}_{Q})$ is the subalgebra of $\prod_{\mbf{w}}\text{End}(K(\mbf{w}))$ generated by the matrix coefficients of $\mc{R}^{s}_{\mc{C}}$.
\end{defn}

In other words, given an auxillary space $V_0=\bigotimes_{\mbf{w}}K(\mbf{w})$, $V=\bigotimes_{\mbf{w}'}K(\mbf{w}')$ and a finite rank operator
\begin{align}
m(a_0)\in\text{End}(V_0)(a_0)
\end{align}
Now the element of $U_{q}^{MO}(\hat{\mf{g}}_{Q})$ is generated by the following operators:
\begin{align}
\oint_{a_0=0,\infty}\frac{da_0}{2\pi ia_0}\text{Tr}_{V_0}((1\otimes m(a_0))\mc{R}^{s}_{V,V_0}(\frac{a}{a_0}))\in\text{End}(V(a))
\end{align}

The coproduct structure, antipode map and the counit map can be defined as follows:

For the coproduct $\Delta_{s}$ on $U_{q}^{MO}(\hat{\mf{g}}_{Q})$ is defined via the conjugation by $\text{Stab}_{\mc{C},s}$, i.e. for $a\in U_{q}^{MO}(\hat{\mf{g}}_{Q})$ as $a:K(\mbf{w})\rightarrow K(\mbf{w})$, $\Delta_{s}(a)$ is defined as:
\begin{equation}\label{coproduct-geometric}
\begin{tikzcd}
K(\mbf{w}_1)\otimes K(\mbf{w}_2)\arrow[r,"\text{Stab}_{\mc{C},s}"]&K(\mbf{w}_1+\mbf{w}_2)\arrow[r,"a"]&K(\mbf{w}_1+\mbf{w}_2)\arrow[r,"\text{Stab}_{\mc{C},s}^{-1}"]&K(\mbf{w}_1)\otimes K(\mbf{w}_2)
\end{tikzcd}
\end{equation}

The antipode map $S_{s}$ is defined as follows. We have the isomorphism of the graded vector space $V_i\cong V_i^*$, and this isomorphism induces the isomorphism:
\begin{align}
\prod_{i}\text{End}(V_i)\cong\prod_{i}\text{End}(V_i^*)
\end{align}

The antipode map $S_{s}:U_{q}^{MO}(\hat{\mf{g}}_{Q})\rightarrow U_{q}^{MO}(\hat{\mf{g}}_{Q})$ is given by:
\begin{align}
\oint_{a_0=0,\infty}\frac{da_0}{2\pi ia_0}\text{Tr}_{V_0}((1\otimes m(a_0))\mc{R}^{s}_{V,V_0}(\frac{a}{a_0}))\mapsto\oint_{a_0=0,\infty}\frac{da_0}{2\pi ia_0}\text{Tr}_{V_0}((1\otimes m(a_0))\mc{R}^{s}_{V^*,V_0}(\frac{a}{a_0}))
\end{align}

The projection of the module $V$ to the trivial module $\mbb{C}$ induce the counit map:
\begin{align}
\epsilon:U_{q}^{MO}(\hat{\mf{g}}_{Q})\rightarrow\mbb{C}
\end{align}

Since $M(0,\mbf{w})$ is just a point, we denote the vector in $K_{\mbf{0},\mbf{w}}$ as $v_{\varnothing}$, and we call it the \textbf{vacuum vector}.  We define the evaluation map:
\begin{align}\label{evaluation-module}
\text{ev}:U_{q}^{MO}(\hat{\mf{g}}_{Q})\rightarrow\prod_{\mbf{w}}K(\mbf{w}),\qquad F\mapsto Fv_{\varnothing}
\end{align}

\begin{defn}
We define the \textbf{positive half of the Maulik-Okounkov quantum affine algebra} $U_{q}^{MO,+}(\hat{\mf{g}}_{Q})$ as the quotient by the kernel of the evaluation map:
\begin{align} 
U_{q}^{MO,+}(\hat{\mf{g}}_{Q}):=U_{q}^{MO}(\hat{\mf{g}}_{Q})/\text{Ker}(ev)
\end{align}
\end{defn}

\subsubsection{\textbf{Wall subalgebra}}

It is known that the $K$-theoretic stable envelope $\text{Stab}_{s}$ is locally constant on $s\in\text{Pic}(X)\otimes\mbb{Q}$. It changes as $s$ crosses certain ratioanl hyperplanes:
\begin{prop}
The $K$-theoretical stable envelope $\text{Stab}_{\mc{C},s}$ is locally constant on $s$ if and only if $s$ crosses the following hyperplanes.
\begin{align}
w=\{s\in\text{Pic}(X)\otimes\mbb{R}|(s,\alpha)+n=0,\forall\alpha\in\text{Pic}(X)\}\subset\text{Pic}(X)\otimes\mbb{R}
\end{align}
\end{prop}

\begin{proof}
We choose the subtorus $\sigma:\mbb{C}^*\rightarrow T_{\mbf{w}}$ such that:
\begin{align}
\mbf{w}=u_1\mbf{w}_1+\cdots+u_k\mbf{w}_k
\end{align}
such that:
\begin{align}
M(\mbf{v},\mbf{w})^{A}=\bigsqcup_{\mbf{v}_1+\cdots+\mbf{v}_k=\mbf{v}}M(\mbf{v}_1,\mbf{w}_1)\times\cdots\times M(\mbf{v}_k,\mbf{w}_k)
\end{align}

So now given the tautological line bundle $\mc{L}_i$ over $M(\mbf{v},\mbf{w})$, the restriction of $\mc{L}_{i}$ to $M(\mbf{v}_1,\mbf{w}_1)\times\cdots\times M(\mbf{v}_k,\mbf{w}_k)$ has the $A$-degree is given by:
\begin{align}
\mc{L}_i|_{M(\mbf{v}_1,\mbf{w}_1)\times\cdots\times M(\mbf{v}_k,\mbf{w}_k)}=u_1^{v_1^{(i)}}\cdots u_k^{v_k^{(i)}}(\cdots)
\end{align}

Thus the Newton polytope of the degree $A$ is a point:
\begin{align}
\text{deg}_{A}(\mc{L}_i|_{M(\mbf{v}_1,\mbf{w}_1)\times\cdots\times M(\mbf{v}_k,\mbf{w}_k)})=(v_1^{(i)},\cdots,v_k^{(i)})
\end{align}

Thus given the fractional line bundle $\mc{L}=\mc{L}_1^{m_1}\cdots\mc{L}_{n}^{m_n}$, we have that:
\begin{align}
\text{deg}_{A}(\mc{L}|_{M(\mbf{v}_1,\mbf{w}_1)\times\cdots\times M(\mbf{v}_k,\mbf{w}_k)})=(\mbf{m}\cdot\mbf{v}_1,\cdots,\mbf{m}\cdot\mbf{v}_k)\in\mbb{R}^k
\end{align}

Thus $\text{deg}_{A}$ gives a linear map $\text{deg}_{A}:\text{Pic}(M(\mbf{v},\mbf{w}))\otimes\mbb{R}\rightarrow\mbb{R}^k$.

Now given $F_{\beta}=M(\mbf{v}_1,\mbf{w}_1)\times\cdots\times M(\mbf{v}_k,\mbf{w}_k)$, $F_{\alpha}=M(\mbf{v}_1+\mbf{a}_1,\mbf{w}_1)\times\cdots\times M(\mbf{v}_k+\mbf{a}_k,\mbf{w}_k)$ such that $\sum_i\mbf{a}_i=0$. By the degree bounding condition:
\begin{align}
\text{deg}_{A}\text{Stab}_{\mc{C},\mc{L}}|_{F_{\beta}\times F_{\alpha}}+\text{deg}_{A}\mc{L}|_{F_{\alpha}}\subset\text{deg}_{A}\text{Stab}_{\mc{C},\mc{L}}|_{F_{\beta}\times F_{\beta}}+\text{deg}_{A}\mc{L}|_{F_{\beta}}
\end{align}

Here we describe the Newton polytope:
\begin{align}
\text{deg}_{A}\text{Stab}_{\mc{C},\mc{L}}|_{F_{\beta}\times F_{\beta}}+\text{deg}_{A}\mc{L}|_{F_{\beta}}-\text{deg}_{A}\mc{L}|_{F_{\alpha}}
\end{align}

Let us describe:
\begin{align}
\text{Stab}_{\mc{C},\mc{L}}|_{F_{\beta}\times F_{\beta}}\otimes\mc{L}|_{F_{\beta}}\otimes\mc{L}^{-1}|_{F_{\alpha}}=\sum_{\bm{\mu}\in\mbb{Z}^k}a_{\bm{\mu}}u_1^{\mu_1+\mbf{m}\cdot\mbf{a}_1}\cdots u_k^{\mu_k+\mbf{m}\cdot\mbf{a}_k},\qquad\sum_{i}\mbf{a}_i=0
\end{align}
This equations stands for the Newton Polytope:
\begin{equation}
\begin{aligned}
&P_{A}(\text{Stab}_{\mc{C},\mc{L}}|_{F_{\beta}\times F_{\beta}}\otimes\mc{L}|_{F_{\beta}}\otimes\mc{L}^{-1}|_{F_{\alpha}})\\
=&\mbb{Z}^k\cap[\text{Convex hull of}(\bm{\mu}\in\mbb{Z}^k|a_{\bm{\mu}}\neq0)+\text{translation as }(\mbf{m}\cdot\mbf{a}_1,\cdots,\mbf{m}\cdot\mbf{a}_k)]
\end{aligned}
\end{equation}

We say that the Newton polytope $P_A$ is \textbf{special} if the boundary of the Newton polytope intersects with the integer points For now we choose some special $\mbf{m}$ such that $P_{A}$ is special. It is clear that in this case $P_{A}$ is at the "critical" level such that it contains an integer point or not. In this way if we fix the Newton polytope $P_{A}$ such that it contains the fixed integer points at the boundary, we can move $\mbf{m}$ in certain directions such that the boundary integer points are still on the boundary. 

Also we suppose that $\partial P_{A}$ consists of the union of several linearly independent different hyperplanes $H_1,\cdots, H_{l}$ in $\mbb{R}^k$ up to translation, and each hyperplane is determined by a vector $\mbf{n}_i\in\mbb{R}^k$. In order to satisfy the above conditions, this means that the fractional line bundle $\mbf{m}$ should be chosen on the intersection of the hyperplanes $H_1\cap\cdots\cap H_s$ containing the integer points. which means that:
\begin{align}
((\mbf{m}'-\mbf{m})\cdot\mbf{a}_1,\cdots,(\mbf{m}'-\mbf{m})\cdot\mbf{a}_k)\in H_1\cap\cdots\cap H_s
\end{align}

In conclusion, we can see that the critical fractional line bundles $\mbf{m}$ lives in the set of hyperplanes such that:
\begin{align}
\{\mbf{m}\in\mbb{R}^n|(\mbf{m}\cdot\mbf{a}_1,\cdots,\mbf{m}\cdot\mbf{a}_k)\in\mbb{Z}^k,\sum_{i}\mbf{a}_i=0\}
\end{align}
for some $\mbf{a}_1,\cdots,\mbf{a}_k$. The choice of $\mbf{a}_1$ depends on the boundary hyperplane of the Newton polytope $\partial P_{A}$, thus the choice of $(\mbf{a}_1,\cdots,\mbf{a}_k)$ is finite.

Now if we choose $A\subset T_{\mbf{w}}$ such that $\mbf{w}=u_1\mbf{w}_1+u_2\mbf{w}_2$, we can have that the critical fractional line bundle should be in:
\begin{align}
\{\mbf{m}\in\mbb{R}^n|(\mbf{m}\cdot\mbf{a},-\mbf{m}\cdot\mbf{a})\in\mbb{Z}^2\}\cong\{\mbf{m}\in\mbb{R}^n|\mbf{m}\cdot\mbf{a}\in\mbb{Z}\}
\end{align}

This finishes the proof of the proposition.
\end{proof}

Now for each quiver variety $M(\mbf{v},\mbf{w})$, we can associate the wall set $w(\mbf{v},\mbf{w})$, it is clear that via the identification $\text{Pic}(M(\mbf{v},\mbf{w}))\otimes\mbb{R}=\mbb{R}^n$ for different $\mbf{v}$, we can define the \textbf{wall set} of $M(\mbf{w}):=\sqcup_{\mbf{v}}M(\mbf{v},\mbf{w})$ as $w(\mbf{w}):=\sqcup w(\mbf{v},\mbf{w})$.

Now fix the slope $\mbf{m}$ and the cocharacter $\sigma:\mbb{C}^*\rightarrow A_{\mbf{w}}$. We choose an ample line bundle $\mc{L}\in\text{Pic}(X)$ with $X=M(\mbf{v},\mbf{w}_1+\mbf{w}_2)$ and a suitable small number $\epsilon$ such that $\mbf{m}$ and $\mbf{m}+\epsilon\mc{L}$ are separated by just one wall $w$, we define the \textbf{wall $R$-matrices} as:
\begin{align}
R_{w}^{\pm}:=\text{Stab}_{\sigma,\mbf{m}+\epsilon\mc{L}}^{-1}\circ\text{Stab}_{\sigma,\mbf{m}}:\bigoplus_{\mbf{v}_1+\mbf{v}_2=\mbf{v}}K(\mbf{v}_1,\mbf{w}_1)\otimes K(\mbf{v}_2,\mbf{w}_2)\rightarrow\bigoplus_{\mbf{v}_1+\mbf{v}_2=\mbf{v}}K(\mbf{v}_1,\mbf{w}_1)\otimes K(\mbf{v}_2,\mbf{w}_2)
\end{align}

It is an integral $K$-theory class in $K_{T}(X^A\times X^A)$. Note that the choice of $\epsilon$ depends on $M(\mbf{v},\mbf{w}_1+\mbf{w}_2)$ just to make sure that there is only one wall between $\mbf{m}$ and $\mbf{m}+\epsilon\mc{L}$ corresponding to $w$. By definition it is easy to see that $R_{w}^{+}$ is upper-triangular, and $R_{w}^{-}$ is lower triangular.

Given a finite-rank operator $m\in\text{End}(V_0)$. We define the \textbf{positive half of the wall subalgebra} $U_{q}^{MO,+}(\mf{g}_{w})$ as the algebra generated by the operators:
\begin{align}
\text{Tr}_{V_0}((1\otimes m)(R^{-}_{w})_{V,V_0}^{-1}|_{a_0=1})\in\text{End}(V)
\end{align}

Similarly, the \textbf{negative half of the wall subalgebra} $U_{q}^{MO,-}(\mf{g}_{w})$ is defined as the algebra generated by the operators:
\begin{align}
\text{Tr}_{V_0}((1\otimes m)(R^{+}_{w})_{V,V_0}|_{a_0=1})\in\text{End}(V)
\end{align}

Since $R_{w}^{\pm}$ is a strictly upper(lower) triangular matrix, we can see that the algebra $U_{q}^{MO,\pm}(\mf{g}_{w})$ is generated by the operators $m$ such that:
\begin{align}
m:K(\mbf{v},\mbf{w})\rightarrow K(\mbf{v}\pm\mbf{n},\mbf{w}),\qquad\mbf{n}\in\mbb{N}^I
\end{align}

Moreover, it can be shown that the operators $q^{\Omega}R_{w}^{\pm}$ satisfy the Yang-Baxter equation. In this way, similarly we can define the non-negative(non-positive) half of the wall subalgebra $U_{q}^{MO,\geq(\leq)}(\mf{g}_{w})$ and also the whole wall subalgebra $U_{q}^{MO}(\mf{g}_{w})$. Obviously we have $U_{q}^{MO,\pm}(\mf{g}_{w})\subset U_{q}^{MO,\geq(\leq)}(\mf{g}_{w})$.

One could define the graded pieces of $U_{q}^{MO}(\mf{g}_w)$ as:
\begin{align}
U_{q}^{MO,\pm}(\mf{g}_w)=\bigoplus_{\mbf{n}\in\mbb{N}^{I}}U_{q}^{MO,\pm}(\mf{g}_{w})_{\pm\mbf{n}}
\end{align}
with $a\in U_{q}^{MO,\pm}(\mf{g}_{w})_{\pm\mbf{n}}$ the elements such that $a:K(\mbf{v},\mbf{w})\rightarrow K(\mbf{v}\pm\mbf{n},\mbf{w})$.

\begin{lem}
Each graded pieces $U_{q}^{MO,\pm}(\mf{g}_{w})_{\pm\mbf{n}}$ is finite-dimensional.
\end{lem}
\begin{proof}
The fact that $R_{w}^{\pm}$ is an integral $K$-theory class implies that each component of $R_w^{\pm}$ is written as $Au^{k}$ with the multiplication $A$ given as the Laurent polynomial of $q$ and $t$, thus each graded pieces $U_{q}^{MO,\pm}(\mf{g}_{w})_{\pm\mbf{n}}$ is finite-dimensional.
\end{proof}

The following proposition was proved in \cite{OS22}:
\begin{prop}
\begin{align}
R_{w}^{\mp}=(R_{w}^{\pm})_{21}|_{u=u^{-1}}
\end{align}
\end{prop}
This implies that $U_{q}^{MO,\geq}(\mf{g}_{w})$
is isomorphic to $U_{q}^{MO,\leq}(\mf{g}_{w})$
as graded vector spaces.

Usually we take $R_{w}^{\pm}$ as the universal $R$-matrix for $U_{q}^{MO}(\mf{g}_{w})$ without spectral parametres with respect to the coproduct $\Delta_{w}$ and $\Delta_{w}^{op}$. Now if we write:
\begin{align}
R_{w}^{+}=\sum_{i}a_i\otimes b_i\in U_{q}^{MO}(\mf{g}_{w})\hat{\otimes} U_{q}^{MO}(\mf{g}_{w})
\end{align}

since $R_{w}^{-}(u)=(R_{w}^{+})_{21}(u^{-1})$, this implies that:
\begin{align}
R_{w}^{-}=\sum_{i}b_i\otimes a_i\in U_{q}^{MO}(\mf{g}_{w})\hat{\otimes} U_{q}^{MO}(\mf{g}_{w})
\end{align}

This implies that one could realize $U_{q}^{MO}(\mf{g}_w)$ as the bialgebra pairing between $U_{q}^{MO,\geq}(\mf{g}_{w})$ and $U_{q}^{MO,\leq}(\mf{g}_{w})$ such that:
\begin{align}
\langle a_i,b_j\rangle=\delta_{ij}
\end{align}

In this way we could define the \textbf{wall subalgebra} $U_{q}^{MO}(\mf{g}_{w})$ as:
\begin{align}
U_{q}^{MO}(\mf{g}_{w}):=U_{q}^{MO,\geq}(\mf{g}_{w})\otimes U_{q}^{MO,\leq}(\mf{g}_{w})/(q^{\Omega}\otimes q^{-\Omega}=1)
\end{align}

Since $R_{w}^{MO}$ is an integral $K$-theory class, this is an algebra over $\mbb{Z}[q^{\pm1}]$. It has the universal $R$-matrix $R^{univ}_{w}$ which can be written as $\sum_{i}F_{i}\otimes G_{i}$ with $\{F_{i}\}$ and $\{G_{i}\}$ the dual bases, and the evaluation in $\text{End}(K(\mbf{w}_1)\otimes K(\mbf{w}_2))$ is equal to the MO wall $R$-matrix $R_{w}^{MO}$.

Throughout the paper we will assume that $U_{q}^{MO}(\mf{g}_{w})$ is localised over $\mbb{Q}$ with $q$ being valued outside of the roots of unity.

\subsection{Factorisation of geometric $R$-matrices}

Fix the stable envelope $\text{Stab}_{\sigma,\mbf{m}}$ and $\text{Stab}_{\sigma,\infty}$, we can have the following factorisation of $\text{Stab}_{\pm,\mbf{m}}$ near $u=0,\infty$:
\begin{equation}
\begin{aligned}
\text{Stab}_{\sigma,\mbf{m}}=&\text{Stab}_{\sigma,-\infty}\cdots\text{Stab}_{\sigma,\mbf{m}_{-2}}\text{Stab}_{\sigma,\mbf{m}_{-2}}^{-1}\text{Stab}_{\sigma,\mbf{m}_{-1}}\text{Stab}_{\sigma,\mbf{m}_{-1}}^{-1}\text{Stab}_{\sigma,\mbf{m}}\\
=&\text{Stab}_{\sigma,-\infty}\cdots R_{\mbf{m}_{-2},\mbf{m}_{-1}}^+R_{\mbf{m}_{-1},\mbf{m}}^+
\end{aligned}
\end{equation}

\begin{equation}
\begin{aligned}
\text{Stab}_{-\sigma,\mbf{m}}=&\text{Stab}_{-\sigma,\infty}\cdots\text{Stab}_{-\sigma,\mbf{m}_2}\text{Stab}_{-\sigma,\mbf{m}_2}^{-1}\text{Stab}_{-\sigma,\mbf{m}_1}\text{Stab}_{-\sigma,\mbf{m}_1}^{-1}\text{Stab}_{-\sigma,\mbf{m}}\\
=&\text{Stab}_{-\sigma,\infty}\cdots R_{\mbf{m}_2,\mbf{m}_1}^{-}R_{\mbf{m}_1,\mbf{m}}^{-}
\end{aligned}
\end{equation}

Here $R_{\mbf{m}_1,\mbf{m}_2}^{\pm}=\text{Stab}_{\pm\sigma,\mbf{m}_1}^{-1}\text{Stab}_{\pm\sigma,\mbf{m}_2}$ is the wall $R$-matrix. For simplicity we always choose generic slope points $\mbf{m}_i$ such that there is only one wall between $\mbf{m}_1$ and $\mbf{m}_2$. In this case we use $R_{w}^{\pm}$ as $R_{\mbf{m}_1,\mbf{m}_2}^{\pm}$. Note that this notation does not mean that $R_{w}$ only depends on the wall $w$, but we still use the notation for simplicity.

This gives the factorisation of the geometric $R$-matrix:
\begin{align}\label{factorisation-geometry}
\mc{R}^{s}(u)=\prod_{i<0}^{\leftarrow}R_{w_i}^{-}R_{\infty}\prod_{i\geq0}^{\leftarrow}R_{w_i}^{+}
\end{align}

This has been proved in \cite{OS22} that this factorisation is well-defined in the topology of the Laurent formal power series in the spectral parametre $u$.

The factorisation \ref{factorisation-geometry} gives the generators of the positive half of the MO quantum affine algebra:
\begin{lem}
$U_{q}^{MO,+}(\hat{\mf{g}}_{Q})$ is generated by the positive half of the wall subalgebras $U_{q}^{MO,+}(\mf{g}_{w})$ for arbitrary walls $w$.
\end{lem}
\begin{proof}
Note that the elements of $U_{q}^{MO}(\hat{\mf{g}}_{Q})$ are of the form:
\begin{align}
\oint_{a_0=0,\infty}\frac{da_0}{2\pi ia_0}\text{Tr}_{V_0}((1\otimes m(a_0))\mc{R}^{s}_{V,V_0}(\frac{a}{a_0}))\in\text{End}(V(a))
\end{align}
Now via the mapping $a\mapsto av_{\varnothing}$ we have that:
\begin{equation}
\begin{aligned}
&\oint_{a_0=0,\infty}\frac{da_0}{2\pi ia_0}\text{Tr}_{V_0}((1\otimes m(a_0))\mc{R}^{s}_{V,V_0}(\frac{a}{a_0}))v_{\varnothing}\\
=&\oint_{a_0=0,\infty}\frac{da_0}{2\pi ia_0}\text{Tr}_{V_0}((1\otimes m(a_0))\prod_{i<0}^{\leftarrow}R_{w_i}^{-}R_{\infty}\prod_{i\geq0}^{\leftarrow}R_{w_i}^{+})v_{\varnothing}\\
=&\oint_{a_0=0,\infty}\frac{da_0}{2\pi ia_0}\text{Tr}_{V_0}((1\otimes m(a_0))\prod_{i<0}^{\leftarrow}R_{w_i}^{-}v_{\varnothing}
\end{aligned}
\end{equation}

This result implies that the image of $\text{ev}$ in \ref{evaluation-module} contains only the coefficients of $R_{w_i}^{-}$, which implies that the lemma.
\end{proof}

\subsection{Cohomological stable envelope and Maulik-Okounkov Yangian}

In this subsection we review the original definition of the stable envelope in the case of the equivariant cohomology and the definition of the Maulik-Okounkov Yangian algebra in \cite{MO12}.

We still fix the notation as above, given the Nakajima quiver variety $M(\mbf{v},\mbf{w})$ with the group action $G_{\mbf{v}}\times A_{E}$ and $T\subset\text{Ker}(q)\subset G_{\mbf{v}}\times A_{E}$. The cohomological stable envelope is a cohomlogy class
\begin{align}
\text{Stab}_{\mc{C}}\subset H_{G}(X\times X^T)
\end{align}
which induces the morphism 
\begin{align}
\text{Stab}_{\mc{C}}:H_{G}(X^T)\rightarrow H_{G}(X)
\end{align}

such that:
\begin{itemize}
	\item The support of $\text{Stab}_{\mc{C}}$ is on $\text{Attr}^f$.
	\item
	\begin{align}
	\text{Stab}_{\mc{C}}|_{F_{\alpha}\times F_{\alpha}}=\pm e(N^{-}_{F_{\alpha}/X})
	\end{align}
	\item
	\begin{align}
	\text{deg}_{T}\text{Stab}_{\mc{C}}|_{F_{\alpha}\times F_{\beta}}<\frac{1}{2}\text{codim}_{\mbb{C}}(F_{\beta}),\qquad\forall F_{\beta}<F_{\alpha}
	\end{align}
\end{itemize}

The cohomological stable envelope also has the triangle lemma as \ref{triangle-lemma}

Similarly, we can choose $\sigma:\mbb{C}^*\rightarrow A_{\mbf{w}}$ in some chamber $\mc{C}$ such that:
\begin{align}
\mbf{w}=u_1\mbf{w}_1+\cdots+u_k\mbf{w}_k
\end{align}

We use the notation $H(\mbf{w}):=\bigoplus_{\mbf{v}}H_{G_{\mbf{w}}\times\mbb{C}_{\hbar}^*}^*(M(\mbf{v},\mbf{w}))$. This gives the stable envelope map of the form:
\begin{align}
\text{Stab}_{\pm\mc{C}}:H(\mbf{w}_1)\otimes\cdots\otimes H(\mbf{w}_k)\rightarrow H(\mbf{w}_1+\cdots+\mbf{w}_k)
\end{align}

Similarly we can define the cohomological geometric $R$-matrix:
\begin{align}
R_{\mc{C}}:=\text{Stab}_{-\mc{C}}^{-1}\circ\text{Stab}_{\mc{C}}
\end{align}
Similar as the case of the $K$-theoretical geometric $R$-matrix, the cohomological geometric $R$-matrix can also be factorised as:
\begin{align}
\mc{R}_{\mc{C}}=\prod_{1\leq i<j\leq k}\mc{R}_{\mc{C}_{ij}}(u_i-u_j),\qquad\mc{R}_{\mc{C}_{ij}}(u_i-u_j):H(\mbf{w}_i)\otimes H(\mbf{w}_j)\rightarrow H(\mbf{w}_i)\otimes H(\mbf{w}_j)
\end{align}
via the triangle lemma. Each $\mc{R}_{\mc{C}_{ij}}(u_i-u_j)$ satisfy the rational Yang-Baxter equation with the spectral parametres.

\subsubsection{Maulik-Okounkov Yangian algebra}
The Maulik-Okounkov Yangian algebra $Y_{\hbar}^{MO}(\mf{g}_{Q})$ is defined in the following way: We consider arbitrary finite-rank operator $m(u)\in\text{End}(F_0)[u]$ with $F_0=\bigotimes_{\mbf{w}}H(\mbf{w})$, the Maulik-Okounkov Yangian algebra $Y_{\hbar}(\mf{g}_{Q})$ is generated by the operators in $\prod_{\mbf{w}}\text{End}(H(\mbf{w}))$ of the form:
\begin{align}
E(m(u))=-\frac{1}{\hbar}\text{Res}_{u=\infty}\text{tr}_{F_0}((m(u)\otimes1)R_{F_0(u),W})\in\text{End}(W),\qquad W=\prod_{\mbf{w}}\text{End}(H(\mbf{w}))
\end{align}

i.e. $Y_{\hbar}(\mf{g}_{Q})$ is generated by the matrix coefficients of the geometric $R$-matrix $R_{F_0,W}$ under the expansion of $u\rightarrow\infty$.

It has also been proved in \cite{MO12} that the geometric $R$-matrix admits the expansion around $u=\infty$:
\begin{align}
R_{\mc{C}_{ij}}(u_i-u_j)=1+\frac{\hbar}{u_i-u_{j}}r_{\mc{C}_{ij}}+O((u_i-u_j)^{-2})
\end{align}

where the operator $r_{\mc{C}_{ij}}:H(\mbf{w}_i)\otimes H(\mbf{w}_j)\rightarrow H(\mbf{w}_i)\otimes H(\mbf{w}_j)$ satisfies the classical Yang-Baxter equation. In this way similarly, we could define the Maulik-Okounkov Lie algebra $\mf{g}_{Q}^{MO}$ as the Lie algebra generated by the operator $E(m_0)$ for $m_0\in\text{End}(F_0)$ constant in $u$. This is equivalent to say that $\mf{g}_{Q}^{MO}$ is generated by the matrix coefficients of the classical $r$-matrix $r_{\mc{C}}=\sum_{i<j}r_{\mc{C}_{ij}}$ by the following formula of brackets:
\begin{align}
[E(m_1),E(m_2)]=E((\text{tr}\otimes1)[r_{VV},m_1\otimes m_2]),\qquad m_1,m_2\in\text{End}(V)
\end{align}

The Maulik-Okounkov Lie algebra $\mf{g}_{Q}^{MO}$ admits the root decomposition:
\begin{align}
\mf{g}_{Q}^{MO}=\mf{h}\oplus\bigoplus_{\alpha\in\mbb{Z}^I}\mf{g}_{\alpha}^{MO}
\end{align}
with $\mf{g}_{\alpha}^{MO}$ spanned by the vector $e_{\alpha}$ such that:
\begin{align}
e_{\alpha}:H(\mbf{v},\mbf{w})\rightarrow H(\mbf{v}+\alpha,\mbf{w})
\end{align}

We denote the set $\Phi=\{\alpha\in\mbb{Z}^I|\mf{g}_{\alpha}\neq0\}$ as the set of roots. It has been known that if $\alpha\in\Phi$, $-\alpha\in\Phi$, and thus $\Phi=\Phi^+\sqcup-\Phi^{+}$. We define the nilpotent Maulik-Okounkov Lie algebra $\mf{n}_{Q}^{MO}$ as:
\begin{align}
\mf{n}_{Q}^{MO}:=\bigoplus_{\alpha\in\Phi^+}\mf{n}_{\alpha}^{MO}
\end{align}
It is a Lie subalgebra of $\mf{g}_{Q}^{MO}$ spanned by the root vectors of positive roots.

The theorem given by Botta and Davison \cite{BD23} states that the nilpotent Maulik-Okounkov Lie algebra is isomorphic to the corresponding BPS Lie algebra of quiver type $Q$:
\begin{thm}{(Botta-Davison,\cite{BD23})}
There is an isomorphism between the nilpotent Maulik-Okounkov Lie algebra and the BPS Lie algebra of quiver type $Q$:
\begin{align}
\mf{n}_{Q}^{MO}\cong\mf{n}_{Q}
\end{align}
\end{thm}

In the case of affine type $A$, the BPS Lie algebra $\mf{n}_{Q}$ is isomorphic to the positive half of $\hat{\mf{gl}}_n$.

This theorem will be the key to construct the isomorphism between the positive half of the quantum toroidal algebras and the positive half of the Maulik-Okounkov quantum affine algebras.

\subsection{Wall subalgebra and cohomologial limit}

The geometric $R$-matrix at slope $s$ can be written as:
\begin{align}
\mc{R}^s(u)=\prod^{\leftarrow}_{w<s}R_{w}^{-}(u)\cdot\mc{R}^{\infty}(u)\cdot\prod_{w\geq s}^{\leftarrow}R_{w}^+(u)
\end{align}

Here $R_{\infty}:=\text{Stab}_{-\mc{C},-\infty}^{-1}\circ\text{Stab}_{\mc{C},\infty}$ such that it is diagonal with the basis given by:
\begin{align}
R_{\infty}|_{F_{\alpha}\times F_{\alpha}}=i_*\frac{\wedge^*[N_{F_{\alpha}}^{-}]}{\wedge^*[N_{F_{\alpha}}^{+}]}=(-1)^{\frac{\text{codim}F_{\alpha}}{2}}\frac{\prod_{v_{\alpha}<0}(v_{\alpha}^{1/2}-v_{\alpha}^{-1/2})}{\prod_{v_{\alpha}>0}(v_{\alpha}^{1/2}-v_{\alpha}^{-1/2})}
\end{align}
Here $v_{\alpha}$ are the Chern roots of the normal bundle $N_{F_{\alpha}}$.

We mainly concern the connection between the geometric $R$-matrix in both $K$-theory and cohomology theory. We denote the Chern character map $ch:K_{T}(M(\mbf{v},\mbf{w}))\rightarrow H_{T}(M(\mbf{v},\mbf{w}))$ as sending $u$ to $e^{\kappa z}$, and $q$ to $e^{\kappa\hbar}$.

\begin{prop}
If we take $u=e^{\kappa z}, q=e^{\kappa\hbar}$ and take $\kappa\rightarrow0$, we obtain that:
\begin{align}
\lim_{\kappa\rightarrow0}\mc{R}^s(e^{\kappa z})=\mc{R}(z)
\end{align}
which is the geometric $R$-matrix for the corresponding Maulik-Okounkov Yangian $Y_{\hbar}(\mf{g}_{Q})$.
\end{prop}

\begin{proof}
First note that the cohomological stable envelope $\text{Stab}_{\sigma}$ is the lowest degree part of the limit of the $K$-theoretic stable envelope $\text{ch}(\text{Stab}_{\sigma,\mbf{m}})$ with $\kappa\rightarrow0$. Thus by the definition of the both cohomological and $K$-theoretic stable envelope, the degeneration limit of $\text{ch}(\text{Stab}_{-\sigma,\mbf{m}}^{-1}\text{Stab}_{\sigma,\mbf{m}})$ coincides with $\text{Stab}_{-\sigma}^{-1}\text{Stab}_{\sigma}$.
\end{proof}

Now we want to compare the cohomological limit of $U_{q}(\mf{g}_{w})$ with the Maulik-Okounkov Lie algebra $\mf{g}_{Q}$. To do this, first note that the wall $R$-matrix satisfy the following translation formula:
\begin{align}
R_{w+\mc{L}}^{\pm}(u)=\mc{L}R_{w}^{\pm}\mc{L}^{-1}
\end{align}
here $w+\mc{L}\subset\mbb{R}^I$ is the wall such that translates from $w$ to the direction $\mc{L}\in\mbb{Z}^I$.

Now we suppose that the wall in the product can be written as:
\begin{align}
\prod_{n\geq0}^{\leftarrow}\text{Ad}_{\mc{L}^n}(\prod_{w\in[s,s+\mc{L})}^{\leftarrow}R_{w}^+(u))
\end{align}

Here $\mc{L}\in\text{Pic}(X)$ is an ample line bundle. We further decompose $R_{w}^+(u)=\text{Id}+U_{w}^+(u)$ into the identity matrix and the strictly upper-triangular matrix $U_w(u)$, and $U_{w}(u)$ is nilpotent by definition. It is known that in this case:
\begin{align}
\prod_{n\geq0}^{\leftarrow}\text{Ad}_{\mc{L}^n}(\prod_{w\in[s,s+\mc{L})}^{\leftarrow}R_{w}^+(u))=\prod_{n\geq0}^{\leftarrow}(\prod_{w\in[s,s+\mc{L})}^{\leftarrow}(Id+u^{n\mc{L}}U_{w}^+(u)))
\end{align}

It has its matrix coefficients as the rational function over $u$ with the rational function of the form:
\begin{align}
\text{Id}+\sum_{w\in[s,s+\mc{L})}\frac{1}{1-u^{\mc{L}}}U_{w}^+(u)+\sum_{w_1,w_2\in[s,s+\mc{L})}\frac{1}{(1-u^{k_1\mc{L}})(1-u^{k_2\mc{L}})}U_{w_1}^+(u)U_{w_2}^+(u)+\cdots
\end{align}

It is known that $U_{w}(u)^+$ has its matrix coefficients as the monomials in $u=e^{\kappa z}$. So take the cohomological limit $\kappa\rightarrow0$ , we can obtain that:
\begin{align}
\text{Cohomological limit of }\prod_{w\geq s}^{\leftarrow}R_{w}^+(u)=\text{Id}+\frac{1}{z}\sum_{w,k\geq0}a_w^{(k)}r_{w,k}^++O(\frac{1}{z^2}),\qquad a_{w,k}\in\mbb{N}
\end{align}
Here $a_{w}^{(k)}$ are some non-negative integers.
Combining the computation,

Here
\begin{align}\label{classical-r-matrix}
\lim_{\kappa\rightarrow0}\frac{R_w^{\pm}-\text{Id}}{q-1}=r_{w}^{\pm}=\sum_{k\geq0}r_{w,k}^{\pm},\qquad q=e^{\kappa\hbar}
\end{align}
$r_{w,k}^{\pm}$ gives a root decomposition of the degeneration limit $r_{w}^{\pm}$ with the wall $w$ corresponding to the root $\alpha$. For the $r_{\infty}$:
\begin{align}
\lim_{\kappa\rightarrow0}\frac{\mc{R}^{\infty}(e^{\kappa z})-1}{e^{\kappa\hbar}-1}=r_{\infty}=\Omega,\qquad\Omega(\gamma)=\frac{\text{codim}(F)}{4}\gamma
\end{align}
is the classical limit of the wall $R$-matrix $R_{w}^{\pm}$ and the classical limit of the diagonal part $\mc{R}^{\infty}(u)$. $\Omega$ is the codimension scaling operator with $F\subset X$ a component of the fixed point subset $X^A\subset X$.

 we can see that:
\begin{align}
\text{Cohomological limit of }\mc{R}^s(u)=\text{Id}+\frac{1}{z}(\sum_{w\geq s,k\geq0}a_w^{(k)}r_{w,k}^++r_{\infty}-\sum_{w<s,k\geq0}a_w^{(k)}r_{w,k}^{-})+O(\frac{1}{z^2})\in\text{End}(H(\mbf{w}))
\end{align}

Now we define $r:=\sum_{w\geq s}a_wr_{w}^++r_{\infty}-\sum_{w<s}a_wr_{w}^{-}$.

If we do the $z^{-1}$-expansion of the Yangian geometric $R$-matrix, note that:
\begin{align}
\mc{R}(z)=\text{Id}+\frac{\hbar}{z}r_{Q}+O(\frac{1}{z^2})\in\text{End}(H(\mbf{w}))
\end{align}

By definition, $r_{Q}$ is the classical $r$-matrix corresponding to the Maulik-Okounkov Lie algebra $\mf{g}_{Q}$.

Also by definition, we can see that $\tilde{r}_{w}^+:=r_w^++r_{\infty}$ satisfies the classical Yang-Baxter equation:
\begin{align}\label{Classcal-Yang}
[\frac{\tilde{r}_{w,12}^{+}}{z_1-z_2},\frac{\tilde{r}_{w,13}^{+}}{z_1-z_3}]+[\frac{\tilde{r}_{w,12}^{+}}{z_1-z_2},\frac{\tilde{r}_{w,23}^{+}}{z_2-z_3}]+[\frac{\tilde{r}_{w,13}^{+}}{z_1-z_3},\frac{\tilde{r}_{w,23}^{+}}{z_2-z_3}]=0
\end{align}
This follows from the fact that $q^{\Omega}R_{w}^{\pm}$ satisfies the quantum Yang-Baxter equation.
This means that $r_{w}^{+}$ generates a Lie subalgebra $\mf{n}_{w}^{MO}\subset\mf{n}_{Q}^{MO}$ of the positive part of $\mf{n}_{Q}^{MO}$ over $\mbb{Z}$ (over $\mbb{Q}$ after being localised). In conclusion, $r_{w}^{\pm}$ generates the Lie subalgebra $\mf{g}_{w}^{MO}\subset\mf{g}_{Q}^{MO}$. Generally we call the Lie subalgebra $\mf{g}_{w}^{MO}$ as the \textbf{root subalgebra }of $\mf{g}_{Q}^{MO}$.

This concludes the following proposition:
\begin{prop}
The degeneration limit $\tilde{r}_{w}^{\pm,MO}=r_{\infty}+r_{w}^{\pm,MO}$ of the wall $R$-matrices $q^{\Omega}R_{w}^{\pm,MO}$ satisfies the classical Yang-Baxter equation \ref{Classcal-Yang}, and it generates a Lie subalgebra $\mf{g}_{w}^{MO}$ of the Maulik-Okounkov Lie algebra $\mf{g}_{Q}$.
\end{prop}

\section{\textbf{Wall subalgebras and the slope subalgebras}}
\subsection{Stable basis and the Pieri rule}
A good set of basis for $K(\mbf{w})$ under the action of the slope subalgebra $\mc{B}_{\mbf{m}}$ of slope $\mbf{m}$ is the stable basis $s^{\mbf{m}}_{\bm{\lambda}}$ of slope $\mbf{m}$.

The stable basis is constructed via the following way: We choose the stable envelope as:
\begin{equation}
\text{Stab}_{\sigma,\mbf{m}}:K_{T}(M(\mbf{v},\mbf{w})^T)\rightarrow K_{T}(M(\mbf{v},\mbf{w}))
\end{equation}
Here $T=A\times\mbb{C}^*_{q}\times\mbb{C}^*_{t}$ is the product of the maximal torus and the two-dimensional torus $\mbb{C}^*_q\times\mbb{C}^*_t$. It is known that the fixed points $M(\mbf{v},\mbf{w})^T$ of $T$-torus action are isolated, which is labeled by the multi-partitions $\bm{\lambda}$. Fixing the character $\sigma:\mbb{C}^*\rightarrow T$ that corresponds to the isolated fixed point subset. We denote the image of $\sigma$ by $u_i\mapsto t^{N_i}$ and assume that $N_1<<\cdots<< N_{\mbf{w}}$. This gives the dominance order of partitions over the fixed point. We would always denote $s^{\mbf{m}}_{\bm{\lambda}}:=\text{Stab}_{\sigma,\mbf{m}}(|\bm{\lambda}\rangle)$.

Also, for $F\in K_{T}(pt)_{loc}$ with $u_i\rightarrow t^{N_i}$ as above. We denote $\text{max(min)}\text{deg}(F)$ as the maximal (minimum) $t$-degree of $F$.

The diagonal part of $s^{\mbf{m}}_{\bm{\lambda}}$ is given by:
\begin{align}
s^{\mbf{m}}_{\bm{\lambda}}|_{\bm{\lambda}}=k_{\bm{\lambda}}=\prod_{\square\in\bm{\lambda}}[\prod_{\blacksquare\in\bm{\lambda}}\zeta(\frac{\chi_{\blacksquare}}{\chi_{\square}})\prod_{k=1}^{\mbf{w}}[\frac{u_k}{q\chi_{\square}}][\frac{\chi_{\square}}{qu_k}]]^{(-)}
\end{align}
Here $(-)$ stands for the product of $[x]$ such that the $t$-degree of $[x]$ is smaller than $0$. The notation for $(+)$ and $(0)$ has the similar meaning. The stable basis satisfy the following degree bounding condition:
\begin{equation}
\begin{aligned}
&\text{max deg }s^{\mbf{m}}_{\bm{\lambda}|\bm{\mu}}\leq\text{max deg}k_{\bm{\mu}}+\mbf{m}\cdot(\mbf{c}_{\bm{\mu}}-\mbf{c}_{\bm{\lambda}})\\
&\text{min deg }s^{\mbf{m}}_{\bm{\lambda}|\bm{\mu}}>-\text{max deg}k_{\bm{\mu}}+\mbf{m}\cdot(\mbf{c}_{\bm{\mu}}-\mbf{c}_{\bm{\lambda}})
\end{aligned}
\end{equation}

Here $\mbf{c}_{\bm{\lambda}}=(c_{\bm{\lambda}}^{1},\cdots,c_{\bm{\lambda}}^{n})$ with $c_{\bm{\lambda}}^{i}=\sum_{\square\in\bm{\lambda}}^{c_{\square}\equiv i}c_{\square}$.

It is known that the slope subalgebra $\mc{B}_{\mbf{m}}$ is generated by $P^{\mbf{m}}_{[i,j)}$ and $Q^{\mbf{m}}_{-[i,j)}$. 

Now given an operator $F:K(\mbf{v}',\mbf{w})\rightarrow K(\mbf{v},\mbf{w})$ whose matrix coefficients in terms of fixed point basis is given by:
\begin{align}
F|\bm{\mu}\rangle=\sum_{|\bm{\lambda}|=\mbf{v}}F^{\bm{\lambda}}_{\bm{\mu}}\cdot|\bm{\lambda}\rangle
\end{align}

If we choose the stable basis $s^{\mbf{m}}_{\bm{\lambda}}$ of the slope $\mbf{m}$, we write down the coefficients with respect to the stable basis as:
\begin{align}
F\cdot s^{\mbf{m}}_{\bm{\mu}}=\sum_{|\bm{\lambda}|=\mbf{v}}\gamma^{\bm{\lambda}}_{\bm{\mu}}\cdot s^{\mbf{m}}_{\bm{\lambda}}
\end{align}

The good property about the stable basis $s^{\mbf{m}}_{\bm{\mu}}$ is that it subtracts the special information of the action of the slope subalgebra $\mc{B}_{\mbf{m}}$.

\begin{lem}{\cite{N15}}\label{Pieri-rule}
For $F\in\mc{B}_{\mbf{m}}^{\pm}$, we have that:
\begin{align}
\gamma^{\bm{\lambda}}_{\bm{\mu}}=(\text{l.d.}F^{\bm{\lambda}}_{\bm{\mu}})\cdot\frac{\text{l.d.}(k_{\bm{\mu}})}{\text{l.d.}(k_{\bm{\lambda}})}
\end{align}

Here $\text{l.d.}$ stands for the lowest $t$-degree part that is equal to $\text{max deg}(k_{\bm{\mu}})-\text{max deg}(k_{\bm{\lambda}})+\mbf{m}\cdot(\mbf{c}_{\bm{\lambda}}-\mbf{c}_{\bm{\mu}})$.
\end{lem}

\subsection{The Hopf algebra embedding}

The advantage of the Pieri rule is that it can be used to construct the Hopf algebra embedding of $\mc{B}_{\mbf{m}}$ into the Maulik-Okounkov quantum affine algebra $U_{q}^{MO}(\hat{\mf{g}}_{Q})$.

\begin{prop}\label{inclusion-slope}
There is a Hopf algebra embedding 
\begin{align}
\mc{B}_{\mbf{m}}\hookrightarrow U_{q}^{MO}(\hat{\mf{g}}_{Q})
\end{align}
such that the restriction to $U_{q}(\mf{g}_{w})$ gives the Hopf algebra embedding $U_{q}(\mf{g}_{w})\hookrightarrow U_{q}^{MO}(\mf{g}_{w})$.
\end{prop}
\begin{proof}
We have shown the proof of the fact in \cite{Z23}. Here we give some detailed version of the proof.

The compatibility of the antipode map and the counit is obvious. We only need to prove the compatibility of the coproduct.
First we prove that the coproduct $\Delta_{\mbf{m}}$ on $\mc{B}_{\mbf{m}}$ intertwine the $\text{Stab}_{\mbf{m}}$, i.e.
\begin{equation}
\begin{tikzcd}
\mc{B}_{\mbf{m}}\arrow[r,hook]\arrow[d,"\Delta_{\mbf{m}}"]&\prod_{\mbf{w}}\text{End}(K_{T_{\mbf{w}}}(M(\mbf{w})))\arrow[d,"\text{Stab}_{\mbf{m}}^{\dagger}\circ(-)\circ\text{Stab}_{\mbf{m}}"]\\
\mc{B}_{\mbf{m}}\hat{\otimes}\mc{B}_{\mbf{m}}\arrow[r,hook]&\prod_{\mbf{w}_1,\mbf{w}_2}\text{End}(K_{T_{\mbf{w}_1}}(M(\mbf{w}_1)))\hat{\otimes}(K_{T_{\mbf{w}_2}}(M(\mbf{w}_2)))
\end{tikzcd}
\end{equation}

For simplicity, we only show the proof for $P^{\mbf{m}}_{[i;j)}$, and this means that we need to prove the following identity:
\begin{align}
P^{\mbf{m}}_{[i;j)}\circ\text{Stab}_{\mbf{m}}(s^{\mbf{m}}_{\bm{\lambda}_1}\otimes s^{\mbf{m}}_{\bm{\lambda}_2})=\text{Stab}_{\mbf{m}}(\Delta_{\mbf{m}}(P^{\mbf{m}}_{[i;j)})(s^{\mbf{m}}_{\bm{\lambda}_1}\otimes s^{\mbf{m}}_{\bm{\lambda}_2}))
\end{align}

For simplicity, using the formula for the coproduct, we can reduce the proof to the case $P^{\mbf{m}}_{[i;i+1)_h}$.

Recall the formula of $P^{\mbf{m}}_{[i;j)_h}$ on the stable basis $s^{\mbf{m}}_{\bm{\lambda}}$ in the Lemma \ref{Pieri-rule}:
\begin{align}
P^{\mbf{m}}_{[i;j)_h}s^{\mbf{m}}_{\bm{\mu}}=\sum_{\bm{\lambda}}\text{l.d.}(P^{\mbf{m}}_{[i;j)_h}(\bm{\lambda}\backslash\bm{\mu})\bm{\Theta}_{\bm{\lambda}\backslash\bm{\mu}})\frac{\text{l.d.}(\mf{K}_{\bm{\mu}})}{\text{l.d.}(\mf{K}_{\bm{\lambda}})}\cdot s^{\mbf{m}}_{\bm{\lambda}}
\end{align}

Here:
\begin{align}
\mf{K}_{\bm{\lambda}}=\prod_{\square\in\bm{\lambda}}[\prod_{\blacksquare\in\bm{\lambda}}\zeta(\frac{\chi_{\blacksquare}}{\chi_{\square}})\prod_{k=1}^{\mbf{w}}[\frac{u_k}{q\chi_{\square}}][\frac{\chi_{\square}}{qu_k}]]^{(-)}
\end{align}
\begin{align}
\bm{\Theta}_{\bm{\lambda}\backslash\bm{\mu}}=\prod_{\blacksquare\in\bm{\lambda}\backslash\bm{\mu}}[\prod_{\square\in\bm{\mu}}\zeta(\frac{\chi_{\blacksquare}}{\chi_{\square}})\prod_{k=1}^{\mbf{w}}[\frac{u_k}{q\chi_{\blacksquare}}]]
\end{align}

We first consider the following part:
\begin{equation}
\begin{aligned}
&\text{l.d.}(\bm{\Theta}_{\bm{\lambda}\backslash\bm{\mu}})\frac{\text{l.d.}(\mf{K}_{\bm{\mu}})}{\text{l.d.}(\mf{K}_{\bm{\lambda}})}\\
=&\frac{\text{l.d.}(\prod_{\square\in\bm{\mu}}^{\blacksquare\in\bm{\lambda}\backslash\bm{\mu}}\zeta(\frac{\chi_{\blacksquare}}{\chi_{\square}})\prod_{i=1}^{\mbf{w}}[\frac{u_i}{q\chi_{\blacksquare}}])^{(+)}}{\text{l.d.}(\prod_{\square\in\bm{\mu}}^{\blacksquare\in\bm{\lambda}\backslash\bm{\mu}}\zeta(\frac{\chi_{\square}}{\chi_{\blacksquare}})\prod_{i=1}^{\mbf{w}}[\frac{\chi_{\blacksquare}}{qu_i}])^{(-)}}\cdot\frac{\prod_{\blacksquare\in\bm{\lambda}\backslash\bm{\mu}}(\prod_{\square\in\bm{\mu}}\zeta(\frac{\chi_{\blacksquare}}{\chi_{\square}})\prod_{i=1}^{\mbf{w}}[\frac{u_i}{q\chi_{\blacksquare}}])^{(0)}}{\prod_{\blacksquare,\blacksquare'\in\bm{\lambda}\backslash\bm{\mu}}\zeta(\frac{\chi_{\blacksquare}}{\chi_{\blacksquare'}})^{(-)}}
\end{aligned}
\end{equation}

The superscripts $(+)$, $(0)$ and $(-)$ refer to that we only retain those factors $[x]$ such that $\text{deg}(x)>0$, $\text{deg}(x)=0$ and $\text{deg}(x)<0$. We can see that the degree $0$ part does not contribute to the lowest degree.

We want to analyze the behavior of the formula in terms of the power of $q$. For the first fraction, by calculation we can see that:
\begin{align}
\frac{\text{l.d.}\prod_{\blacksquare\in\bm{\lambda}\backslash\bm{\mu}}(\prod_{i=1}^{\mbf{w}}[\frac{u_i}{q\chi_{\blacksquare}}])^{(+)}}{\text{l.d.}(\prod_{\blacksquare\in\bm{\lambda}\backslash\bm{\mu}}\prod_{i=1}^{\mbf{w}}[\frac{\chi_{\blacksquare}}{qu_i}])^{(-)}}=\prod_{\blacksquare\in\bm{\lambda}\backslash\bm{\mu}}q^{\sum_{i=w_{\blacksquare}+1}^{\mbf{w}}}(\cdots)
\end{align}

by the following formula:

\begin{align}
\text{l.d.}\prod_{\blacksquare\in\bm{\lambda}\backslash\bm{\mu}}(\prod_{i=1}^{\mbf{w}}[\frac{u_i}{q\chi_{\blacksquare}}])^{(+)}=\text{l.d.}\prod_{\blacksquare\in\bm{\lambda}\backslash\bm{\mu}}(\prod_{i=w_{\blacksquare}+1}^{\mbf{w}}[\frac{u_i}{q\chi_{\blacksquare}}])=\prod_{\blacksquare\in\bm{\lambda}\backslash\bm{\mu}}q^{\frac{1}{2}\sum_{i=w_{\blacksquare}+1}^{\mbf{w}}}(\cdots)
\end{align}

\begin{align}
\text{l.d.}(\prod_{\blacksquare\in\bm{\lambda}\backslash\bm{\mu}}\prod_{i=1}^{\mbf{w}}[\frac{\chi_{\blacksquare}}{qu_i}])^{(-)}=\text{l.d.}(\prod_{\blacksquare\in\bm{\lambda}\backslash\bm{\mu}}\prod_{i=w_{\blacksquare}+1}^{\mbf{w}}[\frac{\chi_{\blacksquare}}{qu_i}])=\prod_{\blacksquare\in\bm{\lambda}\backslash\bm{\mu}}\prod_{i=w_{\blacksquare}+1}^{\mbf{w}}q^{-1/2}(\cdots)
\end{align}

The terms $(\cdots)$ in all above three formulas are the terms that does not change when we do the decomposition $(\bm{\lambda}\backslash\bm{\mu})=(\bm{\lambda}_1,\bm{\lambda}_2)\backslash(\bm{\mu}_1,\bm{\mu}_2)$.

In this way if we have that $\bm{\lambda}\backslash\bm{\mu}=(\bm{\lambda}_1,\bm{\lambda}_2)\backslash(\bm{\mu}_1,\bm{\mu}_2)$ such that $\bm{\lambda}_2\backslash\bm{\mu}_2=[i,j)$, we have that:
\begin{equation}\label{decomposition-w-coeff}
\begin{aligned}
\prod_{\blacksquare\in\bm{\lambda}\backslash\bm{\mu}}q^{\sum_{i=w_{\blacksquare}+1}^{\mbf{w}}}=&\prod_{\blacksquare\in\bm{\lambda}_1\backslash\bm{\mu}_1}q^{\sum_{i=w_{\blacksquare}+1}^{\mbf{w}}}\prod_{\blacksquare\in\bm{\lambda}_2\backslash\bm{\mu}_2}q^{-\sum_{i=w_{\blacksquare}+1}^{\mbf{w}}}\\
=&\prod_{\blacksquare\in\bm{\lambda}_1\backslash\bm{\mu}_1}q^{\sum_{i=w_{\blacksquare}+1}^{\mbf{w}_1+\mbf{w}_2}}\prod_{\blacksquare\in\bm{\lambda}_2\backslash\bm{\mu}_2}q^{\sum_{i=w_{\blacksquare}+1}^{\mbf{w}_2}}q^{\sum_{i=\mbf{w}_2+1}^{\mbf{w}_1+\mbf{w}_2}}\\
=&\prod_{\blacksquare\in\bm{\lambda}_1\backslash\bm{\mu}_1}q^{\sum_{i=w_{\blacksquare}+1}^{\mbf{w}_1+\mbf{w}_2}}\prod_{\blacksquare\in\bm{\lambda}_2\backslash\bm{\mu}_2}q^{\sum_{i=w_{\blacksquare}+1}^{\mbf{w}_2}}q^{\mbf{w}_1\cdot[i,j)}
\end{aligned}
\end{equation}

Also for
\begin{align}\label{important-coeff}
\frac{\text{l.d.}(\prod_{\square\in\bm{\mu}}^{\blacksquare\in\bm{\lambda}\backslash\bm{\mu}}\zeta(\frac{\chi_{\blacksquare}}{\chi_{\square}}))^{(+)}}{\text{l.d.}(\prod_{\square\in\bm{\mu}}^{\blacksquare\in\bm{\lambda}\backslash\bm{\mu}}\zeta(\frac{\chi_{\square}}{\chi_{\blacksquare}}))^{(-)}}=\prod_{\blacksquare\in\bm{\lambda}\backslash\bm{\mu}}q^{(\sum_{\square\in\bm{\mu}}^{\sigma_{\square}=\sigma_{\blacksquare}+1,,c_{\square}>c_{\blacksquare}}+\sum_{\square\in\bm{\mu}}^{\sigma_{\square}=\sigma_{\blacksquare}-1,c_{\square}>c_{\blacksquare}}-2\sum_{\square\in\bm{\mu}}^{\sigma_{\square}=\sigma_{\blacksquare},,c_{\square}>c_{\blacksquare}})}
\end{align}

Here $c_{\square}=x-y+N_i$ stands for the content of the box $\square\in\bm{\lambda}$. $\sigma_{\square}$ stands for the color of the box $\square\in\bm{\lambda}$. The formula is due to the following formula:
\begin{equation}
\begin{aligned}
\text{l.d.}(\prod_{\square\in\bm{\mu}}^{\blacksquare\in\bm{\lambda}\backslash\bm{\mu}}\zeta(\frac{\chi_{\blacksquare}}{\chi_{\square}}))^{(+)}=&\text{l.d.}(\prod_{\blacksquare\in\bm{\lambda}\backslash\bm{\mu}}\frac{\prod_{\square\in\bm{\mu}}^{\sigma_{\square}=\sigma_{\blacksquare}+1}[\frac{t\chi_{\square}}{q\chi_{\blacksquare}}]\prod_{\square\in\bm{\mu}}^{\sigma_{\square}=\sigma_{\blacksquare}-1}[\frac{\chi_{\square}}{qt\chi_{\blacksquare}}]}{\prod_{\square\in\bm{\mu}}^{\sigma_{\square}=\sigma_{\blacksquare}}[\frac{\chi_{\square}}{\chi_{\blacksquare}}][\frac{\chi_{\square}}{q^2\chi_{\blacksquare}}]})^{(+)}\\
=&\prod_{\blacksquare\in\bm{\lambda}\backslash\bm{\mu}}q^{1/2(\sum_{\square\in\bm{\mu}}^{\sigma_{\square}=\sigma_{\blacksquare}+1,,c_{\square}>c_{\blacksquare}}+\sum_{\square\in\bm{\mu}}^{\sigma_{\square}=\sigma_{\blacksquare}-1,c_{\square}>c_{\blacksquare}}-2\sum_{\square\in\bm{\mu}}^{\sigma_{\square}=\sigma_{\blacksquare},,c_{\square}>c_{\blacksquare}})}
\end{aligned}
\end{equation}

If we take $\bm{\lambda}\backslash\bm{\mu}=(\bm{\lambda}_1,\bm{\lambda}_2)\backslash(\bm{\mu}_1,\bm{\mu}_2)$ such that $\bm{\lambda}_1\backslash\bm{\mu}_1=[i,a)$, $\bm{\lambda}_2\backslash\bm{\mu}_2=[a,j)$, we have the decomposition:
\begin{equation}\label{decomposition-v-coeff}
\begin{aligned}
&\prod_{\blacksquare\in\bm{\lambda}\backslash\bm{\mu}}q^{(\sum_{\square\in\bm{\mu}}^{\sigma_{\square}=\sigma_{\blacksquare}+1,,c_{\square}>c_{\blacksquare}}+\sum_{\square\in\bm{\mu}}^{\sigma_{\square}=\sigma_{\blacksquare}-1,c_{\square}>c_{\blacksquare}}-2\sum_{\square\in\bm{\mu}}^{\sigma_{\square}=\sigma_{\blacksquare},,c_{\square}>c_{\blacksquare}})}\\
=&\prod_{\blacksquare\in\bm{\lambda}_1\backslash\bm{\mu}_1}q^{(\sum_{\square\in\bm{\mu}_1}^{\sigma_{\square}=\sigma_{\blacksquare}+1,,c_{\square}>c_{\blacksquare}}+\sum_{\square\in\bm{\mu}_1}^{\sigma_{\square}=\sigma_{\blacksquare}-1,c_{\square}>c_{\blacksquare}}-2\sum_{\square\in\bm{\mu}_1}^{\sigma_{\square}=\sigma_{\blacksquare},,c_{\square}>c_{\blacksquare}})}\\
&\times\prod_{\blacksquare\in\bm{\lambda}_2\backslash\bm{\mu}_2}q^{(\sum_{\square\in\bm{\mu}_2}^{\sigma_{\square}=\sigma_{\blacksquare}+1,,c_{\square}>c_{\blacksquare}}+\sum_{\square\in\bm{\mu}_2}^{\sigma_{\square}=\sigma_{\blacksquare}-1,c_{\square}>c_{\blacksquare}}-2\sum_{\square\in\bm{\mu}_2}^{\sigma_{\square}=\sigma_{\blacksquare},,c_{\square}>c_{\blacksquare}})}\\
&\times\prod_{\blacksquare\in\bm{\lambda}_2\backslash\bm{\mu}_2}q^{(\sum_{\square\in\bm{\mu}_1}^{\sigma_{\square}=\sigma_{\blacksquare}+1,,c_{\square}>c_{\blacksquare}}+\sum_{\square\in\bm{\mu}_1}^{\sigma_{\square}=\sigma_{\blacksquare}-1,c_{\square}>c_{\blacksquare}}-2\sum_{\square\in\bm{\mu}_1}^{\sigma_{\square}=\sigma_{\blacksquare},,c_{\square}>c_{\blacksquare}})}\\
=&\frac{\text{l.d.}(\prod_{\square\in\bm{\mu}_1}^{\blacksquare\in\bm{\lambda}_1\backslash\bm{\mu}_1}\zeta(\frac{\chi_{\blacksquare}}{\chi_{\square}}))^{(+)}}{\text{l.d.}(\prod_{\square\in\bm{\mu}_1}^{\blacksquare\in\bm{\lambda}_1\backslash\bm{\mu}_1}\zeta(\frac{\chi_{\square}}{\chi_{\blacksquare}}))^{(-)}}\frac{\text{l.d.}(\prod_{\square\in\bm{\mu}_2}^{\blacksquare\in\bm{\lambda}_2\backslash\bm{\mu}_1}\zeta(\frac{\chi_{\blacksquare}}{\chi_{\square}}))^{(+)}}{\text{l.d.}(\prod_{\square\in\bm{\mu}_2}^{\blacksquare\in\bm{\lambda}_2\backslash\bm{\mu}_2}\zeta(\frac{\chi_{\square}}{\chi_{\blacksquare}}))^{(-)}}q^{-[i,a)\cdot C\mbf{v}_1}
\end{aligned}
\end{equation}

For the $\frac{1}{\prod_{\blacksquare,\blacksquare'\in\bm{\lambda}\backslash\bm{\mu}}\zeta(\frac{\chi_{\blacksquare}}{\chi_{\blacksquare'}})^{(-)}}$ part, we have the decomposition.
\begin{equation}\label{decomposition-denom-coeff}
\begin{aligned}
\text{l.d.}\frac{1}{\prod_{\blacksquare,\blacksquare'\in\bm{\lambda}\backslash\bm{\mu}}\zeta(\frac{\chi_{\blacksquare}}{\chi_{\blacksquare'}})^{(-)}}=&\text{l.d.}(\frac{1}{\prod_{\blacksquare,\blacksquare'\in\bm{\lambda}_1\backslash\bm{\mu}_1}\zeta(\frac{\chi_{\blacksquare}}{\chi_{\blacksquare'}})^{(-)}\prod_{\blacksquare,\blacksquare'\in\bm{\lambda}_2\backslash\bm{\mu}_2}\zeta(\frac{\chi_{\blacksquare}}{\chi_{\blacksquare'}})^{(-)}}\frac{1}{\prod_{\blacksquare\in\bm{\lambda_1}\backslash\bm{\mu_1}}^{\blacksquare'\in\bm{\lambda_2}\backslash\bm{\mu}_2}\zeta(\frac{\chi_{\blacksquare}}{\chi_{\blacksquare'}})^{(-)}})\\
=&\text{l.d.}(\frac{1}{\prod_{\blacksquare,\blacksquare'\in\bm{\lambda}_1\backslash\bm{\mu}_1}\zeta(\frac{\chi_{\blacksquare}}{\chi_{\blacksquare'}})^{(-)}\prod_{\blacksquare,\blacksquare'\in\bm{\lambda}_2\backslash\bm{\mu}_2}\zeta(\frac{\chi_{\blacksquare}}{\chi_{\blacksquare'}})^{(-)}})\text{l.d.}(\frac{1}{\prod_{\blacksquare\in\bm{\lambda_1}\backslash\bm{\mu_1}}^{\blacksquare'\in\bm{\lambda_2}\backslash\bm{\mu}_2}\zeta(\frac{\chi_{\blacksquare}}{\chi_{\blacksquare'}})})
\end{aligned}
\end{equation}

For the coefficients $\text{l.d.}P^{\mbf{m}}_{[i,j)}(\bm{\lambda}\backslash\bm{\mu})$, we recall the formula in section $6.3$ in \cite{N15} that $\text{l.d.}P^{\mbf{m}}_{[i,j)}(\bm{\lambda}\backslash\bm{\mu})$ is nonzero only for $\bm{\lambda}\backslash\bm{\mu}$ being the cavalcade $C$, and it has the formula:
\begin{align}
\text{l.d.}P^{\mbf{m}}_{[i,j)}(C)=\prod_{\square\in C}\chi_{\square}\cdot q^{\#_C+\text{ht}(C)+\text{ind}^{\mbf{m}}_{C}}\prod_{i\leq a<b<j}^{a\leftrightarrow b}\text{l.d.}\zeta(\frac{\chi_b}{\chi_a})
\end{align}

The double arrow $a\leftrightarrow b$ means that the box $\square_a$ and $\square_b$ should not be next to each other in the same ribbon. 

Now if we take $\bm{\lambda}\backslash\bm{\mu}=(\bm{\lambda}_1,\bm{\lambda}_2)\backslash(\bm{\mu}_1,\bm{\mu}_2)$, for simplicity we can assume that both $\bm{\lambda}_1\backslash\bm{\mu}_1$ and $\bm{\lambda}_2\backslash\bm{\mu}_2$ are both calvacades $C_1,C_2$ such that $C_1=[i,k)$, $C_2=[k,j)$. Easy computation shows that:
\begin{align}\label{decomposition-P-coef}
\text{l.d.}P^{\mbf{m}}_{[i,j)}(C)=\text{l.d.}P^{\mbf{m}}_{[i,k)}(C_1)\text{l.d.}P^{\mbf{m}}_{[k,j)}(C_2)\prod_{\square\in C_1}^{\square'\in C_2}\text{l.d.}\zeta(\frac{\chi_{\square}}{\chi_{\square'}})
\end{align}

For the general $P^{\mbf{m}}_{[i,j)}$, using the coproduct formula:
\begin{align}
\Delta_{\mbf{m}}(P^{\mbf{m}}_{[i;j)_h})=\sum_{a=i}^{j}P^{\mbf{m}}_{[a,j)_h}\varphi_{[i,a)_h}\otimes P^{\mbf{m}}_{[i,a)_h}
\end{align}

Now use the formula for the stable envelope:
\begin{equation}
\begin{aligned}
&P^{\mbf{m}}_{[i,j)_h}\text{Stab}_{\mbf{m}}(s^{\mbf{m}}_{\bm{\mu}_1}\otimes s^{\mbf{m}}_{\bm{\mu}_2})=P^{\mbf{m}}_{[i,j)_h}s^{\mbf{m}}_{(\bm{\mu}_1,\bm{\mu}_2)},\qquad\bm{\lambda}=(\bm{\lambda}_1,\bm{\lambda}_2),\bm{\mu}=(\bm{\mu}_1,\bm{\mu}_2)\\
=&\sum_{\bm{\lambda}_1,\bm{\lambda}_2}\text{l.d.}(P^{\mbf{m}}_{[i;j)_h}((\bm{\lambda}_1,\bm{\lambda}_2)\backslash(\bm{\mu}_1,\bm{\mu}_2))\bm{\Theta}_{(\bm{\lambda}_1,\bm{\lambda}_2)\backslash(\bm{\mu}_1,\bm{\mu}_2)})\frac{\text{l.d.}(\mf{K}_{(\bm{\mu}_1,\bm{\mu}_2)})}{\text{l.d.}(\mf{K}_{(\bm{\lambda}_1,\bm{\lambda}_2)})}\cdot s^{\mbf{m}}_{(\bm{\lambda}_1,\bm{\lambda}_2)}\\
=&\sum_{\bm{\lambda}}\frac{\text{l.d.}(\prod_{\square\in\bm{\mu}}^{\blacksquare\in\bm{\lambda}\backslash\bm{\mu}}\zeta(\frac{\chi_{\blacksquare}}{\chi_{\square}})\prod_{i=1}^{\mbf{w}}[\frac{u_i}{q\chi_{\blacksquare}}])^{(+)}}{\text{l.d.}(\prod_{\square\in\bm{\mu}}^{\blacksquare\in\bm{\lambda}\backslash\bm{\mu}}\zeta(\frac{\chi_{\square}}{\chi_{\blacksquare}})\prod_{i=1}^{\mbf{w}}[\frac{\chi_{\blacksquare}}{qu_i}])^{(-)}}\cdot\frac{\prod_{\blacksquare\in\bm{\lambda}\backslash\bm{\mu}}(\prod_{\square\in\bm{\mu}}\zeta(\frac{\chi_{\blacksquare}}{\chi_{\square}})\prod_{i=1}^{\mbf{w}}[\frac{u_i}{q\chi_{\blacksquare}}])^{(0)}}{\prod_{\blacksquare,\blacksquare'\in\bm{\lambda}\backslash\bm{\mu}}\zeta(\frac{\chi_{\blacksquare}}{\chi_{\blacksquare'}})^{(-)}}\times\\
&\times\text{l.d.}(P^{\mbf{m}}_{[i;j)_h}(\bm{\lambda}\backslash\bm{\mu})s^{\mbf{m}}_{\bm{\lambda}}\\
\end{aligned}
\end{equation}
Combining the formula \ref{decomposition-denom-coeff} \ref{decomposition-w-coeff} \ref{decomposition-P-coef} \ref{decomposition-v-coeff}, we have that
\begin{equation}\label{split-coproduct-shit}
\begin{aligned}
&\text{Stab}_{\mbf{m}}(\Delta_{\mbf{m}}(P^{\mbf{m}}_{[i,j)_h})s^{\mbf{m}}_{\bm{\mu}_1}\otimes s^{\mbf{m}}_{\bm{\mu}_2})=\sum_{a=i}^j\text{Stab}_{\mbf{m}}(P^{\mbf{m}}_{[a,j)_h}\varphi_{[i,a)_h}s^{\mbf{m}}_{\bm{\mu}_1}\otimes P^{\mbf{m}}_{[i,a)_h}s^{\mbf{m}}_{\bm{\mu}_2})\\
=&\sum_{a=i}^j\sum_{\bm{\lambda}_1,\bm{\lambda}_2}\frac{\text{l.d.}(\prod_{\square\in\bm{\mu}_1}^{\blacksquare\in\bm{\lambda}_1\backslash\bm{\mu}_1}\zeta(\frac{\chi_{\blacksquare}}{\chi_{\square}})\prod_{i=1}^{\mbf{w}}[\frac{u_i}{q\chi_{\blacksquare}}])^{(+)}}{\text{l.d.}(\prod_{\square\in\bm{\mu}_1}^{\blacksquare\in\bm{\lambda}_1\backslash\bm{\mu}_1}\zeta(\frac{\chi_{\square}}{\chi_{\blacksquare}})\prod_{i=1}^{\mbf{w}}[\frac{\chi_{\blacksquare}}{qu_i}])^{(-)}}\cdot\frac{\prod_{\blacksquare\in\bm{\lambda}_1\backslash\bm{\mu}_1}(\prod_{\square\in\bm{\mu}_1}\zeta(\frac{\chi_{\blacksquare}}{\chi_{\square}})\prod_{i=1}^{\mbf{w}}[\frac{u_i}{q\chi_{\blacksquare}}])^{(0)}}{\prod_{\blacksquare,\blacksquare'\in\bm{\lambda}_1\backslash\bm{\mu}_1}\zeta(\frac{\chi_{\blacksquare}}{\chi_{\blacksquare'}})^{(-)}}\text{l.d.}(P^{\mbf{m}}_{[a;j)_h}(\bm{\lambda}\backslash\bm{\mu})\\
&\times\frac{\text{l.d.}(\prod_{\square\in\bm{\mu}_2}^{\blacksquare\in\bm{\lambda}_2\backslash\bm{\mu}_2}\zeta(\frac{\chi_{\blacksquare}}{\chi_{\square}})\prod_{i=1}^{\mbf{w}}[\frac{u_i}{q\chi_{\blacksquare}}])^{(+)}}{\text{l.d.}(\prod_{\square\in\bm{\mu}_2}^{\blacksquare\in\bm{\lambda}_2\backslash\bm{\mu}_2}\zeta(\frac{\chi_{\square}}{\chi_{\blacksquare}})\prod_{i=1}^{\mbf{w}}[\frac{\chi_{\blacksquare}}{qu_i}])^{(-)}}\cdot\frac{\prod_{\blacksquare\in\bm{\lambda}_2\backslash\bm{\mu}_2}(\prod_{\square\in\bm{\mu}_2}\zeta(\frac{\chi_{\blacksquare}}{\chi_{\square}})\prod_{i=1}^{\mbf{w}}[\frac{u_i}{q\chi_{\blacksquare}}])^{(0)}}{\prod_{\blacksquare,\blacksquare'\in\bm{\lambda}_2\backslash\bm{\mu}_2}\zeta(\frac{\chi_{\blacksquare}}{\chi_{\blacksquare'}})^{(-)}}\text{l.d.}(P^{\mbf{m}}_{[i;a)_h}(\bm{\lambda}_2\backslash\bm{\mu}_2)\\
&\times q^{(\mbf{w}_1-C\mbf{v}_2)\cdot[i,a)_h}\text{Stab}_{\mbf{m}}(s^{\mbf{m}}_{\bm{\mu}_1}\otimes s^{\mbf{m}}_{\bm{\mu}_2})\\
=&P^{\mbf{m}}_{[i,j)_h}\text{Stab}_{\mbf{m}}(s^{\mbf{m}}_{\bm{\mu}_1}\otimes s^{\mbf{m}}_{\bm{\mu}_2})=P^{\mbf{m}}_{[i,j)_h}s^{\mbf{m}}_{(\bm{\mu}_1,\bm{\mu}_2)}
\end{aligned}
\end{equation}

This gives the intertwining property of the coproduct $\Delta_{\mbf{m}}$ with the stable envelopes $\text{Stab}_{\mbf{m}}$, thus it gives the proof of the Hopf algebra map.

For the algebra map $\mc{B}_{\mbf{m}}\rightarrow U_{q}(\mf{g}_{\mbf{m}})$, we claim the following thing: 
\begin{lem}
If $a\in U_{q}(\mf{g}_{w})$ such that $\Delta_{\mbf{m}}(a)\subset\mc{B}_{\mbf{n}}\hat{\otimes}\mc{B}_{\mbf{n}}$, $\mbf{n}=\mbf{m}$ . Then similarly for $b\in U_q^{MO}(\mf{g}_{w})$ such that $\Delta_{\mbf{m}}(b)\in U_q^{MO}(\mf{g}_{w})\hat{\otimes} U_q^{MO}(\mf{g}_{w})$, $w$ contains the slope point $\mbf{m}$.
\end{lem}
\begin{proof}
Now suppose that $a\in\mc{B}_{\mbf{n}}$ such that $\Delta_{\mbf{m}}(a)\subset\mc{B}_{\mbf{n}}\hat{\otimes}\mc{B}_{\mbf{n}}$, and we choose a monotone path from $\mbf{n}$ to $\mbf{m}$, and now use the formula:
\begin{align}
\Delta_{\mbf{m}}(a)=T_{\mbf{m},\mbf{n}}\Delta_{\mbf{n}}(a)T_{\mbf{m},\mbf{n}}^{-1}
\end{align}

Since $\Delta_{\mbf{n}}(a)\subset\mc{B}_{\mbf{n}}\hat{\otimes}\mc{B}_{\mbf{n}}$, now we write the coproduct of $a$ as:
\begin{align}
\Delta_{\mbf{n}}(a)=\sum a^{(1)}\otimes a^{(2)}
\end{align}

Without loss of generality, we could reduce the case to $P^{\mbf{n}}_{[i;j)}$ and thus by the coproduct formula one could further reduce the proof to $P^{\mbf{n}}_{[i;i+1)_h}$, and in this case:
\begin{align}
\Delta_{\mbf{n}}(P^{\mbf{n}}_{[i;i+1)_h})=P^{\mbf{n}}_{[i;i+1)_h}\otimes\varphi_{[i;i+1)_h}+1\otimes P^{\mbf{n}}_{[i;i+1)_h}
\end{align}

Then the first part of the lemma is true following from the fact that $[a,P^{\mbf{n}}_{[i;i+1)_h}]\notin\mc{B}_{\mbf{n}}$ if $a\in\mc{B}_{\mbf{m}}$ with $\mbf{m}\neq\mbf{n}$ without $\varphi_{i}$ kind of elements. This can be checked via computing the slope of the element $[a,P^{\mbf{n}}_{[i;i+1)_h}]$, which turns out to be not of the slope $\mbf{n}$ unless $a\in\mc{B}_{\mbf{m}}\cap\mc{B}_{\mbf{n}}$.

For the second part of the lemma, recall that $U_{q}^{MO}(\mf{g}_{\mbf{m}})$ is defined as the matrix coefficients of the MO wall $R$-matrix $R_{\mbf{m},\mbf{m}+\epsilon\bm{\theta}}:=\lim_{\epsilon\rightarrow0}\text{Stab}_{\mbf{m}+\epsilon\bm{\theta}}^{-1}\circ\text{Stab}_{\mbf{m}}$ with $\bm{\theta}\in\mbb{Z}^{r}$. By simple calculation, one can show that the matrix elements of $R_{\mbf{m},\mbf{m}+\epsilon\bm{\theta}}$ acting on the corresponding stable basis $s^{\mbf{m}}_{\bm{\lambda}_1}\otimes s^{\mbf{m}}_{\bm{\lambda}_2}$ depends on the value over the fixed point geometric $R$-matrix, i.e.
\begin{align}\label{fixed-point-R-matrix-stable-basis}
R_{\mbf{m},\mbf{m}+\epsilon\bm{\theta}}(s^{\mbf{m}}_{\bm{\lambda}_1}\otimes s^{\mbf{m}}_{\bm{\lambda}_2})=\lim_{\epsilon\rightarrow0}\sum_{\bm{\mu}_1,\bm{\mu}_2}(R^{\sigma}_{\mbf{m},\mbf{m}+\epsilon\bm{\theta}})_{(\lambda_1,\lambda_2),(\mu_1,\mu_2)}s^{\mbf{m}+\epsilon\bm{\theta}}_{\bm{\mu}_1}\otimes s^{\mbf{m}+\epsilon\bm{\theta}}_{\bm{\mu}_2}
\end{align}

For the precise formula, one could use the result of \cite{D21} to write down the explicit formula of the fixed point $R$-matrix.

The $R$-matrix has the root decomposition:
\begin{align}
R_{\mbf{m},\mbf{m}+\epsilon\bm{\theta}}=1+\sum_{\langle\alpha,\theta\rangle>0}(R_{\mbf{m},\mbf{m}+\epsilon\bm{\theta}})_{\alpha},\qquad (R_{\mbf{m},\mbf{m}+\epsilon\bm{\theta}})^{-1}=1+\sum_{\langle\alpha,\theta\rangle>0}(R_{\mbf{m},\mbf{m}+\epsilon\bm{\theta}}^{-1})_{\alpha}
\end{align}
We have the formula of the coproduct acting on the $R$-matrices:
\begin{equation}
\begin{aligned}
\Delta_{\mbf{m}}\langle s_{\bm{\lambda}}^{\mbf{m}}|(1\otimes m)(R_{\mbf{m},\mbf{m}+\epsilon\bm{\theta}})_{\alpha}|s_{\bm{\mu}}^{\mbf{m}}\rangle=\langle s_{\bm{\lambda}}^{\mbf{m}}|(1\otimes m)\bigoplus_{\beta+\gamma=\alpha}(R_{\mbf{m},\mbf{m}+\epsilon\bm{\theta}})_{\beta,13}(R_{\mbf{m},\mbf{m}+\epsilon\bm{\theta}})_{\gamma,23}|s_{\bm{\mu}}^{\mbf{m}}\rangle
\end{aligned}
\end{equation}

Now if we insert a finite-rank operator $m\in\text{End}(K(\mbf{w}))$ of rank $\alpha$ and consider $(1\otimes m)R_{\mbf{m},\mbf{m}+\epsilon\bm{\theta}}$. This implies that:
\begin{equation}
\begin{aligned}
\langle s_{\bm{\lambda}}^{\mbf{m}}|(1\otimes m)(R_{\mbf{m},\mbf{m}+\epsilon\bm{\theta}})_{\alpha}|s_{\bm{\mu}}^{\mbf{m}}\rangle=\text{Stab}_{\mbf{m}}\langle s_{\bm{\lambda}}^{\mbf{m}}|(1\otimes m)\bigoplus_{\beta+\gamma=\alpha}(R_{\mbf{m},\mbf{m}+\epsilon\bm{\theta}})_{\beta,13}(R_{\mbf{m},\mbf{m}+\epsilon\bm{\theta}})_{\gamma,23}|s_{\bm{\mu}}^{\mbf{m}}\rangle\rangle\text{Stab}_{\mbf{m}}^{-1}
\end{aligned}
\end{equation}

This gives that the matrix coefficients of $(R_{\mbf{m},\mbf{m}+\epsilon\bm{\theta}})_{\alpha}$ is given by:
\begin{equation}\label{MO-decomposition}
\begin{aligned}
&\langle s_{\bm{\lambda}_1}^{\mbf{m}}|\otimes \langle s_{\bm{\lambda}}^{\mbf{m}}|(1\otimes m)(R_{\mbf{m},\mbf{m}+\epsilon\bm{\theta}})_{\alpha}|s_{\bm{\mu}}^{\mbf{m}}\rangle\otimes|s_{\bm{\mu}_1}^{\mbf{m}}\rangle\\
=&\langle s^{\mbf{m}}_{\bm{\lambda}_1}\otimes s^{\mbf{m}}_{\bm{\lambda}_2}\otimes s^{\mbf{m}}_{\bm{\lambda}_3}|(1\otimes m)\bigoplus_{\beta+\gamma=\alpha}(R_{\mbf{m},\mbf{m}+\epsilon\bm{\theta}})_{\beta,13}(R_{\mbf{m},\mbf{m}+\epsilon\bm{\theta}})_{\gamma,23}|s^{\mbf{m}}_{\bm{\mu}_1}\otimes s^{\mbf{m}}_{\bm{\mu}_2}\otimes s^{\mbf{m}}_{\bm{\mu}_3}\rangle
\end{aligned}
\end{equation}

Here $\bm{\lambda}=(\bm{\lambda}_1,\bm{\lambda}_2)$, $\bm{\mu}=(\bm{\mu}_1,\bm{\mu}_2)$. Now we use the coproduct $\Delta_{\mbf{n}}$ acting on $(R_{\mbf{m},\mbf{m}+\epsilon\bm{\theta}})_{\alpha}$, and it turns out that 

Thus:
\begin{equation}
\begin{aligned}
&\Delta_{\mbf{n}}\langle s_{\bm{\lambda}}^{\mbf{m}}|(1\otimes m)(R_{\mbf{m},\mbf{m}+\epsilon\bm{\theta}})_{\alpha}|s_{\bm{\mu}}^{\mbf{m}}\rangle\\
=&\sum_{\beta_1+\beta_2+\beta_3+\beta_4=\alpha}(T_{\mbf{n},\mbf{m}})_{\beta_1}\langle s_{\bm{\lambda}}^{\mbf{m}}|(1\otimes m)(R_{\mbf{m},\mbf{m}+\epsilon\bm{\theta}})_{\beta_2,13}(R_{\mbf{m},\mbf{m}+\epsilon\bm{\theta}})_{\beta_3,23}|s_{\bm{\mu}}^{\mbf{m}}\rangle(T_{\mbf{n},\mbf{m}}^{-1})_{\beta_4}
\end{aligned}
\end{equation}

Here $T_{\mbf{n},\mbf{m}}:=\text{Stab}_{\mbf{n}}^{-1}\circ\text{Stab}_{\mbf{m}}$. Without loss of generality we can assume that $R_{\mbf{m},\mbf{m}+\epsilon\bm{\theta}}$ contains only the root of the form $k\alpha$.
\begin{equation}
\begin{aligned}
&\Delta_{\mbf{n}}\langle s_{\bm{\lambda}}^{\mbf{m}}|(1\otimes m)(R_{\mbf{m},\mbf{m}+\epsilon\bm{\theta}})_{k\alpha}|s_{\bm{\mu}}^{\mbf{m}}\rangle\\
&=\bigoplus_{\substack{0<k_1+k_2<k\\k_1+k_2+k_3+k_4=k}}(T_{\mbf{n},\mbf{m}})_{k_1\alpha}\langle s_{\bm{\lambda}}^{\mbf{m}}|(1\otimes m)(R_{\mbf{m},\mbf{m}+\epsilon\bm{\theta}})_{k_3\alpha,13}(R_{\mbf{m},\mbf{m}+\epsilon\bm{\theta}})_{k_4\alpha,23}|s_{\bm{\mu}}^{\mbf{m}}\rangle(T_{\mbf{n},\mbf{m}}^{-1})_{k_2\alpha}\\
&+\bigoplus_{k_1+k_2=k}\langle s_{\bm{\lambda}}^{\mbf{m}}|(1\otimes m)(R_{\mbf{m},\mbf{m}+\epsilon\bm{\theta}})_{k_1\alpha,13}(R_{\mbf{m},\mbf{m}+\epsilon\bm{\theta}})_{k_2\alpha,23}|s_{\bm{\mu}}^{\mbf{m}}\rangle
\end{aligned}
\end{equation}

Now note that $(T_{\mbf{n},\mbf{m}})_{k\alpha}$ is written as the product of the wall $R$-matrices $R_{w_1}\cdots R_{w_n}$. Let us assume that the corresponding root for the wall $w_i$ can be written as $\alpha_i$ and the slope $\mbf{m}_i$. Thus this requires that $k_1\alpha_1+\cdots+k_n\alpha_n=k\alpha$ and $\sum_{i=1}^nk_i\langle\mbf{m}_i,\alpha_i\rangle=k\langle\mbf{m},\alpha\rangle$. while since $\alpha_i\in\mbb{N}^I$ and the path from $\mbf{m}$ to $\mbf{n}$ is monotone, i.e. will go through the wall only once. This implies that $\sum_{i=1}^nk_i\langle\mbf{m}-\mbf{m}_i,\alpha_i\rangle>0$. This implies that:
\begin{align}\label{inequality-good}
\sum_{i=1}^nk_i\langle\mbf{m}-\mbf{m}_i,\alpha_i\rangle>0
\end{align}

\begin{lem}
Given two wall subalgebra $U_{q}(\mf{g}_{w_1})$ and $U_{q}(\mf{g}_{w_2})$ such that they correspond to the same root $\alpha$ while $\langle s_1,\alpha\rangle\neq\langle s_2,\alpha\rangle$. Then the intersection $U_{q}(\mf{g}_{w_1})\cap U_{q}(\mf{g}_{w_2})$ only contains the Cartan elements.
\end{lem}
\begin{proof}
This can be done by the following: If there exists $a\in U_{q}(\mf{g}_{w_1})\cap U_{q}(\mf{g}_{w_2})$, note that we can fix that $a\in(U_{q}(\mf{g}_{w_1})\cap U_{q}(\mf{g}_{w_2}))_{k\alpha}$ such that it can be written as $\text{tr}_{V_1}((1\otimes m_1)R_{w_1})=\text{tr}_{V_2}((1\otimes m_2)R_{w_2})$. But recall that if we fix a subtorus $A\subset T$ which acts on the fixed point component trivially, and we consider the $A$-degree of $a$, the $A$-degree of $(R_{w_1})_{k\alpha}$ and $(R_{w_2})_{k\alpha}$ are different by $\langle s_1-s_2,k\alpha\rangle\in\mbb{Z}$, while the operators $m_1,m_2$ are assumed to have no $A$-degree. Thus it is a contradiction. Thus only the degree zero part intersects, and the degree zero part is generated by the Cartan elements.

\end{proof}

Now back to our cases, since $\sum_{i=1}^{n}k_{i}\langle\mbf{m}-\mbf{m}_i,\alpha\rangle\in\mbb{Z}$. Now first note that the elements $(T_{\mbf{n},\mbf{m}})_{k\alpha}$ contains elements on the wall $w'$ such that $\langle k'\alpha,\mbf{m'}\rangle\in\mbb{Z}$ for some slope point $\mbf{m}'$. Now by the formula \ref{inequality-good}, this implies that $\mbf{m}'\neq\mbf{m}$ for all of $\mbf{m}'$. Thus by the lemma we know that the elements $(T_{\mbf{m},\mbf{n}})_{k\alpha}$ does not live in $U_{q}(\mf{g}_{w})$ with the wall containing $\mbf{m}$. This finishes the proof.

\end{proof}

Now via the embedding $U_{q}(\mf{g}_{w})\hookrightarrow U_{q}^{MO}(\hat{\mf{g}}_{Q})$, since $\Delta_{\mbf{m}}(a)\in\mc{B}_{\mbf{m}}\otimes\mc{B}_{\mbf{m}}$ for $a\in U_{q}(\mf{g}_{w})\subset\mc{B}_{\mbf{m}}$. By the result in \cite{OS22}, the restriction of $\Delta_{\mbf{m}}$ of $U_{q}^{MO}(\hat{\mf{g}}_{Q})$ to $U_{q}^{MO}(\mf{g}_w)$ with $\mbf{m}$ on the wall $w$ is the coproduct $\Delta_{w}$ on $U_{q}^{MO}(\mf{g}_w)$.

Now if $a\notin U_{q}^{MO}(\mf{g}_{w})$, this means that $\Delta_{\mbf{m}}(a)\notin U_{q}^{MO}(\mf{g}_{w})\otimes U_{q}^{MO}(\mf{g}_{w})$, which is a contradiction.
\end{proof}

\subsection{Root subalgebra and dimension counting}
It can be proved that $\mf{n}_{w}^{MO}$ is the classical limit of the wall subalgebra $U_{q}(\mf{n}_{w}^{MO})$:
\begin{lem}\label{q-1-degeneration}
There is an isomorphism:
\begin{align}
U_{q}(\mf{n}_w^{MO})/(q-1)U_{q}(\mf{n}_w^{MO})=U(\mf{n}_w^{MO})
\end{align}
\end{lem}
\begin{proof}
By definition $U_{q}(\mf{n}_{w}^{MO})$ is generated by the matrix coefficients of $R_{w}^+=\text{Id}+U_{w}^+$. The left-ideal $(q-1)U_{q}(\mf{n}_{w}^{MO})$ makes sense since $R_{w}^{MO}$ is an integral $K$-theory class and it is generated by the matrix coefficients of $(q-1)R_{w}^+$, and now by definition:
\begin{align}
\lim_{q\rightarrow1}\frac{R_w^+-\text{Id}}{q-1}=r_{w}^+
\end{align}

which generates the Lie algebra $\mf{n}_w^{MO}$.
\end{proof}

It is easy to see that $U_{q}(\mf{n}_w^{MO})$ and $U(\mf{n}_w^{MO})$ has the grading:
\begin{align}
U_{q}(\mf{n}_w^{MO})=\bigoplus_{\mbf{k}\in\mbb{N}^I}U_{q}(\mf{n}_w^{MO})_{\mbf{k}},\qquad U(\mf{n}_w^{MO})=\bigoplus_{\mbf{k}\in\mbb{N}^I}U(\mf{n}_w^{MO})_{\mbf{k}}
\end{align}

It is easy to see that each graded pieces $U_{q}(\mf{n}_{w})_{\mbf{k}}$ and $U(\mf{n}_{w})_{\mbf{k}}$ are finite-dimensional. From the Lemma \ref{q-1-degeneration} we can have that:
\begin{align}
\text{dim}(U(\mf{n}_{w})_{\mbf{k}})\leq\text{dim}(U_{q}(\mf{n}_w)_{\mbf{k}})
\end{align}

Moreover, for the Lie algebra $\mf{g}_{Q}$, we have the decomposition:
\begin{align}
\mf{g}_{Q}^{MO}=\bigcup_{w\in\text{Walls}_0}\mf{g}_w^{MO}
\end{align}
with the walls defined as:
\begin{align}
w=\{s\in\mbb{R}^I|\langle\alpha,s\rangle=0\}
\end{align}

The wall set for the $K$-theoretic stable envelope is actually the "affinisation" of the roots for the Lie algebra $\mf{n}_{Q}^{MO}$:
\begin{prop}
The wall set for the $K$-theoretic stable envelope is of the following type:
\begin{align}
w\supset\{s\in\mbb{R}^I|\langle\alpha,s\rangle=n\}
\end{align}
with $\alpha$ the roots of the Maulik-Okounkov Lie algebra $\mf{n}_{Q}^{MO}$.
\end{prop}
\begin{proof}
Since $R_{w}^{\pm}$ is an integral $K$-theory class in $K_{T}(X^{A}\times X^A)$, it has the coefficients as the Laurent polynomial of $q$. This means that by the Chern character map $\text{ch}:K_{T}(X^{A}\times X^A)\rightarrow H^*_{T}(X^A\times X^A)$, passing to the substitution $q=e^{\hbar}$ and taking the smallest degree part. This means that we have transformed $U_{q}^{MO}(\mf{g}_{w})$ as a $\mbb{Q}[[\hbar]]$-algebra with $U_{q}^{MO}(\mf{g}_{w})/\hbar U_{q}^{MO}(\mf{g}_{w})\cong U(\mf{g}_{w}^{MO})$. Since $\mf{g}_{w}^{MO}\subset\mf{g}_{Q}^{MO}$. This means that $w$ must be in the set of roots of $\mf{g}_{Q}^{MO}$ corresponding to either the empty subset or root subalgebra of $\mf{g}_{Q}^{MO}$.
\end{proof}

Now we study the generators for the wall Lie algebra $\mf{g}_w^{MO}$.

Recall that given a fractional line bundle $\mc{L}_{w}$ on the wall $w$. It has been proved in \cite{OS22} that $R_{w}^+$ is triangular with monomial in spectral parametres $u$ matrix elements:
\begin{align}
R_{w}^{+}|_{F_2\times F_1}=
\begin{cases}
1&F_1=F_2\\
Au^{\langle\bm{\mu}(F_2)-\bm{\mu}(F_1),\mc{L}_{w}\rangle}&F_1\geq F_2\\
0&\text{otherwise}
\end{cases}
\end{align}

Since the wall $w$ is defined as $\langle\alpha,s\rangle=0$, we need to require that $\mbf{v}_1-\mbf{v}_2=k\alpha$. After the degeneration to the Lie algebra $\mf{n}_{w}$, it means that for $e\in\mbf{n}_{w}$, we have:
\begin{align}
e:K(M(\mbf{v},\mbf{w}))\rightarrow K(M(\mbf{v}+k\alpha,\mbf{w}))
\end{align}

This makes the Lie algebra $\mf{n}_{w}^{MO}$ as a subalgebra of $\bigoplus_{k\geq0}\mf{n}_{k\alpha}^{MO}$. Here $\mf{n}_{k\alpha}^{MO}$ is the space of root vectors of $\mf{n}^{MO}$ of the form $e_{k\alpha}$.

\begin{thm}\label{confirmation-of-generators}
If the quiver $Q$ is of the affine type $A$, then there is an isomorphism
\begin{align}
\mf{n}_{w}^{MO}\cong\bigoplus_{k\geq0}\mf{n}_{k\alpha}^{MO}
\end{align}

for the wall $w$ corresponding to the root $\alpha$ such that $w=\{s|\langle s,\alpha\rangle\in\mbb{Z}\}$.
\end{thm}
\begin{proof}
Obviously by definition we have the inclusion
\begin{align}
\mf{n}_{w}^{MO}\subset\bigoplus_{k\geq0}\mf{n}_{k\alpha}^{MO}
\end{align}

Also we have the inclusion $\mf{n}_{w}\subset\mf{n}_{w}^{MO}$ from the Proposition \ref{inclusion-slope}. But we know that:
\begin{align}
\mf{n}_{w}\cong\bigoplus_{k\geq0}\mf{n}_{k\alpha}
\end{align}

By the result of Botta and Davison \cite{BD23}, we have the isomorphism:
\begin{align}
\bigoplus_{k\geq0}\mf{n}_{k\alpha}^{MO}\cong\bigoplus_{k\geq0}\mf{n}_{k\alpha}
\end{align}

which gives a chain of inclusion:
\begin{align}
\bigoplus_{k\geq0}\mf{n}_{k\alpha}\cong\mf{n}_{w}\subset\mf{n}_{w}^{MO}\subset\bigoplus_{k\geq0}\mf{n}_{k\alpha}
\end{align}

Thus we have the isomorphism $\mf{n}_{w}\cong\mf{n}_{w}^{MO}$, and this finishes the proof.

\end{proof}

\section{\textbf{Quantum difference equation}}
In this section we review the construction of the quantum difference equations in both algebraic and geometric ways. For details see \cite{OS22}\cite{Z23}.

The basic ingredients of the construction is given by the monodromy operators. For the Okounkov-Smirnov geometric monodromy operators, it is given by $\mbf{B}_{w}(z)\in\widehat{U_{q}^{MO}(\mf{g}_{w})}$, which is from the Maulik-Okounkov wall subalgebra $U_{q}^{MO}(\mf{g}_{w})$. For the algebraic one, it is given by the algebraic monodromy operators $\mbf{B}_{\mbf{m}}(z)\in\widehat{\mc{B}_{\mbf{m}}}$ with $\mc{B}_{\mbf{m}}$ the slope subalgebra.
\subsection{Okounkov-Smirnov quantum difference equation}
The Okounkov-Smirnov geometric quantum difference equation was proposed in \cite{N15} as the difference equation of the capping operator $\mbf{J}(u,z)\in K_{G}(M(\mbf{v},\mbf{w}))^{\otimes 2}\otimes\mbb{Q}[[z^{d}]]_{d\in\text{Pic}_{eff}(M(\mbf{v},\mbf{w}))}$ over the Kahler variable $z$:
\begin{align}
\mbf{J}(u,p^{\mc{L}}z)\mc{L}=\mbf{M}_{\mc{L}}(z)\mbf{J}(u,z)
\end{align}

Here we review the construction of the geometric quantum difference operator $\mbf{M}_{\mc{L}}(z)$ given by Okounkov and Smirnov in \cite{OS22}:

Here we denote $\lambda$ as $q^{\lambda}=z$, $p=q^{\tau}$ as in \cite{OS22}, and we abbreviate $R_{w}$ as $R_{w}^{MO}$ the MO wall $R$-matrix. Using the notation, we denote $F(\lambda):=F(q^{\lambda})$ as the function $F(z)$ of the Kahler variable.
For each MO wall subalgebra $U_{q}^{MO}(\mf{g}_{w})$ there are corresponding ABRR equations:
\begin{align}
J_{w}^{+}(\lambda)q_{(1)}^{-\lambda}q^{-\Omega}R_{w}^+=\hbar_{(1)}^{-\lambda}q^{\Omega}J_{w}^+(\lambda),\qquad q^{-\Omega}R_{w}^-q_{(1)}^{-\lambda}J_{w}^-(\lambda)=J_{w}^-(\lambda)q^{\Omega}q_{(1)}^{-\lambda}
\end{align}

We define:
\begin{align}
\mbf{J}_{w}^{\pm}(\lambda)=J_{w}^{\pm}(\lambda-\tau\mc{L}_w)=J_{w}^{\pm}(zp^{-\mc{L}_w}),\qquad\kappa=\frac{C\mbf{v}-\mbf{w}}{2}
\end{align}

and thus the monodromy operators are defined as:
\begin{align}
\mbf{B}_{w}(\lambda)=\mbf{m}((1\otimes S_{w})(\mbf{J}^{-}_{w}(\lambda)^{-1}))|_{\lambda\rightarrow\lambda+\kappa}
\end{align}

The first ABRR equation can be written as:
\begin{align}
\text{Ad}_{q_{(1)}^{\lambda}q^{-\Omega}}(J_{w}^+(\lambda))=J_{w}^+(\lambda)(R_{w}^+)^{-1}\in U_{q}(\mf{g}_{w})\hat{\otimes}U_{q}(\mf{g}_{w})
\end{align}

This gives the formal solution for the fusion operators $J_{w}^{+}(\lambda)$:
\begin{align}
J_{w}^+(\lambda)=\prod_{k=0}^{\substack{\rightarrow\\\infty}}\text{Ad}_{(q^{\lambda}_{(1)}q^{-\Omega})^k}(R_{w}^+)
\end{align}

Now choose a path from $s$ to $s-\mc{L}$, the Okounkov-Smirnov quantum difference operator $\mbf{B}^{s}_{\mc{L}}(\lambda)$ is defined as:
\begin{align}\label{geometric-product}
\mbf{B}^s_{\mc{L}}(\lambda)=\mc{L}\prod^{\leftarrow}_{w}\mbf{B}_{w}(\lambda)
\end{align}

Also we denote the $p$-difference operator as $\mc{A}^s_{\mc{L}}=T_{\mc{L}}^{-1}\mbf{B}^s_{\mc{L}}(\lambda)$ with $T_{\mc{L}}f(z)=f(p^{\mc{L}})$. It has been proved in \cite{OS22} that the difference operators $\{\mc{A}^{s}_{\mc{L}}\}$ form a holonomic module over the Picard torus $\text{Pic}(X)\otimes\mbb{C}^*$, i.e.
\begin{align}
[\mc{A}^{s}_{\mc{L}},\mc{A}^s_{\mc{L}'}]=0
\end{align}
and $\mc{A}^{s}_{\mc{L}}$ is independent of the choice of the path from $s$ to $s-\mc{L}$. Moreover, for the slope $s$ being in a subset of the opposite of the ample cone $\nabla\subset-C_{\text{ample}}$, $\mc{A}^{s}_{\mc{L}}$ is conjugate to $T_{\mc{L}}^{-1}\mbf{M}_{\mc{L}}(z)$ via $\text{Stab}_{\mc{C},s}$.

In this paper we will focus on the case of the affine type $A$. We will always use the expression \ref{geometric-product} as the geometric quantum difference operator. To distinguish from the algebraic quantum difference operators, we will denote them all by $\mbf{B}_{\mc{L}}^{s,MO}(\lambda)$ and $\mbf{B}_{w,MO}(\lambda)$.

\subsection{Algebraic quantum difference equation}
Here we review the construction of the algebraic quantum difference equation of affine type $A$ constructed in \cite{Z23}.

The quantum toroidal algebra $U_{q,t}(\hat{\hat{\mf{sl}}}_{n})$ admits the slope factorization:
\begin{align}
U_{q,t}(\hat{\hat{\mf{sl}}}_{n})=\bigotimes^{\rightarrow}_{\mu\in\mbb{Q}}\mc{B}_{\mbf{m}+\mu\bm{\theta}}
\end{align}

For each slope subalgebra $\mc{B}_{\mbf{m}}$, one can associate an element $J_{\mbf{m}}^{\pm}(\lambda)\in\mc{B}_{\mbf{m}}\hat{\otimes}\mc{B}_{\mbf{m}}$ such that they satisfy the ABRR equation:
\begin{align}
J_{\mbf{m}}^{+}(\lambda)q_{(1)}^{-\lambda}q^{\Omega}R_{\mbf{m}}^{+}=q_{(1)}^{-\lambda}q^{\Omega}J_{\mbf{m}}^{+}(\lambda),\qquad q^{\Omega}R_{\mbf{m}}^{-}q_{(1)}^{-\lambda}J_{\mbf{m}}^{-}(\lambda)=J_{\mbf{m}}^{-}(\lambda)q^{\Omega}q_{(1)}^{-\lambda}
\end{align}

the monodromy operator is defined as:
\begin{align}\label{defn-of-quantum-difference-operator}
\mbf{B}_{\mbf{m}}(\lambda)=m(1\otimes S_{\mbf{m}}(\mbf{J}_{\mbf{m}}^{-}(\lambda)^{-1}))|_{\lambda\rightarrow\lambda+\kappa}
\end{align}

Here $\kappa=\frac{C\mbf{v}-\mbf{w}}{2}$.

Let $\mc{L}\in Pic(X)$ be a line bundle. Now we fix a slope $s\in H^2(X,\mbb{R})$ and choose a path in $H^2(X,\mbb{R})$ from $s$ to $s-\mc{L}$. This path crosses finitely many slope points in some order $\{\mbf{m}_1,\mbf{m}_2,\cdots,\mbf{m}_m\}$.  And for this choice of a slope, line bundle and a path we associate the following operator:
\begin{align}\label{defnofqdeoperator}
\mbf{B}_{\mc{L}}^{s}(\lambda)=\mc{L}\mbf{B}_{\mbf{m}_m}(\lambda)\cdots\mbf{B}_{\mbf{m}_1}(\lambda)
\end{align}

We define the $q$-difference operators:
\begin{align}
\mc{A}^{s}_{\mc{L}}=T_{\mc{L}}^{-1}\mbf{B}^s_{\mc{L}}(\lambda)
\end{align}

It has been proved in \cite{Z23} that the $q$-difference operator $\mc{A}^{s}_{\mc{L}}$ is independent of the choice of the path from $s$ to $s-\mc{L}$ if we choose the generic path.

\subsection{Review of the analysis of the algebraic quantum difference equations}

The algebraic quantum difference equations has been analysed in detail in \cite{Z24}. Here we briefly review the result in there.

First note that at $z=0\in\text{Pic}(X)\otimes\mbb{C}^{\times}$, i.e. the corresponding chamber in $\text{Pic}(X)\otimes\mbb{R}$ corresponds to the chamber such that all the parametres $a_1,\cdots,a_n$ goes to $-\infty$. Let $M(0)=M_{n}(0)M_{n-1}(0)\cdots M_{1}(0)=\mc{L}_{n}\otimes\cdots\otimes\mc{L}_{1}$ and this matrix is diagonal in $K_{T}(X)$ with respect to the fixed point basis. And let $P$ be the matrix with columns given by fixed point eigenvectors $[\lambda]$. We denote by $E_{0}$ the diagonal matrix of eigenvalues, so that:
\begin{align}
M(0)P=PE_0
\end{align}

The fundamental solution around $z=0$ can be written as:
\begin{align}
\Psi_{0}(z)=P\Psi^{reg}_0(z)\exp(\sum_{i}\frac{\ln(E_0^{(i)})\ln(z_i)}{\ln(p)})
\end{align}

For the solution around $z=\infty$, by the Lemma 3.2 in \cite{Z23}, the operator $\mbf{M}_{\mc{O}(1)}(\infty)$ consists of the product of the form $\mbf{m}((1\otimes S_{\mbf{m}})(R_{\mbf{m}}^{-})^{-1})$. It is a diagonalizable matrix over $\mbb{Q}((q,t))$ with eigenvalues given by the formal power series of $q$ and $t$, i.e. it has distinct eigenvalues for the generic value of $q$ and $t$. We shall denote the corresponding matrix of eigenvectors as $H$, and $\mbf{E}_{\infty}$ be the digaonal matrix of eigenvalues:
\begin{align}
\mbf{M}_{\mc{O}(1)}(\infty)H=H\mbf{E}_{\infty}
\end{align}
the solution around $z=\infty$ can be written as:
\begin{align}
\Psi_{\infty}(z)=H\Psi_{\infty}^{reg}(z)\exp(\sum_{i}\frac{\ln(E_{\infty}^{(i)})\ln(z_i)}{\ln(p)})
\end{align}
such that
\begin{align}
\Psi^{reg}_{\infty}(z)\mbf{E}_{\infty}=H^{-1}\mbf{M}_{\mc{O}(1)}(z)H\Psi^{reg}_{\infty}(z)
\end{align}

The transition matrix between two fundamental solutions $\Psi_{0}(z)$ and $\Psi_{\infty}(z)$ is defined as:
\begin{align}
\textbf{Mon}(z):=\Psi_{0}(z)^{-1}\Psi_{\infty}(z)
\end{align}

Now we choose our path such that $\mbf{B}_{\mbf{m}}(z)$ are all of the generic point $\mbf{m}$. It remains to analyze $\mbf{B}_{w}(p^sz)$, by the result above, 
\begin{align}\label{limit-of-monodromy-operator}
\lim_{p\rightarrow0}\mbf{B}_{\mbf{m}}(p^sz)=
\begin{cases}
\mbf{m}((1\otimes S_{\mbf{m}})(R_{\mbf{m}}^{+})_{21}^{-1})&s<\mbf{m},\mbf{m}\in U_{s}\\
\mbf{B}_{\mbf{m}}(z)|_{p=1}&s=\mbf{m}\\
\mbf{B}_{\mbf{m}}(z_{I},\theta_{J})&s_{I}=\mbf{m}_{I},s_{I^c}\neq \mbf{m}_{I^c}\\
1&\text{Otherwise}
\end{cases}
\end{align}

Now we choose a generic representation of the quantum difference operator $\mbf{M}_{\mc{O}(1)}(z)=\mc{O}(1)\prod^{\rightarrow}_{\mbf{m}\in\text{Walls}}\mbf{B}_{\mbf{m}}(z)$ such that each monodromy operator $\mbf{B}_{\mbf{m}}(z)$ are either of the $U_{q}(\mf{sl}_2)$-type or of the $U_{q}(\hat{\mf{gl}}_1)$-type. The following theorem gives the expression for the $p\rightarrow0$ limit of the connection matrix:
\begin{thm}\label{p0-limit-connection}
For generic $s\in\mbb{Q}^n$ the connection matrix as the following asymptotic at $p\rightarrow0$:
\begin{align}
\lim_{p\rightarrow0}\textbf{Mon}^{reg}(p^sz)=
\begin{cases}
\prod^{\leftarrow}_{0\leq\mbf{m}<s}(\mbf{B}_{\mbf{m}}^*)^{-1}\cdot\mbf{T},\qquad s\geq0\\
\prod_{s<\mbf{m}<0}\mbf{B}_{\mbf{m}}^*\cdot\mbf{T},\qquad s<0
\end{cases}
\end{align}

Here $\mbf{T}:=P^{-1}H$, $0\leq\mbf{m}<s$ means the slope points $\mbf{m}$ in one generic path from $0$ to the point $s$ without intersecting $s$. Here $s\geq0$ and $s<0$ stands for $s=(s_1,\cdots,s_n)$ such that $s_{i}\leq 0$ or $s_{i}>0$ for every component $s_i$. The $\mbf{B}_{\mbf{m}}^*$ is defined as $P^{-1}\mbf{B}_{\mbf{m}}P$.
\end{thm}

The main advantage of using the theorem is that we can use it to compute the monodromy representation for the Dubrovin connection of the affine type $A$ quiver varieties:

\begin{thm}\label{degeneration-qde-thm}
The degeneration limit of the quantum difference operator $\mbf{M}_{\mc{L}}(z)$ coincides with the quantum multiplication operator $Q(\mc{L})$ up to a constant operator.
\end{thm}

This has been proved in \cite{Z24}. A direct consequence is that we can use this result to compute the monodromy representation.

Let $\Psi_{0,\infty}(q_1,q_2,p,z)$ be the solution to the quantum difference equation described above. Then we take $z=e^{2\pi s}$, $q_i=e^{2\pi i\hbar_i\tau}$, $q=e^{-2\pi i\tau}$. By the Theoreom \ref{degeneration-qde-thm}, for the degeneration limit of the solution $\psi_{0,\infty}(z)$, which is defined as follows:
\begin{align}
\psi_{0,\infty}(z)=\lim_{\tau\rightarrow0}\Psi_{0,\infty}(e^{2\pi i\hbar_1\tau},e^{2\pi i\hbar_2\tau},e^{-2\pi i\tau},z)\in \widehat{H_{T}(M(\mbf{v},\mbf{w}))}_{loc}
\end{align}

It is the solution to the corresponding Dubrovin connection. The solution lies in the completion of $H_{T}(M(\mbf{v},\mbf{w}))_{loc}$. Using this, we can define the transport of the solution of the Dubrovin connection:
\begin{align}
\text{Trans}(s):=\psi_{0}(e^{2\pi is})^{-1}\psi_{\infty}(e^{2\pi is})\in \widehat{H_{T}(M(\mbf{v},\mbf{w}))}_{loc}
\end{align}
Similarly as a result of the Theorem \ref{degeneration-qde-thm}, the transport of the solution is a limit the monodromy:
\begin{align}\label{trans}
\text{Trans}(s)=\lim_{\tau\rightarrow0}\textbf{Mon}(z=e^{2\pi is},q_1=e^{2\pi i\hbar_1\tau},q_2=e^{2\pi i\hbar_2\tau},q=e^{-2\pi i\tau})
\end{align} 

The monodromy corresponds to the element $\text{Trans}(s')^{-1}\text{Trans}(s)$, and if we denote the slope point of the difference by $\mbf{m}$, we can see that:
\begin{align}
\text{Trans}(s')^{-1}\text{Trans}(s)=\mbf{B}_{\mbf{m}}^*
\end{align}

In conclusion, we have the following monodromy representation:
\begin{thm}\label{monodromy-rep-casimir}
The monodromy representation:
\begin{align}
\pi_{1}(\mbb{P}^n\backslash\textbf{Sing},0^{+})\rightarrow\text{End}(H_{T}(M(\mbf{v},\mbf{w})))
\end{align}
of the Dubrovin connection is generated by $\mbf{B}_{\mbf{m}}^*$ with $q_1=e^{2\pi i\hbar_1},q_2=e^{2\pi i\hbar_2}$, i.e. the monodromy operators $\mbf{B}_{\mbf{m}}$ in the fixed point basis. The based point $0^+$ is a point sufficiently close to $0\in\mbb{P}^n$.
\end{thm}

\subsection{Asymptotic behavior of the geometric monodromy operators}

Before we analyze the solution of the difference equation, we first analyze the asymptotic behavior of the monodromy operators $\mbf{B}_{w}(z)$:
\begin{prop}\label{asymptotic-of-geo-mon}
Suppose that $\mbf{B}_{w}(z)$ corresponds to the wall $w$ such that there is only one connected component, i.e. $w$ corresponds to only one root $\alpha\in\mbb{Z}^I$. Then we have:
\begin{align}
\lim_{p\rightarrow0}\mbf{B}_{w}(zp^s)=
\begin{cases}
\mbf{m}((1\otimes S_{w})(R_{w}^{-})^{-1})&\langle\alpha,s\rangle>-m\text{ for the integer }m=\langle\mc{L}_w,\alpha\rangle\\
1&\text{otherwise}
\end{cases}
\end{align}
\end{prop}

\begin{proof}
We can see this via the analysis of the fusion operator $J_{w}^{+}(\lambda)$. We write down the ABRR equation in the form:
\begin{align}
\text{Ad}_{q_{(1)}^{\lambda}q^{-\Omega}}(J_{w}^{+}(z))=J_{w}^+(z)(R_{w}^{+})^{-1}
\end{align}
Since both $J_{w}^+(z)$ and $R_{w}^+$ are upper-triangular, they admit the root decomposition:
\begin{align}
J_{w}^{+}=1+\sum_{\alpha>0}J_{\alpha},\qquad
(R_{w}^+)^{-1}=1+\sum_{\alpha>0}R_{\alpha}
\end{align}
The $\alpha$-component of $J_{w}^+(z)$ can be solved as:
\begin{align}
J_{\alpha}(z)=\frac{1}{z^{\alpha}q^{m}-1}\sum_{\substack{\gamma+\delta=\alpha\\\gamma<\alpha}}J_{\gamma}(z)R_{\delta}
\end{align}

We substitute $z\mapsto p^sz$ and we have:
\begin{align}
J_{\alpha}(p^sz)=\frac{1}{z^{\alpha}q^mp^{s\cdot\alpha}-1}\sum_{\substack{\gamma+\delta=\alpha\\\gamma<\alpha}}J_{\gamma}(p^sz)R_{\delta}
\end{align}

Now since $w$ contains only one hyperplane corresponding to $\alpha$, this means that all the roots of $R_{w}$ are of the form $k\alpha$, and this means that:
\begin{align}\label{fusion-recursion}
J_{k\alpha}(p^sz)=\frac{1}{z^{k\alpha}q^mp^{ks\cdot\alpha}-1}\sum_{k_1+k_2=k,k_1<k}J_{k_1\alpha}(p^sz)R_{k_2\alpha}
\end{align}

Specifically for $k=1$:
\begin{align}
J_{\alpha}(p^sz)=\frac{1}{z^{\alpha}q^mp^{s\cdot\alpha}-1}R_{\alpha}
\end{align}
and when $p\rightarrow0$:
\begin{align}
\lim_{p\rightarrow0}J_{\alpha}(p^sz)=
\begin{cases}
R_{\alpha}&s\cdot\alpha>0\\
\frac{1}{z^{\alpha}q^m-1}R_{\alpha}&s\cdot\alpha=0\\
0&\text{otherwise}
\end{cases}
\end{align}

In conclusion we have that:
\begin{align}
\lim_{p\rightarrow0}J_{w}^{+}(p^sz)=
\begin{cases}
R_{w}^{+}&\qquad s\cdot\alpha>0\\
J_{w}^+(z)&\qquad s\cdot\alpha=0\\
1&\qquad\text{otherwise}
\end{cases}
\end{align}

Now for $\mbf{J}_{w}^{+}(z)$, it is defined via the shift $z\mapsto zp^{\mc{L}_{w}}$, we have that:
\begin{align}
\lim_{p\rightarrow0}\mbf{B}_{w}(p^sz)=
\begin{cases}
\mbf{m}((1\otimes S_{w})(R_{w}^{-})^{-1})&\qquad (s+\mc{L}_w)\cdot\alpha>0\\
\mbf{B}_w(z)|_{p=1}&\qquad (s+\mc{L}_w)\cdot\alpha=0\\
1&\qquad (s+\mc{L}_w)\cdot\alpha<0\\
\end{cases}
\end{align}

\end{proof}
\subsection{Analysis of the Okounkov-Smirnov quantum difference equations}\label{analysis-geometric}
The analysis of the Okounkov-Smirnov quantum difference equations is similar to the analysis given in \cite{Z24} for the algebraic quantum difference equations. The origin of the analysis of the quantum difference equation can be dated back to Smirnov's paper\cite{S21} on the analysis for the case of the Hilbert scheme $\text{Hilb}_{n}(\mbb{C}^2)$. For simplicity, we can analyze the solution for the quantum difference operator $\mbf{B}^{s,MO}_{\mc{L}}(z)$ at $\mc{L}=\mc{O}(1)$.

First we analyze the fundamental solution $\Psi^{MO}_{0}(z)$ around $z=0$. When $z=0$, $\mbf{B}^{s,MO}_{\mc{L}}(0)=\mc{L}$, and it has the eigenvectors given by the fixed point basis $|\bm{\lambda}\rangle$ with the eigenvalues $\mbf{E}_{0}$. In this case the fundamental solution $\Psi^{MO}_{0}(z)$ can be written as:

\begin{align}
\Psi_{0}^{MO}(z)=P\Psi^{reg,MO}_0(z)\exp(\sum_{i}\frac{\ln(E_0^{(i)})\ln(z_i)}{\ln(p)})
\end{align}

Here $\Psi^{reg,MO}_{0}(z)=1+\sum_{\mbf{d}\in\mbb{N}^I}a_{\mbf{d}}z^{\mbf{d}}$ is the regular solution satisfying the following difference equations:
\begin{align}
\Psi_{0}^{reg,MO}(zp^{\mc{L}_i})E_{0}^{(i)}=P^{-1}\mbf{B}_{\mc{O}(1)}^{s,MO}(z)P\Psi_{0}^{reg,MO}(z)
\end{align}

For simplicity, we denote $\mbf{M}^*(z)=P^{-1}\mbf{B}_{\mc{O}(1)}^{s,MO}(z)P$. Now the regular solution can be formally written as:
\begin{align}\label{formula-for-0}
\Psi_0^{reg}(z)=\prod^{\substack{\leftarrow\\\infty}}_{k=0}\mbf{M}^*_{k}(z)^{-1},\qquad\mbf{M}_{k}^*(z)=\mbf{E}_0^{k-1}\mbf{M}^*(p^kz)\mbf{E}_0^{-k}
\end{align}

For the fundamental solution around $z=\infty$, note that for $z=\infty$:
\begin{align}
\mbf{B}^{s}_{\mc{L}}(\infty)=\mc{L}\prod^{\leftarrow}_{w}\mbf{B}_{w}
\end{align}

In this case one could also find the eigenvector $H^{MO}$ with the distinct eigenvalues $E_{\infty}$ the proof is the same as the proof in the Appendix II of \cite{Z24}.

the solution around $z=\infty$ can be written as:
\begin{align}
\Psi_{\infty}^{MO}(z)=H^{MO}\Psi_{\infty}^{reg}(z)\exp(\sum_{i}\frac{\ln(E_{\infty}^{(i)})\ln(z_i)}{\ln(p)})
\end{align}
such that
\begin{align}
\Psi^{reg,MO}_{\infty}(z)\mbf{E}_{\infty}=(H^{MO})^{-1}\mbf{M}_{\mc{O}(1)}(z)H^{MO}\Psi^{reg}_{\infty}(z)
\end{align}

We use the similar procedure to compute the regular part of $\Psi^{reg}_{\infty}(z)$, and it is easy to see that the solution is of the form:
\begin{align}
\Psi^{reg}_{\infty}(z)=\prod^{\substack{\rightarrow\\\infty}}_{k=0}\mbf{M}^*_{-k}(z)
\end{align}
with
\begin{align}
\mbf{M}^{*}_{-k}(z)=\mbf{E}_{\infty}^{k-1}\mbf{M}^*
(zp^{-k})\mbf{E}_{\infty}^{-k}
\end{align}

The transition matrix between two fundamental solutions $\Psi_{0}(z)$ and $\Psi_{\infty}(z)$ is defined as:
\begin{align}
\textbf{Mon}^{MO}(z):=\Psi_{0}^{MO}(z)^{-1}\Psi_{\infty}^{MO}(z)
\end{align}

Furthermore we consider the regular part of the connection matrix:
\begin{align}
\textbf{Mon}^{MO,reg}(z):=\Psi_{0}^{MO,reg}(z)^{-1}\Psi_{\infty}^{MO,reg}(z)
\end{align}

Similar analysis as in \cite{Z24} can show that the degeneration limit of $\textbf{Mon}^{MO,reg}(p^sz)$ as $p\rightarrow0$ is similar to that of $\textbf{Mon}^{MO}(p^sz)$:
\begin{thm}\label{geo-p0-limit-connection}
For generic $s\in\mbb{Q}^n$ the connection matrix as the following asymptotic at $p\rightarrow0$:
\begin{align}
\lim_{p\rightarrow0}\textbf{Mon}^{reg,MO}(p^sz)=
\begin{cases}
\prod^{\leftarrow}_{0\leq w<s}(\mbf{B}_{w,MO}^*)^{-1}\cdot\mbf{T}^{MO},\qquad s\geq0\\
\prod_{s<w<0}\mbf{B}_{w,MO}^*\cdot\mbf{T}^{MO},\qquad s<0
\end{cases}
\end{align}

Here $\mbf{T}^{MO}:=P^{-1}H^{MO}$, $0\leq w<s$ means the wall $w$ that intersects with the generic path from $0$ to the point $s$ without intersecting $s$. Here $s\geq0$ and $s<0$ stands for $s=(s_1,\cdots,s_n)$ such that $s_{i}\leq 0$ or $s_{i}>0$ for every component $s_i$. The $\mbf{B}_{w,MO}^*$ is defined as $P^{-1}\mbf{B}_{w,MO}P$.
\end{thm}

We want to use this result to compute the monodromy representation of the Dubrovin connection. In the next section we are going to show how to give the degeneration limit of the Okounkov-Smirnov geometric quantum difference equation to the Dubrovin connection.

\section{\textbf{Cohomological limit to the Dubrovin connection}}

\subsection{Degeneration limit of the geometric quantum difference equations}
Now we focus on the case of the affine type $A$.

From \ref{fusion-recursion} one could have the general formula for the geometric monodromy operators :
\begin{align}
\mbf{B}_{w,MO}(\lambda)=\sum_{l\geq0}\sum_{k_1+\cdots+k_l=k}\frac{1}{(z^{-k_1\alpha}q^{-m}p^{-k_1\alpha\cdot\mc{L}_w}-1)\cdots(z^{-k_l\alpha}q^{-m}p^{-k_l\alpha\cdot\mc{L}_w}-1)}\mbf{m}((1\otimes S_{w})(R_{w,k_1\alpha}^{-,MO}\cdots R_{w,k_l\alpha}^{-,MO})^{-1})
\end{align}

Now we take the cohomological limit of the quantum difference operators $\mc{A}^{s}_{\mc{L}}$:
\begin{align}
\mc{A}^{s,coh}_{\mc{L}}=-d_{\mc{L}}+c_{1}(\mc{L})\cup+\sum_{w\in[s,s+\mc{L})}\mbf{B}_{w,MO}^{coh}(\lambda)
\end{align}

Now we take $p=e^{-2\pi i\tau}$, $q_{i}=e^{2\pi i\hbar_i\tau}$, and we let $\tau\rightarrow0$. Since we have the asymptotics of the wall $R$-matrices as $q\rightarrow1$:
\begin{align}
R_{w}^{-,MO}=\text{Id}+(q-1)r_{w}^{-,MO}+O((q-1)^2),\qquad\text{as }q\rightarrow1
\end{align}

We have the degeneration limit of $\mbf{B}_{w,MO}(\lambda)$:

\begin{equation}
\begin{aligned}
\mbf{B}^{coh}_{w,MO}(\lambda)=&-\sum_{k\geq0}\frac{\hbar}{1-z^{-k\alpha}}\mbf{m}((1\otimes S_{w})r_{w,k\alpha}^{-,MO})
=\frac{\hbar}{1-Ad_{q_{(1)}^{\lambda}}}\mbf{m}(r_{w}^-)
\end{aligned}
\end{equation}

By construction, $r_{w}^-$ corresponds to the root subalgebra $\mf{g}_{w}$ of $\mf{g}_{Q}$. By definition, the wall $w$ is defined as 
\begin{align}
w=\{s\in H^2(X,\mbb{R})|(s,\alpha)+n=0\}
\end{align}

We have known from the Theorem \ref{confirmation-of-generators} that each root pieces $\mf{n}_{w}^{MO}$ is isomorphic to $\bigoplus_{k\geq0}\mf{n}_{k\alpha}\subset\mf{sl}_n^+$. 

In this we can write down the orthogonal basis of $\mf{g}_{Q}^{\pm}$ as $\{e_{\pm\alpha}^{(i)}\}$. It is known that $r_{w+\mc{L}}^{\pm}=r_{w+\mc{L}}$, so in this way we have that:
\begin{equation}
\begin{aligned}
\mc{A}^{s,coh}_{\mc{L}}=&-d_{\mc{L}}+c_{1}(\mc{L})\cup+\sum_{w\in[s,s+\mc{L})}\mbf{B}_{w,MO}^{coh}(\lambda)\\
=&-d_{\mc{L}}+c_{1}(\mc{L})\cup+\sum_{w\in[s,s+\mc{L})}\frac{\hbar}{1-Ad_{q_{(1)}^{\lambda}}}\mbf{m}(r_{w}^{-,MO})\\
=&-d_{\mc{L}}+c_{1}(\mc{L})\cup+\sum_{w\in[s,s+\mc{L})}\sum_{i\in w}\frac{\hbar}{1-Ad_{q_{(1)}^{\lambda}}}e_{\alpha}^{(i)}e_{-\alpha}^{(i)}\\
=&-d_{\mc{L}}+c_{1}(\mc{L})\cup+\hbar\sum_{\alpha}\frac{\langle\alpha,\mc{L}\rangle}{1-q^{-\alpha}}e_{\alpha}e_{-\alpha}
\end{aligned}
\end{equation}

Here $\langle\alpha,\mc{L}\rangle$ counts the number of root vectors $e_{\pm\alpha}$ appearing in the path $s$ and $s+\mc{L}$. The roots $\alpha$ corresponds to the vector $[i,j)$ in the affine type $A$ case, and $e_{\pm\alpha}$ coincides with $E_{\pm[i,j)}$ via the comparison of the classical $r$-matrix:
\begin{align}
r_{Q}^{+}=\sum_{\alpha>0}e_{\alpha}\otimes e_{-\alpha}=\sum_{j>i}E_{[i,j)}\otimes E_{-[i,j)}
\end{align}

Combining these facts we can see that the degeneration limit of the geometric quantum difference equation coincides with the Dubrovin connection.

In this way we obtain the following theorem:
\begin{thm}\label{degeneration-of-geometric-qde}
The degeneration limit $\mc{A}^{s,coh}_{\mc{L}}$ of the Okounkov-Smirnov quantum difference equation $\mc{A}^{s}_{\mc{L}}$ of the affine type $A$ quiver varieties is the Dubrovin connection of the affine type $A$ quiver varieties.
\end{thm}

\subsection{Monodromy representations for the Dubrovin connection}
Now we can express the monodromy representation of the Dubrovin connection in terms of the degeneration limit of the connection matrix of the Okounkov-Smirnov quantum difference equation.

Let $\Psi_{0,\infty}^{MO}(q_1,q_2,p,z)$ be the solution of the Okounkov-Smirnov quantum difference equation as constructed in the subsection \ref{analysis-geometric}. We take $p=e^{-2\pi i\tau}$, $q_{i}=e^{2\pi i\hbar_i\tau}$ and $z=e^{2\pi is}$, and take the degeneration limit of the solution as:
\begin{align}
\psi_{0,\infty}^{MO}(z)=\lim_{\tau\rightarrow0}\Psi_{0,\infty}^{MO}(e^{2\pi i\hbar_1\tau},e^{2\pi i\hbar_2\tau},e^{-2\pi i\tau})
\end{align}

The transport of the two solutions is defined as:
\begin{align}
\text{Trans}^{MO}(s):=\psi_{0}^{MO}(e^{2\pi is})^{-1}\psi_{\infty}^{MO}(e^{2\pi is})
\end{align}

As a result of the Theorem \ref{degeneration-of-geometric-qde}, the transport of the solution can be written as:
\begin{align}
\text{Trans}^{MO}(s)=\lim_{\tau\rightarrow0}\textbf{Mon}^{MO}(z=e^{2\pi is},q_1=e^{2\pi i\hbar_1\tau},q_2=e^{2\pi i\hbar_2\tau},q=e^{-2\pi i\tau})
\end{align}

Similar as the argument in the Proposition $7.4$ in \cite{Z24}, we have that:
\begin{align}
\text{Trans}^{MO}(s)=\lim_{p\rightarrow0}\textbf{Mon}^{MO,reg}(p^s,e^{2\pi i\hbar_1},e^{\pi i\hbar_2},p)
\end{align}
for generic $s\in\mbb{R}^n$.

Also the similar argument in \cite{Z24} applies for the case of the geometric quantum difference equations. The monodromy representation corresponds to the difference of the transport matrix $\text{Trans}^{MO}(s')^{-1}\text{Trans}^{MO}(s)$. Using the similar computation as in \cite{Z24}, one can show that if there is a wall $w$ between $s$ and $s'$, we have:
\begin{align}
\text{Trans}^{MO}(s')^{-1}\text{Trans}^{MO}(s)=\mbf{B}_{w,MO}^*
\end{align}

Here we denote $\mbf{B}_{w,MO}^*$ as $\mbf{B}_{w}^*$ to distinguish the geometric and the algebraic monodromy operators.

In this way we obtain the similar result as in the Theorem \ref{monodromy-rep-casimir}

\begin{thm}\label{geometry-monodromy-rep-casimir}
The monodromy representation:
\begin{align}
\pi_{1}(\mbb{P}^n\backslash\textbf{Sing},0^{+})\rightarrow\text{Aut}(H_{T}(M(\mbf{v},\mbf{w})))
\end{align}
of the Dubrovin connection is generated by $\mbf{B}_{w,MO}^*$ with $q_1=e^{2\pi i\hbar_1},q_2=e^{2\pi i\hbar_2}$, i.e. the monodromy operators $\mbf{B}_{w,MO}$ in the fixed point basis. The based point $0^+$ is a point sufficiently close to $0\in\mbb{P}^n$.
\end{thm}

\section{\textbf{Proof of the Okounkov's Conjecture}}
In this section we prove the following statement:

\begin{thm}\label{main-theorem-1}
The positive half of the quantum toroidal algebra $U_{q,t}^{+}(\hat{\hat{\mf{sl}}}_n)$ is isomorphic to the positive half of the Maulik-Okounkov quantum affine algebra $U_{q}^{MO,+}(\hat{\mf{g}}_{Q})$ for the affine type $A$.
\end{thm}

The statement can be proved once we prove the following statement:
\begin{thm}\label{main-theorem-2}
Restricted to the morphism space $\text{End}(\bigoplus_{\mbf{v}_1+\mbf{v}_2=\mbf{v}}K(\mbf{v}_1,\mbf{w}_1)\otimes K(\mbf{v}_2,\mbf{w}_2))$:
\begin{align}
R_{w}^{MO}=R_{w}
\end{align}
\end{thm}

\begin{lem}
The Theorem \ref{main-theorem-2} implies the Theorem \ref{main-theorem-1}.
\end{lem}
\begin{proof}
By the Negu\c{t}'s result \cite{N23}, we have an algebra embedding:
\begin{align}
U_{q,t}(\hat{\hat{\mf{sl}}}_n)\hookrightarrow U_{q}^{MO}(\hat{\mf{g}}_{Q})\hookrightarrow\prod_{\mbf{w}}\text{End}(K(\mbf{w}))
\end{align}
Thus the equality $R^{MO}_{w}=R_{w}$ on $\text{End}(\bigoplus_{\mbf{v}_1+\mbf{v}_2=\mbf{v}}K(\mbf{v}_1,\mbf{w}_1)\otimes K(\mbf{v}_2,\mbf{w}_2))$ implies the equality in $U_{q}^{MO}(\hat{\mf{g}}_{Q})$. The equality of the universal $R$-matrix implies the isomorphism of the Hopf algebra $U_{q}(\mf{g}_{w})\cong U_{q}^{MO}(\mf{g}_{w})$. This implies \ref{main-theorem-1} by the factorisation property of the universal $R$-matrix in the Theorem \ref{factorisation} and the formula \ref{factorisation-geometry}, and thus implies the Theorem \ref{main-theorem-1}.

\end{proof}

\subsection{Proof of the main theorem}
Since we know that The monodromy representation of the Dubrovin connection for the affine type $A$ quiver varieties is generated by $\mbf{m}((1\otimes S_{w})R_{w}^{MO})$ and $\mbf{m}((1\otimes S_{w})R_{w})$. We can check that actually these two generators are equal:
\begin{thm}\label{same-monodromy}
As the element in $\text{End}(K(\mbf{v},\mbf{w}))$, we have that:
\begin{align}
\mbf{m}((1\otimes S_{w})R_{w}^{MO})=\mbf{m}((1\otimes S_{w})R_{w})
\end{align}
\end{thm}
\begin{proof}
First we prove that given two quantum difference equations $\Psi(p^{\mc{L}}z)=\mbf{M}_{\mc{L}}(z)\Psi(z)$, $\Psi^{MO}(p^{\mc{L}}z)=\mbf{M}^{MO}_{\mc{L}}(z)\Psi^{MO}(z)$. The degeneration limit of their solutions $\Psi_{0,\infty}(z)$ and $\Psi_{0,\infty}^{MO}(z)$ are the same. This implies that the transport of solutions $\text{Trans}^{MO}(s)$ and $\text{Trans}(s)$ from both geometric and algebraic quantum difference equations are the same.

Since the solution of both difference equations are of the form:
\begin{align}
\Psi_{0,\infty}(z)=P_{0,\infty}\Psi_{0,\infty}^{reg}(z)z^{(\cdots)}
\end{align}

Here $P_{0,\infty}$ are the matrix of eigenvectors of $\mbf{M}_{\mc{L}}(0,\infty)$. $\Psi_{0,\infty}^{reg}(z)$ are the regular part of the solution around $z=0,\infty$ such that $\Psi_{0,\infty}^{reg}(z)=1+O(z^{\pm1})$.

Now since $\mbf{M}_{\mc{L}}^{(MO)}(0)=\mc{L}$, $P_{0}^{(MO)}$ can be chosen as the fixed point basis. In fact, they have the same degeneration limit:
\begin{prop}\label{degeneration-prop-0}
 The degeneration limit of $\Psi_{0}(z)$ and $\Psi_{0}^{(MO)}(z)$ are the same.
\end{prop}
\begin{proof}
The proof is based on the analysis of the expression for $\Psi_{\infty}(z)$ and $\Psi_{\infty}^{(MO)}(z)$. It is known that they are expressed in terms of the infinite product of the form:
\begin{align}\label{infinite-product}
&\prod_{k\geq0}^{\infty}E_{\infty}^{-k-1}H^{-1}\mbf{B}_{\mbf{m}_1}(p^kz)^{-1}\cdots\mbf{B}_{\mbf{m}_l}(p^kz)HE_{\infty}^k,\qquad\text{Algebraic QDE}\\
&\prod_{k\geq0}^{\infty}E_{\infty,MO}^{-k-1}H_{MO}^{-1}\mbf{B}_{w_1,MO}(p^kz)\cdots\mbf{B}_{w_r,MO}(p^kz)HE_{\infty,MO}^k,\qquad\text{Okounkov-Smirnov QDE}
\end{align}

It is known that both for generic $\mbf{B}_{\mbf{m}}(z)$ and $\mbf{B}_{w,MO}(z)$ has the expression of the form:
\begin{align}
\sum_{\substack{\alpha_1+\cdots+\alpha_r=k\alpha\\r\geq0}}\frac{(-1)^{r}}{(1-q^{-m}z^{-\alpha_1}p^{-k|\alpha_1|})\cdots(1-q^{-m}z^{-\alpha_r}p^{-k|\alpha_r|})}\mbf{m}((1\otimes S_{w})(R_{\alpha_1,MO}^{-}\cdots R_{\alpha_r,MO}^{-})^{-1})
\end{align}

Thus when written into the infinite product of the form \ref{infinite-product}, the contribution of $R_{\alpha}$ after doing the degeneration is determined by the leading term of the form:
\begin{align}
-\sum_{k\geq0}\frac{1}{1-q^{-m}z^{-\alpha}p^{-k|\alpha|}}E_{\infty,(MO)}^{-k-1}H^{-1}\mbf{m}((1\otimes S_{w})(R_{\alpha,MO}^{-})^{-1})HE_{\infty,MO}^k
\end{align}

Doing the taylor expansion of $\frac{1}{1-q^{-m}z^{-\alpha}p^{-k|\alpha|}}$, we have that this leading term is also determined by the following terms:
\begin{align}\label{simplest-leading}
E_{\infty,(MO)}^{-1}\sum_{k\geq0}q^{-m}z^{-\alpha}p^{-k|\alpha|}\text{Ad}_{E_{\infty,(MO)}^{k}H}\mbf{m}((1\otimes S_{w})R_{\alpha})=\frac{q^{-m}z^{-\alpha}}{1-p^{-|\alpha|}\text{Ad}_{E_{\infty,(MO)}H}}\mbf{m}((1\otimes S_{w})R_{\alpha})
\end{align}
Now if we do the degeneration limit $p,q\rightarrow1$ with $\ln(p)/\ln(q)$ fixed. This implies that if $R_{\alpha}$ has the degeneration limit as the root pieces of the classical wall $r$-matrix $r_{\alpha}$. Then the term \ref{simplest-leading} survives.  If not, it will be zero. 

Also since the degeneration limit of both $E_{\infty},E_{\infty,MO}$ and $H,H_{MO}$ are the same, it is reduced to analyze the degeneration limit of $R_{\mbf{m}}=\prod_{\alpha}R_{\alpha}$ for slope subalgebras and $R_{w}^{MO}=\prod_{\alpha}R_{\alpha}^{MO}$ for the Maulik-Okounkov wall subalgebras. But as stated before, the degeneration limit of $R_{w}^{MO}$ which is not zero coincides with the degeneration limit of $R_{\mbf{m}}$ for some slope point $\mbf{m}$ on the wall $w$. Thus this means that the fundamental regular solutions $\Psi_{0}^{reg}(z)$ and $\Psi_{0}^{reg,MO}(z)$ have the same degeneration limit $\psi_0(z)$.
\end{proof}

It remains to check that the degeneration limit of $\Psi_{\infty}(z)$ and $\Psi_{\infty}^{(MO)}(z)$ are the same. Since:
\begin{align}
\mbf{M}_{\mc{L}}(\infty)=\mc{L}\prod^{\leftarrow}_{\mbf{m}}\mbf{B}_{\mbf{m}},\qquad\mbf{M}_{\mc{L}}^{MO}(\infty)=\mc{L}\prod^{\leftarrow}_{w}\mbf{B}_{w}^{MO}
\end{align}

the degeneration limit of $\mbf{M}_{\mc{L}}(\infty)$ and $\mbf{M}_{\mc{L}}^{MO}(\infty)$ are of the form:
\begin{align}
\mbf{M}_{\mc{L}}^{coh}(\infty)=c_1(\mc{L})+\sum_{\mbf{m}}\mbf{m}(r_{\mbf{m}}),\qquad\mbf{M}_{\mc{L}}^{(MO),coh}(\infty)=c_1(\mc{L})+\sum_{w}\mbf{m}(r_{w})
\end{align}

Since for generic choice of $\mbf{m}$, $r_{\mbf{m}}=r_{w}$, we have that $\mbf{M}_{\mc{L}}^{coh}(\infty)=\mbf{M}_{\mc{L}}^{(MO),coh}(\infty)$. It means that we can choose $P_{\infty}$ and $P_{\infty}^{MO}$ such that their corresponding degeneration limit are the same $P_{\infty}^{coh}$.

In this case we can see that we can construct the solutions $\Psi_{\infty}(z)$ and $\Psi_{\infty}^{MO}(z)$ around $z=\infty$ such that they have the same degeneration limit $\psi_{\infty}(z)$ when $z=\infty$. Using the proof in the Proposition \ref{degeneration-prop-0}, we can also prove the following statement:

\begin{prop}\label{degeneration-prop-infty}
 The degeneration limit of $\Psi_{\infty}(z)$ and $\Psi_{\infty}^{(MO)}(z)$ are the same.
\end{prop}

Now from the definition of $\text{Trans}(s)=\psi_{0}^{-1}(e^{2\pi is})\psi_{\infty}(e^{2\pi is})$ and the fact that the degeneration limit of the solution from the algebraic QDE and Okounkov-Smirnov QDE are of the same form.
\end{proof}

Using the fact that $U_{q,t}(\hat{\hat{\mf{sl}}}_n)\hookrightarrow U_{q}(\mf{g}_{Q})\hookrightarrow\prod_{\mbf{w}}\text{End}(K(\mbf{w}))$, a direct consequence is that as the element of $U_{q}^{MO}(\mf{g}_{w})$, we also have the equality:
\begin{align}\label{identity-for-monodromy}
\mbf{m}((1\otimes S_{w})R_{w}^{MO})=\mbf{m}((1\otimes S_{w})R_{w})
\end{align}
This means that:
\begin{align}
\mbf{m}((1\otimes S_{w})(R_{w}^{MO}-R_{w}))=0
\end{align}
which implies that $R_{w}^{MO}=R_{w}$ from the following general lemma. 

\begin{lem}\label{important-lemma}
Let $A^{\leq}$, $A^{\geq}$ be two finite-dimensional $\mbb{F}$-Hopf algebra with a non-degenerate bilinear pairing $\langle-,-\rangle:A^{\leq}\otimes A^{\geq}\rightarrow\mbb{F}$. Let $D(A):=A^{\leq}\otimes A^{\geq}$ be the Drinfeld double. Let $B^{\leq}\subset A^{\leq}$ and $B^{\geq}\subset A^{\geq}$ be the Hopf subalgebra such that the restriction of the bilinear pairing $\langle-,-\rangle$ to $B^{\leq}\otimes B^{\geq}$ is still non-degenerate, and using the bilinear pairing we can also define the Drinfeld double $D(B):=B^{\leq}\otimes B^{\geq}$. Denote $R_{A}$ and $R_{B}$ the universal $R$-matrix for $D(A)$ and $D(B)$. Then denote $\mbf{m}$ the product map for $D(A)$ and $D(B)$. If $\mbf{m}(R_{A})=\mbf{m}(R_{B})$, $R_A=R_B$.
\end{lem}
\begin{proof}
We choose $\{F_{i}\}_{i\in I_{A}}$ and $\{G_{i}\}_{i\in I_{A}}$ the specific dual bases in $A^{\leq}$ and $A^{\geq}$ with respect to the non-degenerate bilinear pairing $\langle-,-\rangle$ such that for $D(B)\subset D(A)$ as Hopf algebra with respect to the same bilinear pairing, the dual bases $\{F_{i}\}_{i\in I_{B}}$ and $\{G_{i}\}_{i\in I_{B}}$ in $B^{\leq}$ and $B^{\geq}$ are the same such that $I_{B}\subset I_{A}$. The existence of such basis will be proved in the appendix. Now it can be seen that the basis for $D(A)$ can be given as $\{F_{i}\otimes G_{j}\}_{i,j\in I_A}$, and similarly $\{F_{i}\otimes G_{j}\}_{i,j\in I_B}$ for $D(B)$.

By definition of the universal $R$-matrix we have the formula:
\begin{align}
R_{A}=\sum_{i\in I_{A}}(F_i\otimes1)\otimes(1\otimes G_{i})\in D(A)\otimes D(A)
\end{align}
\begin{align}
R_{B}=\sum_{i\in I_{B}}(F_i\otimes1)\otimes(1\otimes G_{i})\in D(B)\otimes D(B)
\end{align}

So we have that:
\begin{align}
\mbf{m}(R_{A})-\mbf{m}(R_{B})=\sum_{i\in I_A\backslash I_B}F_i\otimes G_i
\end{align}
Now $\{F_{i}\otimes G_{i}\}_{i\in I_A\backslash I_{B}}$ are also bases which are linearly independent. So if $\sum_{i\in I_A\backslash I_{B}}F_i\otimes G_i=0$, then $F_i\otimes G_i=0$ for $i\in I_A\backslash I_{B}$.
\end{proof}

Now for our case, since $S_{w}$ is an isomorphism of vector spaces. \ref{identity-for-monodromy} implies that $\mbf{m}(R_{w}^{MO})=\mbf{m}(R_{w})$, which means that $R_{w}^{MO}=R_{w}$.
Hence finish the proof of the Theorem \ref{main-theorem-2}, and now we have our main result the Theorem \ref{main-theorem-1}.

\section{\textbf{Appendix: Basics about the Drinfeld double}}
In this appendix we prove some facts about the Drinfeld double of the Hopf algebra which is useful in the computation of the article.

Given two bialgebras $A^-$ and $A^+$ over the field $\mbb{F}$, a bilinear pairing $\langle-,-\rangle:A^{+}\otimes A^-\rightarrow\mbb{F}$ is called a \textbf{bialgebra pairing} if it satisfies:
\begin{align}
\langle aa',b\rangle=\langle a\otimes a',\Delta(b)\rangle,\qquad\langle a,bb'\rangle=\langle\Delta^{op}(a),b\otimes b'\rangle
\end{align}
for all $a,a'\in A^+$ and $b,b'\in A^-$. All our pairings will be Hopf, which means that $\langle Sa,b\rangle=\langle a,S^{-1}b\rangle$. Given such a bialgebra pairing, we can construct the \textbf{Drinfeld double} of the two bialgebras:
\begin{align}
D(A):=A^{+}\otimes A^{-}
\end{align}
as a vector space. We further require that the pairing is nondegenerate, which implies that $A^{+}\cong A^{-}$ as vector spaces. 

There are natural Hopf algebra embeddings $A^{\pm}\hookrightarrow D(A)$ as $a\hookrightarrow a\otimes 1$ for $a\in A^{+}$ and $b\hookrightarrow 1\otimes b$ for $b\in A^{-}$. $D(A)$ satisfies the following extra relation:
\begin{align}
\langle a_1,b_1\rangle(a_2\otimes1)\cdot(1\otimes b_2)=(a_1\otimes1)\cdot(1\otimes b_1)\langle a_2,b_2\rangle,\qquad\forall a\in A^{+},b\in A^{-}
\end{align}
This is equivalent to:
\begin{align}
(a\otimes1)\cdot(1\otimes b)=\langle Sa_1,b_1\rangle(1\otimes b_2)\cdot(a_2\otimes 1)\langle a_3,b_3\rangle,\qquad\forall a\in A^{+},b\in A^{-} 
\end{align}

Furthermore we can write the product structure as:
\begin{align}
(a\otimes b)\cdot(a'\otimes b')=\langle Sa_1,b'_1\rangle\langle a_3,b_3'\rangle a_2a'\otimes bb'_2
\end{align}

Now we suppose that $A^{\pm}$ are finite-dimensional, and we choose the dual bases $\{F_{i}\}_{i\in I}$ $\{G_{i}\}_{i\in I}$ with respect to the bialgebra pairing. Define the \textbf{universal R-matrix as}:
\begin{align}
R=\sum_{i\in I}(1\otimes G_i)\otimes (F_i\otimes 1)\in A^{-}\otimes A^{+}\subset D(A)\otimes D(A)
\end{align}

The following lemma shows that $R$ is independent of the choice of the basis:
\begin{lem}
The definition of the universal $R$-matrix is independent of the choice of the basis.
\end{lem}

\begin{proof}

Now we choose another set of dual bases $\{F_{i}'\}$ and $\{G_{i}'\}$ of $A^{\pm}$. We write down the transformation of the bases:
\begin{align}
F_{i}'=\sum_{j\in I}a_{ij}F_{j},\qquad G_{i}'=\sum_{j\in I}b_{ij}G_{j}
\end{align}

In terms of matrix:
\begin{align}
\mbf{F}'=\mbf{A}\mbf{F},\qquad\mbf{G}'=\mbf{B}\mbf{G}
\end{align}

The condition that $\langle F_{i}',G_{l}'\rangle=\delta_{il}$, we have the condition that:
\begin{equation}
\begin{aligned}
\langle F_{i}',G_{l}'\rangle=\sum_{j,k\in I}a_{ij}b_{lk}\langle F_{j},G_{k}\rangle=\sum_{j\in I}a_{ij}b_{lj}=\sum_{j\in I}(\mbf{A})_{ij}(\mbf{B}^t)_{jl}=\delta_{il}
\end{aligned}
\end{equation}
This is equivalent to say that $\mbf{A}\mbf{B}^t=\text{Id}$.

Now for the universal $R$-matrix, note that:
\begin{equation}
\begin{aligned}
&\sum_{i\in I}(1\otimes G_{i}')\otimes(F_{i}'\otimes1)=\sum_{i,j,k\in I}b_{ij}a_{ik}(1\otimes G_j)\otimes(F_k\otimes1)\\
=&\sum_{j,k\in I}(1\otimes G_{j})\otimes (F_k\otimes1)(\sum_{i\in I}b_{ij}a_{ik})=\sum_{j,k\in I}(1\otimes G_{j})\otimes (F_k\otimes1)(\sum_{i\in I}(\mbf{A}^t)_{ki}(\mbf{B})_{ij})\\
=&\sum_{j,k\in I}(1\otimes G_{j})\otimes (F_k\otimes1)(\mbf{A}^t\mbf{B})_{kj}=\sum_{j,k\in I}(1\otimes G_{j})\otimes (F_k\otimes1)\delta_{kj}\\
=&\sum_{j\in I}(1\otimes G_{j})\otimes (F_j\otimes1)
\end{aligned}
\end{equation}
\end{proof}

Now given the Hopf subalgebras $B^{\pm}\subset A^{\pm}$ such that the bialgebra pairing on $B^{\pm}$ is the restriction of the bialgebra pairing on $A^{\pm}$. Without loss of generality, we assume that $\text{dim}_{\mbb{F}}(A^{\pm})=n+m$, $\text{dim}_{\mbb{F}}(B^{\pm})=n$.

\begin{prop}
There exists dual bases $\{F_{i}\}_{i=1,\cdots,n+m}$, $\{G_{i}\}_{i=1,\cdots,n+m}$ of $A^{\pm}$ such that $\{F_{i}\}_{i=1,\cdots,n}$, $\{G_{i}\}_{i=1,\cdots,n}$ are the dual bases for $B^{\pm}$.
\end{prop}
\begin{proof}
Now suppose that we have the dual bases $\{F_{i}'\}_{i=1,\cdots,n+m}$, $\{G_{i}'\}_{i=1,\cdots,n+m}$ of $A^{\pm}$, and $\{F_{i}\}_{i=1,\cdots,n}$, $\{G_{i}\}_{i=1,\cdots,n}$ are the dual bases for $B^{\pm}$. By the requirement of the proposition, we want to construct the dual bases $\{F_{i}\}_{i=1,\cdots,n+m}$, $\{G_{i}\}_{i=1,\cdots,n+m}$ of $A^{\pm}$ such that $\{F_{i}\}_{i=1,\cdots,n}$, $\{G_{i}\}_{i=1,\cdots,n}$ are the dual bases for $B^{\pm}$.

Now we write $\{F_{i}\}_{i=1,\cdots,n}$, $\{G_{i}\}_{i=1,\cdots,n}$ of $B^{\pm}$ in terms of
$\{F_{i}'\}_{i=1,\cdots,n+m}$, $\{G_{i}'\}_{i=1,\cdots,n+m}$ of $A^{\pm}$:
\begin{align}
F_{i}=\sum_{j=1}^{n+m}a_{ij}F_{j}',\qquad G_{i}=\sum_{j=1}^{n+m}b_{ij}G_{j}'
\end{align}

Written in matrix:
\begin{align}
\begin{pmatrix}\mbf{F}_{[1,n]}\\\mbf{F}'_{[n+1,n+m]}\end{pmatrix}=\begin{pmatrix}A&T\\0&\text{Id}\end{pmatrix}\begin{pmatrix}\mbf{F}'_{[1,n]}\\\mbf{F}'_{[n+1,n+m]}\end{pmatrix}
\end{align}

\begin{align}
\begin{pmatrix}\mbf{G}_{[1,n]}\\\mbf{G}'_{[n+1,n+m]}\end{pmatrix}=\begin{pmatrix}B&S\\0&\text{Id}\end{pmatrix}\begin{pmatrix}\mbf{G}'_{[1,n]}\\\mbf{G}'_{[n+1,n+m]}\end{pmatrix}
\end{align}

Now we can suppose the transfer matrix:
\begin{align}
\begin{pmatrix}\mbf{F}_{[1,n]}\\\mbf{F}_{[n+1,n+m]}\end{pmatrix}=\begin{pmatrix}\text{Id}&0\\U_1&V_1\end{pmatrix}\begin{pmatrix}\mbf{F}_{[1,n]}\\\mbf{F}'_{[n+1,n+m]}\end{pmatrix}
\end{align}

\begin{align}
\begin{pmatrix}\mbf{G}_{[1,n]}\\\mbf{G}_{[n+1,n+m]}\end{pmatrix}=\begin{pmatrix}\text{Id}&0\\U_2&V_2\end{pmatrix}\begin{pmatrix}\mbf{G}_{[1,n]}\\\mbf{G}'_{[n+1,n+m]}\end{pmatrix}
\end{align}

So combining these we obtain that:
\begin{equation}
\begin{aligned}
\begin{pmatrix}\mbf{F}_{[1,n]}\\\mbf{F}_{[n+1,n+m]}\end{pmatrix}=&\begin{pmatrix}\text{Id}&0\\U_1&V_1\end{pmatrix}\begin{pmatrix}\mbf{F}_{[1,n]}\\\mbf{F}'_{[n+1,n+m]}\end{pmatrix}=\begin{pmatrix}\text{Id}&0\\U_1&V_1\end{pmatrix}\begin{pmatrix}A&T\\0&\text{Id}\end{pmatrix}\begin{pmatrix}\mbf{F}'_{[1,n]}\\\mbf{F}'_{[n+1,n+m]}\end{pmatrix}\\
=&\begin{pmatrix}A&T\\U_1A&U_1T+V_1\end{pmatrix}\begin{pmatrix}\mbf{F}'_{[1,n]}\\\mbf{F}'_{[n+1,n+m]}\end{pmatrix}
\end{aligned}
\end{equation}

\begin{equation}
\begin{aligned}
\begin{pmatrix}\mbf{G}_{[1,n]}\\\mbf{G}_{[n+1,n+m]}\end{pmatrix}=&\begin{pmatrix}\text{Id}&0\\U_2&V_2\end{pmatrix}\begin{pmatrix}\mbf{G}_{[1,n]}\\\mbf{G}'_{[n+1,n+m]}\end{pmatrix}=\begin{pmatrix}\text{Id}&0\\U_2&V_2\end{pmatrix}\begin{pmatrix}B&S\\0&\text{Id}\end{pmatrix}\begin{pmatrix}\mbf{G}'_{[1,n]}\\\mbf{G}'_{[n+1,n+m]}\end{pmatrix}\\
=&\begin{pmatrix}B&S\\U_2B&U_2S+V_2\end{pmatrix}\begin{pmatrix}\mbf{G}'_{[1,n]}\\\mbf{G}'_{[n+1,n+m]}\end{pmatrix}
\end{aligned}
\end{equation}

By the orthogonality condition we need that:
\begin{equation}
\begin{aligned}
&\begin{pmatrix}A&T\\U_1A&U_1T+V_1\end{pmatrix}\begin{pmatrix}B^t&B^tU_2^t\\S^t&S^tU_2^t+V_2^t\end{pmatrix}=\begin{pmatrix}\text{Id}&0\\0&\text{Id}\end{pmatrix}\\
=&\begin{pmatrix}AB^t+TS^t&AB^tU_2^t+TS^tU_2^t+TV_2^t\\U_1AB^t+U_1TS^t+V_1S^t&U_1AB^tU_2^t+U_1TS^tU_2^t+U_1TV_2^t+V_1S^tU_2^t+V_1V_2^t\end{pmatrix}
\end{aligned}
\end{equation}

The matrix relation gives us the equation:
\begin{align}
AB^t+TS^t=\text{Id}
\end{align}

which reduce the matrix to:
\begin{align}
\begin{pmatrix}\text{Id}&U_2^t+TV_2^t\\U_1+V_1S^t&V_1S^tU_2^t+V_1V_2^t\end{pmatrix}=\begin{pmatrix}\text{Id}&0\\0&\text{Id}\end{pmatrix}
\end{align}

This gives us the equation $V_1(S^tU_2^t+V_2^t)=V_1(-S^tT+\text{Id})V_2^t=\text{Id}$. Thus for the generic choice of $S$ and $T$, there exists solution for $V_1$ and $V_2$. Thus finish the proof.
\end{proof}


\begin{thebibliography}{100}
\bibitem[AD24]{AD24}
Ayers J, Dinkins H. Wreath Macdonald polynomials, quiver varieties, and quasimap counts[J]. arXiv preprint arXiv:2410.07399, 2024.

\bibitem[AO21]{AO21}
Aganagic M, Okounkov A. Elliptic stable envelopes[J]. Journal of the American Mathematical Society, 2021, 34(1): 79-133.

\bibitem[BD23]{BD23}
Botta T M, Davison B. Okounkov's conjecture via BPS Lie algebras[J]. arXiv preprint arXiv:2312.14008, 2023.

\bibitem[BM15]{BM15}
Balagović M. Degeneration of trigonometric dynamical difference equations for quantum loop algebras to trigonometric Casimir equations for Yangians[J]. Communications in Mathematical Physics, 2015, 334(2): 629-659.

\bibitem[D21]{D21}
Dinkins H. Elliptic stable envelopes of affine type $ A $ quiver varieties[J]. arXiv preprint arXiv:2107.09569, 2021.

\bibitem[GL11]{GL11}
Gautam S, Laredo V T. Monodromy of the trigonometric Casimir connection for sl2[R]. Noncommutative birational geometry, representations and combinatorics, 2011.

\bibitem[KR90]{KR90}
Kirillov A N, Reshetikhin N. q-Weyl group and a multiplicative formula for universal R-matrices[J]. 1990.

\bibitem[KPSZ21]{KPSZ21}
Koroteev, P., Pushkar, P. P., Smirnov, A. V., \& Zeitlin, A. M. (2021). Quantum K-theory of quiver varieties and many-body systems. Selecta Mathematica, 27(5), 87.

\bibitem[KT91]{KT91}
Khoroshkin S M, Tolstoy V N. Universal R-matrix for quantized (super) algebras[J]. Communications in Mathematical Physics, 1991, 141: 599-617.

\bibitem[KT92]{KT92}
Tolstoy V N, Khoroshkin S M. The universal R-matrix for quantum untwisted affine Lie algebras[J]. Functional analysis and its applications, 1992, 26(1): 69-71.


\bibitem[MO12]{MO12}
Maulik D, Okounkov A. Quantum groups and quantum cohomology[J]. arXiv preprint arXiv:1211.1287, 2012.

\bibitem[N01]{N01}
Nakajima H. Quiver varieties and finite dimensional representations of quantum affine algebras[J]. Journal of the American Mathematical Society, 2001, 14(1): 145-238.

\bibitem[N15]{N15}
Negu\c{t} A. Quantum algebras and cyclic quiver varieties[M]. Columbia University, 2015.

\bibitem[N19]{N19}
Negu\c{t} A. The PBW basis of $ U_ {q,\bar {q}}(\ddot {\mathfrak {gl}} _n) $[J]. arXiv preprint arXiv:1905.06277, 2019.

\bibitem[N20]{N20}
Negu\c{t} A. Quantum toroidal and shuffle algebras[J]. Advances in Mathematics, 2020, 372: 107288.

\bibitem[N22]{N22}
Negu\c{t} A. SHUFFLE ALGEBRAS FOR QUIVERS AND R-MATRICES[J]. Journal of the Institute of Mathematics of Jussieu, 2022: 1-36.

\bibitem[N23]{N23}
Negu\c{t} A. Quantum loop groups and $ K $-theoretic stable envelopes[J]. arXiv preprint arXiv:2303.12041, 2023.

\bibitem[O15]{O15}
Okounkov A. Lectures on K-theoretic computations in enumerative geometry[J]. arXiv preprint arXiv:1512.07363, 2015.

\bibitem[O20]{O20}
Okounkov A. Nonabelian stable envelopes, vertex functions with descendents, and integral solutions of $ q $-difference equations[J]. arXiv preprint arXiv:2010.13217, 2020.

\bibitem[O21]{O21}
Okounkov A. Inductive construction of stable envelopes[J]. Letters in Mathematical Physics, 2021, 111: 1-56.

\bibitem[OS22]{OS22}
Okounkov A, Smirnov A. Quantum difference equation for Nakajima varieties[J]. Inventiones mathematicae, 2022, 229(3): 1203-1299.

\bibitem[PSZ20]{PSZ20}
Pushkar P P, Smirnov A V, Zeitlin A M. Baxter Q-operator from quantum K-theory[J]. Advances in Mathematics, 2020, 360: 106919.

\bibitem[S16]{S16}
Smirnov A. On the Instanton R-matrix[J]. Communications in Mathematical Physics, 2016, 345: 703-740.

\bibitem[S20]{S20}
Smirnov A. Elliptic stable envelope for Hilbert scheme of points in the plane[J]. Selecta Mathematica, 2020, 26(1): 1-57.

\bibitem[S21]{S21}
Smirnov A. Quantum differential and difference equations for $\mathrm {Hilb}^{n}(\mathbb {C}^ 2) $[J]. arXiv preprint arXiv:2102.10726, 2021.

\bibitem[S87]{S87}
Salvetti M. Topology of the complement of real hyperplanes in CN[J]. Invent. math, 1987, 88(3): 603-618.

\bibitem[SV17]{SV17}
Schiffmann O, Vasserot E. On cohomological Hall algebras of quivers: Yangians (2017)[J]. arXiv preprint arXiv:1705.07491.

\bibitem[SV23]{SV23}
Schiffmann O, Vasserot E. Cohomological Hall algebras of quivers and Yangians[J]. arXiv preprint arXiv:2312.15803, 2023.

\bibitem[Z23]{Z23}
Zhu T. Quantum Difference equation for the affine type $ A $ quiver varieties I: General Construction[J]. arXiv preprint arXiv:2308.00550, 2023.

\bibitem[Z24]{Z24}
Zhu T. From quantum difference equation to Dubrovin connection of affine type A quiver varieties[J]. arXiv preprint arXiv:2405.02473, 2024.

\bibitem[Z24-2]{Z24-2}
Zhu T. Instanton moduli space, stable envelopes and quantum difference equations[J]. arXiv preprint arXiv:2408.15560, 2024.

\bibitem[Zhu]{Zhu}
Zhu T. Quantum difference equations, shuffle algebras and Maulik-Okounkov quantum affine algebras. In preparation.

\bibitem[ZZ23]{ZZ23}
Zhou Z. Virtual Coulomb branch and vertex functions[J]. Duke Mathematical Journal, 2023, 172(17): 3359-3428.
\end{thebibliography}
\end{document}